\documentclass[12pt]{amsart}
\usepackage{amsthm}
\usepackage{newlfont}
\usepackage{amsmath}
\usepackage{amssymb}
\usepackage{amsfonts}
\usepackage[mathscr]{euscript}
\usepackage[ngerman,french,english]{babel}
\usepackage{mathrsfs}
\usepackage{fancybox}
\usepackage{nicefrac}
\usepackage{pifont}
\usepackage[arrow,curve,matrix,tips]{xy}
\usepackage{enumitem}
\usepackage{stmaryrd}
\usepackage{srcltx}
\usepackage{ulem}
\usepackage{accents}
\usepackage{hyperref}
\usepackage{wasysym}
\usepackage{leftidx}
\usepackage{eurosym}

\DeclareMathAlphabet{\mathpzc}{OT1}{pzc}{m}{it}

\newtheorem{theorem}{Theorem}[section]

\newtheorem{lemma}[theorem]{Lemma}
\newtheorem{corollary}[theorem]{Corollary}

\newtheorem{proposition}[theorem]{Proposition}

\theoremstyle{definition}
\newtheorem{definition}[theorem]{Definition}
\newtheorem{example}[theorem]{Example}

\theoremstyle{remark}
\newtheorem{remark}[theorem]{Remark}

\newtheorem{question}[theorem]{Question}

\newcommand{\op}{\operatorname}

\newcommand{\ra}{\rightarrow}

\newcommand{\la}{\leftarrow}

\newcommand{\hra}{\hookrightarrow}

\newcommand{\xra}[1]{\xrightarrow{#1}}
\newcommand{\xla}[1]{\xleftarrow{#1}}

\newcommand{\sira}{\xra{\sim}}
\newcommand{\xsira}[1]{\xrightarrow[\sim]{#1}}

\newcommand{\sila}{\overset{\sim}{\leftarrow}}

\newcommand{\mar}{\ar@{|->}}
\newcommand{\sar}{\ar@{->>}}
\newcommand{\iar}{\ar@{^{(}->}}
\newcommand{\gar}{\ar@{=}}
\newcommand{\gleichar}{\ar@{}|{=}}
\newcommand{\congar}{\ar@{}|{\cong}}

\newcommand{\Bl}[1]{{\mathbb{#1}}}
\newcommand{\DZ}{\Bl{Z}}
\newcommand{\DN}{\Bl{N}}

\newcommand{\Hom}{\op{Hom}}

\newcommand{\Tr}{{\op{Tr}}}

\newcommand{\Mod}{\op{Mod}}
\newcommand{\modfin}{\op{mod}}

\newcommand{\End}{\op{End}}

\newcommand{\id}{\op{id}}

\newcommand{\ol}[1]{{\overline{#1}}}
\newcommand{\ul}[1]{{\underline{#1}}}

\newcommand{\Kern}{\op{ker}}

\newcommand{\inv}{^{-1}}
\newcommand{\can}{\op{can}}

\newcommand{\supp}{\op{supp}}

\newcommand{\tildew}[1]{\widetilde{#1}}

\newcommand{\comp}{\circ}

\newcommand{\gr}{\op{gr}}

\newcommand{\mathovalbox}[1]{{\text{\ovalbox{${#1}$}}}}

\newcommand{\define}[1]{{\textbf{#1}}}

\numberwithin{equation}{section}

\newcommand{\add}{\op{add}}

\newcommand{\proj}{\op{proj}}

\newcommand{\zvek}[2]{{\big[\begin{smallmatrix} {#1} & {#2} \end{smallmatrix}\big]}}
\newcommand{\svek}[2]{{\big[\begin{smallmatrix} {#1} \\ {#2} \end{smallmatrix}\big]}}

\newcommand{\tzmat}[4]{{\big[\begin{smallmatrix} {#1} & {#2} \\ {#3} & {#4} \end{smallmatrix}\big]}}

\renewcommand{\phi}{\varphi}

\newcommand{\Cone}{\op{Cone}}

\theoremstyle{remark}

\renewcommand{\epsilon}{{\varepsilon}}

\newcommand{\heart}{{\heartsuit}}

\newcommand{\WCweak}{{WC}}
\newcommand{\candidateWCweak}{{WC_{\op{c}}}}
\newcommand{\WCfun}{{\widetilde{WC}}}
\newcommand{\weak}{{\op{weak}}}

\newcommand{\ic}{{\op{ic}}}

\newcommand{\anticomm}{\circleddash}

\newcommand{\range}{\op{range}}

\newcommand{\anti}{{\op{anti}}}

\begin{document}
\selectlanguage{english}

\title[Homotopy categories and weight complex functors]{Homotopy
  categories and idempotent completeness, weight structures and weight
  complex functors}

\author{Olaf M.~Schn{\"u}rer}
\address{Mathematisches Institut\\ 
  Universit{\"a}t Bonn\\
  Endenicher Allee 60\\
  53115 Bonn\\
  Germany
}
\email{olaf.schnuerer@math.uni-bonn.de}


\begin{abstract}
This article provides some basic results on weight structures, weight
complex functors and homotopy categories.
We prove that the full subcategories 
$K(\mathcal{A})^{w \leq n}$, $K(\mathcal{A})^{w \geq n}$,
$K(\mathcal{A})^-$ and $K(\mathcal{A})^+$ (of objects isomorphic to
suitably bounded complexes) of the 
homotopy category $K(\mathcal{A})$
of an additive category $\mathcal{A}$ are idempotent complete, which
confirms that 
$(K(\mathcal{A})^{w \leq 0}, K(\mathcal{A})^{w \geq 0})$ is a weight
structure on $K(\mathcal{A})$. 

We discuss weight complex functors and provide full details of an
argument sketched by M.~Bondarko, which shows that if $w$ is a bounded
weight structure on a triangulated category $\mathcal{T}$ that has a filtered
triangulated enhancement $\tildew{\mathcal{T}}$ then there exists a
strong weight complex functor $\mathcal{T} \ra
K(\heart(w))^\anti$. Surprisingly, in
order to carry out the proof, we 
need to impose an additional axiom on the filtered triangulated
category $\tildew{\mathcal{T}}$ which seems to be new. 
\end{abstract}

\maketitle

\tableofcontents

\section{Introduction}
\label{sec:introduction-weight-structures}

The aim of this article is to provide and review some foundational
results on homotopy categories, weight structures and weight complex
functors.

Weight structures were independently introduced by M.~Bondarko
in \cite{bondarko-weight-str-vs-t-str}, and D.~Pauksztello in
\cite{pauk-co-t} where they are called co-t-struc\-tures. Their definition is
formally very similar to that of a t-structure. In both cases there is
an axiom demanding that each objects fits into a triangle of a certain
form. In the case of a t-structure these ``truncation
triangles" are functorial whereas in the case of a weight structure
these ``weight decomposition triangles" are not functorial. 
This is a technical issue but the theory remains amazingly rich.
Weight structures have applications in various branches of
mathematics, for example in algebraic geometry (theory of motives),
algebraic topology (stable homotopy category) and representation
theory, see e.\,g.\ the work of M.~Bondarko (e.\,g.\
\cite{bondarko-weight-str-vs-t-str}), 
J.~Wildeshaus (e.\,g.\ \cite{wildeshaus-chow-arXiv}) and
D.~Pauksztello (\cite{pauk-co-t, pauk-note-co-t-arXiv}) and references 
therein; other references are 
\cite{achar-treumann-arXiv},
\cite{iyama-yoshino-mutation},
\cite{geordie-ben-HOMFLYPT},
\cite{aihara-iyama-silting-arXiv},
\cite{keller-nicolas-simple-arXiv},
\cite{MSSS-AB-context-co-t-str-arXiv},
\cite{achar-kitchen-koszul-mixed-arXiv}.

A crucial observation due to M.~Bondarko is that, in the presence of a
weight structure 
$w=(\mathcal{T}^{w \leq 0},\mathcal{T}^{w \geq 0})$
on a triangulated category $\mathcal{T}$, there is a weak weight
complex functor
$\WCweak: \mathcal{T} \ra K_\weak(\heart)$ where 
$\heart= \mathcal{T}^{w \leq 0} \cap \mathcal{T}^{w \geq 0}$ 
is the
heart of $w$ and the weak homotopy
category $K_\weak(\heart)$ is a certain quotient of the homotopy category
$K(\heart)$. 
M.~Bondarko explains
that in various natural settings this functor lifts to a 
``strong weight complex functor" $\mathcal{T} \ra K(\heart)^\anti$
(the upper index does not appear in
\cite{bondarko-weight-str-vs-t-str}; it will be explained below). 
We expect that this strong weight complex functor
will be an important tool.

The basic example of a weight structure is the ``standard weight
structure" 
$(K(\mathcal{A})^{w \leq 0}, K(\mathcal{A})^{w \geq 0})$
on the homotopy category $K(\mathcal{A})$ of an additive
category $\mathcal{A}$;
here 
$K(\mathcal{A})^{w \leq n}$ (resp.\ 
$K(\mathcal{A})^{w \geq n}$) is the full subcategory of 
$K(\mathcal{A})$ consisting of complexes $X=(X^i, d^i:X^i \ra
X^{i+1})$ that are isomorphic to a complex concentrated in
degrees $\leq n$ (resp.\ $\geq n$) (for fixed $n \in \DZ$).
The subtle point to confirm this example 
(which appears in \cite{bondarko-weight-str-vs-t-str} and
\cite{pauk-co-t})
is to check that $K(\mathcal{A})^{w \leq 0}$ 
and $K(\mathcal{A})^{w \geq 0}$ are both 
closed under retracts in $K(\mathcal{A})$.

This basic example was our motivation for the first part of this article
where we discuss the idempotent completeness of (subcategories of)
homotopy categories. Let $\mathcal{A}$ be an additive category as
above.
Then a natural question is whether 
$K(\mathcal{A})$ is idempotent complete.
Since $K(\mathcal{A})$ is a triangulated category this question can be
rephrased as follows:
\begin{question}
  \label{q:KA-ic}
  Given any idempotent endomorphism $e:X \ra X$ in
  $K(\mathcal{A})$, is there an isomorphism $X\cong E \oplus F$ such
  that $e$ corresponds to 
  $\tzmat 1000: E \oplus F \ra E \oplus F$?
\end{question}
We do not know the answer to this question in general.
We can show that certain full subcategories of $K(\mathcal{A})$ are
idempotent complete:
\begin{theorem}
  [{see Thm.~\ref{t:one-side-bounded-hot-idempotent-complete}}] 
  \label{t:one-side-bounded-hot-idempotent-complete-einleitung}
  The full subcategories 
  $K(\mathcal{A})^{w \leq n}$ and 
  $K(\mathcal{A})^{w \geq n}$ of
  $K(\mathcal{A})$ are idempotent complete.
  In particular
  $K(\mathcal{A})^-:=\bigcup_{n \in \DZ}K(\mathcal{A})^{w \leq n}$,
  $K(\mathcal{A})^+:=\bigcup_{n \in \DZ}K(\mathcal{A})^{w \geq n}$ and
  $K(\mathcal{A})^{bw}:= K(\mathcal{A})^- \cap K(\mathcal{A})^+$ 
  are idempotent complete.
\end{theorem}
Possibly this result is known to the
experts but we could not find a proof in the literature. 
Our proof of this result is constructive and based on work of
R.~W.~Thomason \cite{thomason-class} 
and ideas of P.~Balmer and M.~Schlichting
\cite{balmer-schlichting}.
Theorem~\ref{t:one-side-bounded-hot-idempotent-complete-einleitung}
implies that
$(K(\mathcal{A})^{w \leq 0}, K(\mathcal{A})^{w \geq 0})$
is a weight structure on $K(\mathcal{A})$ (see
Prop.~\ref{p:ws-hot-additive}).

Another approach to Question~\ref{q:KA-ic} is to impose further
assumptions on $\mathcal{A}$. 
We can show (see Thm.~\ref{t:hot-idempotent-complete}):
$K(\mathcal{A})$ is idempotent complete if
$\mathcal{A}$ has countable coproducts (this follows directly
from results of 
M.~B{\"o}kstedt and A.~Neeman \cite{neeman-homotopy-limits}
or from a variation of our proof of 
Theorem~\ref{t:one-side-bounded-hot-idempotent-complete-einleitung})
or if $\mathcal{A}$ is abelian (this is an application of results from
\cite{beligiannis-reiten-torsion-theories} and
\cite{Karoubianness}).
If $\mathcal{A}$ itself is idempotent complete then 
projectivization (see \cite{Krause-krull-schmidt-categories}) shows
that the full 
subcategory $K(\mathcal{A})^b$ of $K(\mathcal{A})$ of complexes that
are isomorphic to a bounded complex is idempotent complete.

It may seem natural to assume that
$\mathcal{A}$ is idempotent complete and additive in 
Question~\ref{q:KA-ic}. However, if $\mathcal{A}^\ic$ is the idempotent
completion of $\mathcal{A}$, then $K(\mathcal{A})$ is idempotent
complete if and only if $K(\mathcal{A}^{\ic})$ is idempotent complete
(see Rem.~\ref{rem:KA-idempotent-complete-wlog-ic}).

Results of R.~W.~Thomason \cite{thomason-class} indicate that it might
be useful to consider the Grothendieck group $K_0(K(\mathcal{A}))$ of
the triangulated category $K(\mathcal{A})$ for additive (essentially)
small $\mathcal{A}$. We show that the Grothendieck groups 
$K_0(K(\mathcal{A}))$,
$K_0(K(\mathcal{A})^-)$ and
$K_0(K(\mathcal{A})^+)$ all vanish for such $\mathcal{A}$ 
(see Prop.~\ref{p:grothendieck-homotopy-category}).

The second part of this article concerns weight complex functors.
In the example of the standard weight structure 
$(K(\mathcal{A})^{w \leq 0}, K(\mathcal{A})^{w \geq 0})$ on
$K(\mathcal{A})$, the heart $\heart$ is the idempotent completion of
$\mathcal{A}$ 
(see Cor.~\ref{c:heart-standard-wstr})
and the weak weight complex functor $\WCweak:
\mathcal{T} \ra K_\weak(\heart)$ 
naturally 
and easily
lifts 
to a
triangulated functor $\WCfun:K(\mathcal{A}) \ra K(\heart)^\anti$
(see Prop.~\ref{p:strong-WC-for-standard-wstr}); here the triangulated
categories $K(\heart)^\anti$ and $K(\heart)$ coincide as additive
categories with translation but a candidate triangle 
$X \xra{u} Y \xra{v} Z \xra{w} [1]X$
is a triangle in $K(\heart)^\anti$ if and only if 
$X \xra{-u} Y \xra{-v} Z \xra{-w} [1]X$
is a triangle in $K(\heart)$.
This functor $\WCfun$ is an example of a ``strong weight complex
functor".

Let us return to the general setup of a weight structure $w$ on a
triangulated category $\mathcal{T}$ with heart $\heart$.
Assume that  
$\tildew{\mathcal{T}}$ is a filtered triangulated category over
$\mathcal{T}$ in the sense of \cite[App.]{Beilinson}.
The main result of the second part of this article is a complete proof
of the following Theorem.

\begin{theorem}
  [{see Thm.~\ref{t:strong-weight-cplx-functor} and cf.\
    \cite[8.4]{bondarko-weight-str-vs-t-str}}] 
  \label{t:strong-weight-cplx-functor-einleitung}
  Assume that $w$ is bounded and that $\tildew{\mathcal{T}}$
  satisfies axiom~\ref{enum:filt-tria-cat-3x3-diag} stated 
  in Section~\ref{sec:additional-axiom}.
  Then there is a strong weight complex functor
  \begin{equation*}
    \WCfun: \mathcal{T} \ra K^b(\heart)^\anti.
  \end{equation*}
  This means that $\WCfun$ is a triangulated functor whose composition
  with $K^b(\heart)^\anti \ra K_\weak(\heart)$ 
  is isomorphic to the weak weight complex functor
  (as a functor of additive categories with translation).
\end{theorem}

Our proof of this theorem relies on the ideas of
A.~Beilinson and M.~Bondarko sketched in
\cite[8.4]{bondarko-weight-str-vs-t-str}. We explain the idea of the proof
in Section~\ref{sec:idea-construction-strong-WC-functor}.
The additional axiom~\ref{enum:filt-tria-cat-3x3-diag} 
imposed on $\tildew{\mathcal{T}}$ in
Theorem~\ref{t:strong-weight-cplx-functor-einleitung} 
seems to be new. It states that any morphism gives rise to a certain
$3\times 3$-diagram; see Section~\ref{sec:additional-axiom} for the
precise formulation. It is used in the proof of
Theorem~\ref{t:strong-weight-cplx-functor-einleitung} at 
two important points; we do not know if this axiom can be removed.

We hope that this axiom is satisfied for reasonable filtered triangulated
categories; it is true in the basic example of a filtered
triangulated category: If $\mathcal{A}$ is an abelian category
its filtered derived category $DF(\mathcal{A})$
is a filtered triangulated category (see
Prop.~\ref{p:basic-ex-f-cat})
that satisfies axiom~\ref{enum:filt-tria-cat-3x3-diag} 
(see Lemma~\ref{l:additional-axiom-true-DFA}).

In the short third part of this article we prove a result which
naturally fits into the context of weight structures and filtered
triangulated categories:
Given a filtered triangulated category $\tildew{\mathcal{T}}$ over a
triangulated category $\mathcal{T}$
with a weight structure $w$, there is a unique weight structure on
$\tildew{\mathcal{T}}$ that is compatible with $w$ (see
Prop.~\ref{p:compatible-w-str}).

\subsection*{Plan of the article}
We fix our notation and gather together some results on additive
categories, idempotent completeness, triangulated categories and
homotopy categories in Section~\ref{sec:preliminaries}; we suggest
skimming through this section and coming back as required.
Sections~\ref{sec:hot-cat-idem-complete}
and~\ref{sec:weight-structures} constitute the first part of this
article - they contain the results on idempotent
completeness of homotopy categories and some basic results on weight
structures.
The next two sections lay the groundwork for the proof of
Theorem~\ref{t:strong-weight-cplx-functor-einleitung}.
In Section~\ref{sec:weak-wc-fun} we construct the weak weight complex
functor in detail. We recall the notion of a filtered triangulated category in 
Section~\ref{sec:filt-tria-cat} and prove some important results
stated in \cite[App.]{Beilinson} as no proofs are available in the literature. 
We prove Theorem~\ref{t:strong-weight-cplx-functor-einleitung} in
Section~\ref{sec:strong-WCfun}.
Section~\ref{sec:weight-structures-and-filtered-triang} contains the
results on compatible weight structures.

\subsection*{Acknowledgements}

This work was motivated by an ongoing joint project with Geordie Williamson.
We thank him for the collaboration, for comments and in particular for
the encouragement to turn our private notes into this article. 
We thank Mikhail Bondarko for useful correspondence, discussions and
encouragement.
We thank Catharina Stroppel for her interest and
helpful comments on a preliminary version of this article.
We thank Peter J\o{}rgensen,
Henning Krause, Maurizio Martino, David Pauksztello, Jorge Vit{\'o}ria and Dong Yang
for their interest, discussions and useful references.
Thanks are also due to Apostolos Beligiannis and Idun Reiten for
explanations,
and to
Paul Balmer and Amnon Neeman for useful correspondence.

Our work was partially supported by the priority program SPP 1388 of
the German Science foundation, and by the 
Collaborative Research Center SFB Transregio 45 of the German Science foundation.

\section{Preliminaries}
\label{sec:preliminaries}

For the definition of an additive category (with translation
(= shift)), a functor of 
additive categories (with translation), and of a triangulated category
see \cite{KS-cat-sh}.

\subsection{(Additive) categories}
\label{sec:additive-categories}

Let $\mathcal{A}$ be a category and $X$ an object of $\mathcal{A}$. 
An object $Y \in \mathcal{A}$ is a 
\textbf{retract of $X$} if there are morphisms $p:X \ra Y$ and $i:Y
\ra X$ such that $pi=\id_Y$. Then $ip:X \ra X$ is an idempotent
endomorphism.
A subcategory $\mathcal{B} \subset \mathcal{A}$ is 
\textbf{closed under retracts} (= Karoubi-closed) if it contains all
retracts in $\mathcal{A}$ of any object of $\mathcal{B}$.
In this case $\mathcal{B}$ is a
\textbf{strict} subcategory of $\mathcal{A}$, i.\,e.\ it is closed under
isomorphisms.

Let $\mathcal{B}$ be a subcategory of an additive
category $\mathcal{A}$. We say that $\mathcal{B}$ is 
\textbf{dense in $\mathcal{A}$} if each object of $\mathcal{A}$ is a
summand of an object of $\mathcal{B}$.
We
define full subcategories $\leftidx{^\perp}{\mathcal{B}},
\mathcal{B}^\perp \subset \mathcal{A}$ by
\begin{align*}
  \leftidx{^\perp}{\mathcal{B}} & =
  \{Z \in \mathcal{A} \mid \mathcal{A}(Z, \mathcal{B})=0\},\\
  \mathcal{B}^\perp & =
  \{Z \in \mathcal{A} \mid \mathcal{A}(\mathcal{B},Z)=0\}.
\end{align*}

\subsection{Idempotent completeness}
\label{sec:idemp-completeness}

Let $\mathcal{A}$ be a category and $X$ an object of $\mathcal{A}$. 
An idempotent
endomorphism $e \in \End(X)$ \textbf{splits} if there is a \textbf{splitting
of $e$}, i.\,e.\ there are an object $Y \in \mathcal{A}$ and morphisms
$p:X \ra Y$ 
and $i: Y \ra X$ such that $ip=e$ and $pi=\id_Y$. 
A splitting of $e$ is unique up to
unique isomorphism. 
If every idempotent endomorphism splits we say that $\mathcal{A}$ is
\textbf{idempotent complete} (= Karoubian). 
If $\mathcal{A}$ is additive, an idempotent $e:X \ra X$ has a splitting
$(Y,p,i)$ and $1-e$ has a 
splitting $(Z, q, j)$, then obviously $(i,j): Y\oplus Z \ra X$ is an
isomorphism with inverse $\svek p q$.
In particular in an idempotent complete additive category 
any idempotent endomorphism of an object $X$ induces a 
direct sum decomposition of $X$.

Any additive category $\mathcal{A}$ has an \textbf{idempotent
  completion} (= Ka\-rou\-bi completion)
$(\mathcal{A}^\ic,i)$, i.\,e.\ there is an idempotent complete additive
category $\mathcal{A}^\ic$ together with an additive functor
$i:\mathcal{A} \ra \mathcal{A}^\ic$ such
that any additive functor $F:\mathcal{A} \ra \mathcal{C}$ to an
idempotent complete additive category $\mathcal{C}$ factors as $F=i
\comp F^\ic$ for an additive functor $F^\ic:\mathcal{A}^\ic \ra
\mathcal{C}$ which is unique up to unique isomorphism; see 
e.\,g.\ \cite{balmer-schlichting} for an explicit construction. Then
$i$ is fully faithful and we usually view $\mathcal{A}$ as a full
subcategory of $\mathcal{A}^\ic$; it is a dense subcategory.
Conversely if
$\mathcal{A}$ is a full dense
additive subcategory of an idempotent complete additive category
$\mathcal{B}$, 
then $\mathcal{B}$ together with 
the inclusion $\mathcal{A} \hra \mathcal{B}$ is an idempotent
completion of $\mathcal{A}$.

\subsection{Triangulated categories}
\label{sec:triang-categ}

Let $\mathcal{T}$ (more precisely $(\mathcal{T}, [1])$ together with a
certain class of candidate triangles) be a
triangulated category (see \cite[Ch.~10]{KS-cat-sh},
\cite{neeman-tricat}, \cite{BBD}). We follow the terminology of
\cite{neeman-tricat} and call candidate triangle (resp.\ triangle)
what is called triangle (resp.\ distinguished triangle) in
\cite{KS-cat-sh}. We say that a subcategory $\mathcal{S} \subset
\mathcal{T}$ is \define{closed under extensions} if for any triangle $X
\ra Y \ra Z \ra [1]X$ in $\mathcal{T}$ with $X$ and $Z$ in
$\mathcal{S}$ we have $Y \in \mathcal{S}$.

\subsubsection{Basic statements about triangles}
\label{sec:basic-triangles}

\begin{lemma}
  \label{l:zero-triang-cat}
  Let $X \xra{f} Y \xra{g} Z \xra{h} [1]X$ be a triangle in
  $\mathcal{T}$. Then 
  \begin{enumerate}
  \item 
    \label{enum:zero-triang-cat-equiv}
    $h=0$ if and only if there is a morphism $\epsilon: Y \ra X$
    such that $\epsilon f =\id_X$.
  \end{enumerate}
  Let $\epsilon: Y \ra X$ be given satisfying $\epsilon f =\id_X$. Then:
  \begin{enumerate}[resume]
  \item 
    \label{enum:zero-triang-unique-delta}
    There
    is a unique morphism $\delta: Z \ra Y$ such that $\epsilon \delta
    = 0$ and $g \delta =\id_Z$.
  \item 
    \label{enum:zero-triang-isom}
    The morphism
    \begin{equation}
      \label{eq:split-triangle}
      \xymatrix{
        X \ar[r]^-f \ar[d]^-{\id_X} 
        & Y \ar[r]^-g \ar[d]^-{\svek \epsilon g} 
        & Z \ar[r]^-{h=0} \ar[d]^-{\id_Z} 
        & {[1]X} \ar[d]^-{\id_{[1]X}} \\
        {X} \ar[r]^-{\svek 10} 
        & {X \oplus Z} \ar[r]^-{\zvek 01}
        & Z \ar[r]^-0 
        & {[1]X}
      }
    \end{equation}
    is an isomorphism of triangles and $\svek \epsilon g$ is invertible with
    inverse $ \zvek f \delta$.
    Under the isomorphism $\svek \epsilon g:Y \sira X \oplus Z$ the
    morphism $\epsilon$ corresponds to $\zvek 10: X \oplus Z \ra X$ and the
    morphism $\delta$ to $\svek 01:Z \ra X \oplus Z$.
  \end{enumerate}
\end{lemma}

\begin{proof}
  If $\epsilon f=\id_X$ then $h =\id_{[1]X} \comp h = [1]\epsilon \comp
  [1]f \comp h =0$. If $h=0$ then use the cohomological functor
  $\Hom(?,X)$ to find $\epsilon$.

  Let $\epsilon: Y \ra X$ be given satisfying $\epsilon f=\id_X$. 
  The morphism \eqref{eq:split-triangle} is a morphism of candidate
  triangles and even of triangles since coproducts of triangles are
  triangles (e.\,g.\ \cite[Prop.~1.2.1]{neeman-tricat}). Hence $\svek
  \epsilon g$ is an isomorphism. 

  For any $\delta: Z \ra Y$ satisfying $\epsilon \delta
  = 0$ and $g \delta =\id_Z$ we have 
  $\svek \epsilon g \zvek f\delta =\tzmat 1001$. Hence $\delta$ is
  unique if it exists. 

  Let $\zvek ab$ be the inverse of $\svek \epsilon g$. Then $\id_Y =
  a\epsilon +bg$ and hence $f= \id_Y f= a\epsilon f+bgf =a$; on
  the other hand $\tzmat 1001 =
  \svek \epsilon g \zvek ab =\tzmat {\epsilon f}{\epsilon
    b}{gf}{gb}=\tzmat 1 {\epsilon b} 0 {g b}$. Hence $b$ satisfies the
  conditions imposed on $\delta$.
\end{proof}

We say that a triangle 
$X \xra{f} Y \xra{g} Z \xra{h} [1]X$ splits if it is isomorphic (by an
arbitrary isomorphism of triangles) to the
triangle $X \xra{\svek 10} X \oplus Z \xra{\zvek 01} Z \xra{0} [1]X$.
This is the case if and only if $h=0$ as we see from
Lemma~\ref{l:zero-triang-cat}.

\begin{corollary}
  [{cf.~\cite[Lemma~2.2]{Karoubianness}}]
  \label{c:retract-is-summand-in-tricat}
  Let $e: X \ra X$ be an idempotent endomorphism in $\mathcal{T}$. Then $e$
  splits if and only if $1-e$ splits. 
  In particular, an object $Y$ is a retract of an object $X$ if and
  only if $Y$ is a summand of $X$.
\end{corollary}

This corollary shows that the question of idempotent completeness of a
triangulated category is equivalent to the analog of
Question~\ref{q:KA-ic}.

\begin{proof}
  Let $(Y,p,i)$ be a splitting of $e$. Complete $i:Y \ra X$ into a
  triangle $Y \xra{i} X \xra{q} Z \ra [1]Y$. 
  Lemma~\ref{l:zero-triang-cat} \eqref{enum:zero-triang-unique-delta}
  applied
  to this triangle and $p:X \ra Y$ yields a morphism $j:Z \ra X$ and then
  Lemma~\ref{l:zero-triang-cat} \eqref{enum:zero-triang-isom} shows
  that $(Z, q, j)$ is a splitting of $1-e$.
\end{proof}

\begin{proposition}
  [{\cite[Prop.~1.1.9]{BBD}}]
  \label{p:BBD-1-1-9-copied-for-w-str}
  Let $(X,Y,Z)$ and $(X',Y',Z')$ be triangles and let
  $g:Y \ra Y'$ be a morphism in $\mathcal{T}$:
  \begin{equation}
    \label{eq:BBD-1-1-9-copied-for-w-str}
    \xymatrix{
      X \ar[r]^u \ar@{..>}[d]_f \ar@{}[dr]|{(1)} & Y \ar[r]^v \ar[d]_g
      \ar@{}[dr]|{(2)} & Z \ar[r]^d \ar@{..>}[d]_h & {[1]X} \ar@{..>}[d]_{[1]f}\\
      X' \ar[r]^{u'} & Y' \ar[r]^{v'} & Z' \ar[r]^{d'} & {[1]X' }
    }
  \end{equation}
  The following conditions are equivalent:
  \begin{enumerate}
  \item 
    $v'gu=0$;
  \item
    \label{enum:fkomm}
    there is a morphism $f$ such that (1) commutes;
  \item 
    \label{enum:hkomm}
    there is a morphism $h$ such that (2) commutes;
  \item 
    there is a morphism $(f,g,h)$ of triangles as indicated in diagram
    \eqref{eq:BBD-1-1-9-copied-for-w-str}.
  \end{enumerate}
  If these conditions are satisfied and $\Hom([1]X,Z')=0$ then
  $f$ in \eqref{enum:fkomm} and $h$ in \eqref{enum:hkomm}
  are unique.
\end{proposition}

\begin{proof}
  See {\cite[Prop.~1.1.9]{BBD}}.
\end{proof}

\begin{proposition}
  [{\cite[1.1.11]{BBD}}]
  \label{p:3x3-diagram-copied-for-w-str}
  Every commutative square
  \begin{equation*}
    \xymatrix{
      {X} \ar[r] & Y\\
      {X'} \ar[u] \ar[r] & {Y'} \ar[u]
    }
  \end{equation*}
  can be completed to a so called \textbf{$3\times 3$-diagram}
  \begin{equation*}
    \xymatrix{
      {[1]X'} \ar@{..>}[r] & 
      {[1]Y'} \ar@{..>}[r] &
      {[1]Z'} \ar@{..>}[r] \ar@{}[rd]|{\anticomm}& 
      {[2]X'} \\
      {X''} \ar[u] \ar[r] & 
      {Y''} \ar[u] \ar[r] &
      {Z''} \ar[u] \ar[r] & 
      {[1]X''} \ar@{..>}[u] \\
      {X} \ar[u] \ar[r] & 
      {Y} \ar[u] \ar[r] &
      {Z} \ar[u] \ar[r] & 
      {[1]X} \ar@{..>}[u] \\
      {X'} \ar[u] \ar[r] & 
      {Y'} \ar[u] \ar[r] &
      {Z'} \ar[u] \ar[r] & 
      {[1]X',} \ar@{..>}[u] 
    }
  \end{equation*}
  i.\,e.\ a diagram as above having the following properties: The
  dotted arrows are obtained 
  by translation $[1]$, all small squares are commutative except the
  upper right square marked with $\anticomm$ which is
  anti-commutative, and all three rows and columns with solid arrows
  are triangles. (The column/row with the dotted arrows becomes a
  triangle if an odd number of its morphisms is multiplied by
  $-1$.) 

  Variation: Any of the eight small commutative squares in the above
  diagram can be completed to a $3\times 3$-diagram as above.
  It is also possible to complete the anti-commutative square to a
  $3\times 3$ diagram.
\end{proposition}

\begin{proof}
  See \cite[1.1.11]{BBD}. To obtain the variations remove the column on
  the right and add the $[-1]$-shift of the ``$Z$-column" on the
  left. Modify the signs suitably. Iterate this and use the diagonal
  symmetry.
\end{proof}

\subsubsection{Anti-triangles}
\label{sec:anti-triangles}

Following \cite[10.1.10]{KS-cat-sh} we call a candidate triangle $X
\xra{f} Y \xra{g} Z \xra{h} [1]X$ an 
\textbf{anti-triangle} if $X \xra{f} Y \xra{g} Z \xra{-h} [1]X$ is a
triangle. Then $(\mathcal{T}, [1])$ with the class of all
anti-triangles is again a triangulated category that we denote by
$(\mathcal{T}^\anti, [1])$ or $\mathcal{T}^\anti$.
The triangulated categories $\mathcal{T}$ and $\mathcal{T}^\anti$ are
equivalent as triangulated categories
(cf.\ \cite[Exercise~10.10]{KS-cat-sh}): Let
$F=\id_\mathcal{T}:\mathcal{T} \ra \mathcal{T}$ be the identity
functor (of additive 
categories) and $\tau: F[1] \sira [1]F$ the isomorphism given by
$\tau_X=-\id_{[1]X}: [1]X=F[1]X \sira [1]X=[1]FX$. Then
\begin{equation}
  \label{eq:T-antiT-triequi}
  (\id_\mathcal{T}, \tau):\mathcal{T} \sira \mathcal{T}^\anti  
\end{equation}
is a triangulated equivalence.

\subsubsection{Adjoints are triangulated}
\label{sec:adjoints-triang}

Assume that $(G,\gamma): \mathcal{T} \ra \mathcal{S}$ is a
triangulated functor and that $F: \mathcal{S} \ra \mathcal{T}$ is left
adjoint to $G$. 
Let $X \in \mathcal{S}$. The morphism 
${[1]X} \ra {[1]GFX} \xsira{\gamma_{FX}\inv} {G[1]FX}$ 
(where the first morphism is the translation of the unit of the adjunction)
corresponds under the adjunction to a morphism
$\phi_X: F[1]X \ra [1]FX$.
This construction is natural in $X$ and defines a morphism $\phi: F[1]
\ra [1]F$ (which in fact is an isomorphism).
We omit the tedious proof of the following Proposition.
\begin{proposition}
  [{\cite[Prop.~1.6]{keller-vossieck} (without proof); cf.\
    \cite[Exercise~10.3]{KS-cat-sh}}] 
  \label{p:linksad-triang-for-wstr-art}
  Let $F$ be left adjoint to a triangulated functor $(G,\gamma)$.
  Then $(F, \phi)$ as defined above is a triangulated functor.

  Similarly the right adjoint to a triangulated functor is
  triangulated.
\end{proposition}

\subsubsection{Torsion pairs and t-structures}
\label{sec:t-str-and-torsion-theories}

The notion of a t-structure \cite[1.3.1]{BBD}
and that of a torsion pair 
\cite{beligiannis-reiten-torsion-theories}
on a triangulated category essentially coincide,
see \cite[Prop.~I.2.13]{beligiannis-reiten-torsion-theories}:
A pair $(\mathcal{X}, \mathcal{Y})$ is a torsion pair if and only if
$(\mathcal{X}, [1]\mathcal{Y})$ is a t-structure.
We will use both terms. 

\subsection{Homotopy categories and variants}
\label{sec:homotopy-categories}

Let $\mathcal{A}$ be an additive category and $C(\mathcal{A})$ the
category of (cochain) complexes in $\mathcal{A}$ with cochain maps as
morphisms: 
A morphism $f:X \ra Y$ in $C(\mathcal{A})$ is a sequence 
$(f^n)_{n \in \DZ}$ of morphisms $f^n: X^n \ra Y^n$ such that 
$d^n_Y f^n =f^{n+1} d^n_X$ for all $n \in \DZ$
(or in shorthand notation $df=fd$).

Let $f, g: X \ra Y$ be morphisms in $C(\mathcal{A})$. Then $f$ and
$g$ are \define{homotopic} if there is a sequence $h=(h^n)_{n \in
  \DZ}$ of morphisms $h^n: X^n \ra Y^{n-1}$ in $\mathcal{A}$ such that
$f^n-g^n= d^{n-1}_Y h^n + h^{n+1} d^n_X$ for all $n \in \DZ$ (or in
shorthand 
notation $f-g=dh+hd$).
The \define{homotopy category} $K(\mathcal{A})$ has the same objects 
as $C(\mathcal{A})$, but 
\begin{equation*}
  \Hom_{K(\mathcal{A})}(X,Y):= 
  \frac{\Hom_{C(\mathcal{A})}(X,Y)}{\{\text{morphisms homotopic to zero}\}}.
\end{equation*}

Let $f, g: X \ra Y$ be morphisms in $C(\mathcal{A})$. Then $f$ and
$g$ are \define{weakly homotopic} 
(see \cite[3.1]{bondarko-weight-str-vs-t-str})
if there is a pair $(s, t)$ of
sequences $s=(s^n)_{n \in \DZ}$ and 
$t=(t^n)_{n \in \DZ}$ of morphisms $s^n, t^n: X^n \ra Y^{n-1}$
such that 
\begin{equation*}
  f^n-g^n= d^{n-1}_Y s^n + t^{n+1} d^n_X    
\end{equation*}
for all $n \in \DZ$
(or in shorthand notation $f-g=ds+td$).
The \define{weak homotopy category} $K_{\weak}(\mathcal{A})$ has the
same objects  
as $C(\mathcal{A})$, but 
\begin{equation*}
  \Hom_{K(\mathcal{A})}(X,Y):= 
  \frac{\Hom_{C(\mathcal{A})}(X,Y)}{\{\text{morphisms weakly
      homotopic to zero}\}}. 
\end{equation*}

\begin{remark}
  Let $h$ and $(s,t)$ be as in the above definition (without asking
  for $f-g=dh+hd$ or $f-g=ds+td$). 
  Then $dh+hd:X \ra Y$ is homotopic to zero. 
  But
  $ds+td:X \ra Y$ is
  not necessarily weakly homotopic to zero: It need not be a morphism
  in $C(\mathcal{A})$.
  It is weakly homotopic to zero if and only if $dsd=dtd$ (which is of
  course the case if $f-g=ds+td$).
  
  Note that 
  weakly homotopic maps induce the same map on cohomology.
\end{remark}

All categories $C(\mathcal{A})$, $K(\mathcal{A})$ and
$K_{\weak}(\mathcal{A})$ are additive categories.
Let $[1]: C(\mathcal{A}) \ra C(\mathcal{A})$ be the functor that
maps an object $X$ to
$[1]X$ where $([1]X)^n=X^{n+1}$ and
$d_{[1]X}^n=-d_X^{n+1}$ and a morphism $f: X \ra Y$ to $[1]f$ where
$([1]f)^n=f^{n+1}$. This is an automorphism of $C(\mathcal{A})$ and
induces automorphisms $[1]$ of $K(\mathcal{A})$ and
$K_{\weak}(\mathcal{A})$. The categories $C(\mathcal{A})$,
$K(\mathcal{A})$ and $K_{\weak}(\mathcal{A})$ become additive categories with
translation. Sometimes we write $\Sigma$ instead of $[1]$.
Obviously there are canonical functors
\begin{equation}
  \label{eq:can-CKKweak}
  C(\mathcal{A}) \ra K(\mathcal{A}) \ra K_\weak(\mathcal{A})
\end{equation}
of additive categories with translation.

The category $K(\mathcal{A})$ has a natural structure of triangulated
category:
Given a morphism 
$m: M \ra N$ in $C(\mathcal{A})$ we define its mapping cone $\Cone(m)$
of $m$
to be the complex $\Cone(m)=N \oplus [1]M$ with differential
$\tzmat{d_N}{m}{0}{d_{[1] M}}=\tzmat{d_N}{m}{0}{-d_M}$. 
It fits into the following mapping cone sequence 
\begin{equation}
  \label{eq:mapping-cone-triangle}
  \xymatrix{
    & {M} \ar[r]^-{m}
    & {N} \ar[r]^-{\svek {1} 0}
    & {\Cone(m)} \ar[r]^-{\zvek 0{1}}
    & {[1]M}
  }
\end{equation}
in $C(\mathcal{A})$. 
The triangles of $K(\mathcal{A})$ are precisely the candidate
triangles that are 
isomorphic to the
image of a mapping cone sequence 
\eqref{eq:mapping-cone-triangle} in $K(\mathcal{A})$;
this image is called the mapping cone triangle of $m$.
We will later use: If we rotate the mapping cone triangle 
for $-m$ twice we obtain
the triangle
\begin{equation}
  \label{eq:mapcone-minus-m-rot2}
  \xymatrix{
    {[-1]N} \ar[r]^-{\svek {-1} 0}
    & {[-1]\Cone(-m)} \ar[r]^-{\zvek 0{-1}}
    & {M} \ar[r]^-{-m}
    & {N.}
  }
\end{equation}

In this setting
there
is apart from \eqref{eq:T-antiT-triequi}
another triangulated equivalence
between $K(\mathcal{A})$ and $K(\mathcal{A})^\anti$: 
The functor $S:C(\mathcal{A})\sira C(\mathcal{A})$ which sends a complex
$(X^n, d_X^n)$ to $(X^n, -d_X^n)$ and a morphism $f$ to $f$ descends to
a functor $S:K(\mathcal{A}) \sira K(\mathcal{A})$. Then 
\begin{equation}
  \label{eq:KA-antiKA-triequi}
  (S,\id):K(\mathcal{A}) \sira K(\mathcal{A})^\anti  
\end{equation}
(where $\id: S[1]\sira
[1]S$ is the obvious identification) is a triangulated equivalence:
Observe that $S$ maps
\eqref{eq:mapping-cone-triangle} to
\begin{equation*}
  \xymatrix{
    & {S(M)} \ar[r]^-{S(m)=m}
    & {S(N)} \ar[r]^-{\svek {1} 0}
    & {\Cone(-S(m))} \ar[r]^-{\zvek 0{1}}
    & {[1]S(M)}
  }
\end{equation*}
which becomes an anti-triangle in $K(\mathcal{A})$.

We introduce some notation:
Any functor $F: \mathcal{A} \ra \mathcal{B}$ of additive
categories obviously induces
a functor $F_{C}:C(\mathcal{A}) \ra C(\mathcal{B})$ of additive
categories with translation and a functor
$F_{K}:K(\mathcal{A}) \ra K(\mathcal{B})$ of triangulated categories.
If $F: \mathcal{A} \ra \mathcal{A}$ is an endofunctor we denote these
functors often by $F_{C(\mathcal{A})}$ and $F_{K(\mathcal{A})}$.

\section{Homotopy categories and idempotent completeness}
\label{sec:hot-cat-idem-complete}

Let $K(\mathcal{A})$ be the homotopy category of an additive category
$\mathcal{A}$. 
We define full subcategories of $K(\mathcal{A})$ as follows.
Let $K^{w \leq n}(\mathcal{A})$ consist of all
objects that are zero in all degrees $>n$
(where $w$ means ``weights"; the terminology will become clear
from Proposition~\ref{p:ws-hot-additive} below). The union of all $K^{w \leq 
n}(\mathcal{A})$ for $n \in \DZ$ is $K^-(\mathcal{A})$, the category
of all bounded above complexes.

Similarly we define 
$K^{w \geq n}(\mathcal{A})$ and $K^+(\mathcal{A})$.
Let $K^b(\mathcal{A})=K^-(\mathcal{A}) \cap K^+(\mathcal{A})$ be the
full subcategory of all bounded complexes.

For $* \in \{+, -, b, w \leq n, w \geq n\}$ let
$K(\mathcal{A})^*$ be the closure under isomorphisms of 
$K^*(\mathcal{A})$ in $K(\mathcal{A})$. Then all inclusions $K^*(\mathcal{A}) \subset
K(\mathcal{A})^*$ are equivalences of categories. 

We define $K(\mathcal{A})^{bw}:= K(\mathcal{A})^- \cap
K(\mathcal{A})^+$ (where $bw$ means ``bounded weights"). In general
the inclusion  $K(\mathcal{A})^b \subset 
K(\mathcal{A})^{bw}$ is not an equivalence (see
Rem.~\ref{rem:hot-bd-cplx-add-cat-not-idempotent-complete}).

\begin{theorem}
  \label{t:one-side-bounded-hot-idempotent-complete}
  Let $\mathcal{A}$ be an additive category and $n \in \DZ$. 
  The following categories are idempotent complete: 
  \begin{enumerate}
  \item
    \label{enum:one-side-bounded-hot-idempotent-complete-neg}
    $K^{w \leq n}(\mathcal{A})$, 
    $K(\mathcal{A})^{w \leq n}$, 
    $K^-(\mathcal{A})$,
    $K(\mathcal{A})^{-}$;
  \item 
    \label{enum:one-side-bounded-hot-idempotent-complete-pos}
    $K^{w \geq n}(\mathcal{A})$, 
    $K(\mathcal{A})^{w \geq n}$,
    $K^+(\mathcal{A})$,
    $K(\mathcal{A})^{+}$; 
  \item 
    \label{enum:one-side-bounded-hot-idempotent-complete-bounded}
    $K(\mathcal{A})^{bw}$.
  \end{enumerate}
\end{theorem}

\begin{remark}
  \label{rem:hot-bd-cplx-add-cat-not-idempotent-complete}
  In general the (equivalent) categories $K^b(\mathcal{A})$ and $K(\mathcal{A})^b$
  are not idempotent complete (and hence the 
  inclusion into the idempotent complete category $K(\mathcal{A})^{bw}$
  is not an equivalence): 
  Let $\modfin(k)$ be the category of finite dimensional vector spaces
  over a field $k$ and let 
  $\mathcal{E} \subset \modfin(k)$ be the full subcategory of even
  dimensional vector spaces. Note that $K(\mathcal{E}) \subset
  K(\modfin(k))$ is a full triangulated subcategory. 
  Then $X \mapsto \sum_{i \in \DZ} (-1)^i \dim H^i(X)$ is well defined
  on the objects of $K(\modfin(k))^b$. It takes even values on
  all objects of $K(\mathcal{E})^b$; hence $\tzmat 1000:k^2
  \ra k^2$ cannot split in $K(\mathcal{E})^b$. 
\end{remark}

\begin{remark}
  \label{rem:thomason-and-balmer-schlichting}
  Let us indicate the ideas behind the rather explicit proof of 
  Theorem~\ref{t:one-side-bounded-hot-idempotent-complete} which might
  perhaps be seen as a variation of the Eilenberg swindle.

  Let $\mathcal{T}'$ be an essentially small triangulated category.
  R.~W.~Thomason shows that taking the Grothendieck group establishes
  a bijection between 
  dense strict full triangulated subcategories of $\mathcal{T}'$
  and subgroups of its Grothendieck
  group $K_0(\mathcal{T}')$, see
  \cite[Section~3]{thomason-class}. 
  
  Now let $\mathcal{T}^\ic$ be the idempotent completion
  of an essentially small triangulated category $\mathcal{T}$,
  cf.\ Section~\ref{sec:idemp-completeness}; it carries a natural
  triangulated structure \cite[Thm.~1.5]{balmer-schlichting}.
  The previous result applied to $\mathcal{T}'=\mathcal{T}^\ic$ shows
  that the vanishing of $K_0(\mathcal{T}^\ic)$ implies that
  $\mathcal{T}$ is idempotent complete; this was observed by
  P.~Balmer and M.~Schlichting \cite[2.2-2.5]{balmer-schlichting}
  where they also provide a method that sometimes shows this vanishing
  condition.

  These results show directly that 
  $K^-(\mathcal{A})$, $K(\mathcal{A})^{-}$, $K^+(\mathcal{A})$,
  $K(\mathcal{A})^{+}$ and $K(\mathcal{A})^{bw}$ are idempotent
  complete if $\mathcal{A}$ is an essentially small additive category.

  A careful analysis of the proofs of these results essentially gives
  our proof of Theorem~\ref{t:one-side-bounded-hot-idempotent-complete} below.
\end{remark}

\begin{proof}
  We prove \eqref{enum:one-side-bounded-hot-idempotent-complete-neg} first.
  It is obviously sufficient to prove that 
  $\mathcal{T}^{w \leq n} := K^{w \leq n}(\mathcal{A})$ is 
  idempotent complete.
  Let $\mathcal{T}:= K^-(\mathcal{A})$ and consider the endofunctor 
  \begin{align*}
    S: \mathcal{T} & \ra \mathcal{T},\\
    X & \mapsto \bigoplus_{n \in \DN} [2n]X = X \oplus [2] X \oplus [4] X \oplus \dots
  \end{align*}
  (it is even triangulated by \cite[Prop.~1.2.1]{neeman-tricat}).
  Note that $S$ is well defined: Since $X$ is bounded above, the
  countable direct sum has only finitely many nonzero summands in
  every degree. There is an obvious 
  isomorphism of functors $S \sira \id \oplus [2] S$.
  The functor $S$ extends to a (triangulated) endofunctor
  of the idempotent completion $\mathcal{T}^\ic$ of $\mathcal{T}$,
  denoted by the same symbol, and we still have an   
  isomorphism $S \sira \id \oplus [2] S$.

  Now let $M$ be an object of $\mathcal{T}^{w \leq n}$ with an
  idempotent endomorphism $e: M \ra M$. In $\mathcal{T}^\ic$
  we obtain a direct sum 
  decomposition $M \cong E \oplus F$ such that $e$ corresponds to
  $\tzmat 1000: E \oplus F \ra E \oplus F$. We have to show that $E$ is
  isomorphic to an object of $\mathcal{T}^{w \leq n}$.
  
  Since $S$ preserves $\mathcal{T}^{w \leq n}$, we obtain that
  \begin{equation}
    \label{eq:smint}
    SM \cong
    SE \oplus SF \sira
    (E \oplus [2] SE) \oplus SF 
    \in \mathcal{T}^{w \leq n},
  \end{equation}
  where we use the convention just to write $X \in
  \mathcal{T}^{w \leq n}$ if an object $X \in \mathcal{T}^\ic$ is
  isomorphic to an object of $\mathcal{T}^{w \leq n}$.
  The direct sum of the triangles 
  \begin{align*}
    0 \ra & SE \xra{1} SE \ra 0,\\
    SE \ra & 0 \ra [1] SE \xra{1} [1] SE,\\
    SF \xra{1} & SF \ra 0 \ra [1] SF
  \end{align*}
  in $\mathcal{T}^\ic$ is the triangle
  \begin{equation*}
    SE \oplus SF \xra{\tzmat 0001}  
    SE \oplus SF \xra{\tzmat 1000}
    SE \oplus [1] SE \xra{\tzmat 0100}
    [1] (SE \oplus SF).
  \end{equation*}
  The first two vertices are isomorphic to 
  $SM \in \mathcal{T}^{w \leq n}$. The mapping cone of a map between
  objects of $\mathcal{T}^{w \leq n}$ is again in $\mathcal{T}^{w \leq
    n}$. We obtain that
  \begin{equation}
    \label{eq:sesigseint}
    SE \oplus [1] SE \in \mathcal{T}^{w \leq n}.
  \end{equation}
  Applying $[1]$ (which preserves $\mathcal{T}^{w \leq n}$) to the
  same statement for $F$ yields
  \begin{equation}
    \label{eq:sigsfsigsigsfint}
    [1] SF \oplus [2] SF \in\mathcal{T}^{w \leq n}.
  \end{equation}
  Taking the direct sum of the objects in \eqref{eq:smint},
  \eqref{eq:sesigseint} and \eqref{eq:sigsfsigsigsfint} shows that
  \begin{align*}
    E \oplus [2] SE \oplus SF
    \oplus &
    SE \oplus [1] SE 
    \oplus 
    [1] SF \oplus [2] SF\\
    \cong & E \oplus \big[SM \oplus [1] SM \oplus [2]
    SM\big] \in \mathcal{T}^{w \leq n}.
  \end{align*}
  Define 
  $R:= SM \oplus [1] SM \oplus [2] SM$ which is obviously in
  $\mathcal{T}^{w \leq n}$. Then
  the ``direct sum" triangle
  \begin{equation*}
    R \ra R\oplus E \ra E \xra{0} [1] R
  \end{equation*}
  shows that $E$ is isomorphic to an object 
  of $\mathcal{T}^{w \leq n}$.

  Now we prove
  \eqref{enum:one-side-bounded-hot-idempotent-complete-pos}.
  (The proof is essentially the same, but one has to pay attention to the fact
  that the mapping cone of a map between objects of
  $K^{w\geq n}(\mathcal{A})$ is only in
  $K^{w\geq n-1}(\mathcal{A})$.)
  Again it is
  sufficient to show that 
  $\mathcal{T}^{w \geq n} := K^{w\geq n}(\mathcal{A})$ is idempotent complete.
  Let 
  $\mathcal{T}:= K^+(\mathcal{A})$ and
  consider the (triangulated)
  functor 
  $S: \mathcal{T} \ra \mathcal{T}$, 
  $X \mapsto \bigoplus_{n \in \DN} [-2n]X = X \oplus [-2] X \oplus \dots$.
  It is well defined, extends to the idempotent completion
  $\mathcal{T}^\ic$ of $\mathcal{T}$, and we have
  an isomorphism $S \sira \id \oplus [-2] S$ of functors.
  Let $M$ in $\mathcal{T}^{w \geq n}$ with an
  idempotent endomorphism $e: M \ra M$. In $\mathcal{T}^\ic$
  we have a direct sum 
  decomposition $M \cong E \oplus F$ such that $e$ corresponds to
  $\tzmat 1000: E \oplus F \ra E \oplus F$. We have to show that $E$ is
  isomorphic to an object of $\mathcal{T}^{w \geq n}$.
  
  Since $S$ preserves $\mathcal{T}^{w \geq n}$, we obtain (with the
  analog of the
  convention introduced above) that
  \begin{equation}
    \label{eq:smintplus}
    SM \cong
    SE \oplus SF \sira
    (E \oplus [-2] SE) \oplus SF 
    \in \mathcal{T}^{w \geq n}.
  \end{equation}
  As above we have a triangle
  \begin{equation*}
    SE \oplus SF \xra{\tzmat 0001}  
    SE \oplus SF \xra{\tzmat 1000}
    SE \oplus [1] SE \xra{\tzmat 0100}
    [1] (SE \oplus SF).
  \end{equation*}
  The first two vertices are isomorphic to 
  $SM \in \mathcal{T}^{w \geq n}$; hence we get
  $SE \oplus [1] SE \in \mathcal{T}^{w \geq n-1}$ or equivalently
  \begin{equation}
    \label{eq:sesigseintplus}
    [-1] SE \oplus SE \in \mathcal{T}^{w \geq n}.
  \end{equation}
  Applying $[-1]$ (which preserves $\mathcal{T}^{w \geq n}$) to the
  same statement for $F$ yields
  \begin{equation}
    \label{eq:sigsfsigsigsfintplus}
    [-2] SF \oplus [{-1}] SF \in\mathcal{T}^{w \geq n}.
  \end{equation}
  Taking the direct sum of the objects in \eqref{eq:smintplus},
  \eqref{eq:sesigseintplus} and \eqref{eq:sigsfsigsigsfintplus} shows that
  \begin{align*}
    E \oplus [{-2}] SE \oplus SF 
    \oplus &
    [-1] SE \oplus SE
    \oplus
    [{-2}] SF \oplus [{-1}] SF\\
    \cong & E \oplus \big[SM \oplus [-1] SM \oplus [{-2}]
    SM\big] \in \mathcal{T}^{w \geq n}.
  \end{align*}
  Define 
  $R:= SM \oplus [-1] SM \oplus [{-2}] SM$ which is obviously in
  $\mathcal{T}^{w \geq n}$. Then
  the ``direct sum" triangle
  \begin{equation*} 
    E \ra E\oplus R \ra R \xra{0} [1] E
  \end{equation*}
  shows that $[1] E$ is isomorphic to an object 
  of $\mathcal{T}^{w \geq n-1}$. Hence $E$ is isomorphic to an
  object of $\mathcal{T}^{w \geq n}$.

  The statement 
  \eqref{enum:one-side-bounded-hot-idempotent-complete-bounded} 
  that $K(\mathcal{A})^{bw}$
  is idempotent complete is a consequence
  of
  \eqref{enum:one-side-bounded-hot-idempotent-complete-neg}
  and 
  \eqref{enum:one-side-bounded-hot-idempotent-complete-pos}.
\end{proof} 

\begin{theorem}
  \label{t:hot-idempotent-complete}
  Let $\mathcal{A}$ be an additive category. 
  \begin{enumerate}
  \item 
    \label{enum:hot-idempotent-complete-abelian}
    If $\mathcal{A}$ is abelian 
    then $K(\mathcal{A})$ is idempotent complete.
  \item 
    \label{enum:hot-idempotent-complete-count-coprod}
    If $\mathcal{A}$ has countable coproducts
    then $K(\mathcal{A})$ is idempotent complete.
  \item 
    \label{enum:hot-idempotent-complete-if-idempotent-complete}
    If $\mathcal{A}$ is idempotent complete then
    $K^b(\mathcal{A})$ and $K(\mathcal{A})^b$ are idempotent complete.
  \end{enumerate}
\end{theorem}

\begin{remark}
  \label{rem:KA-idempotent-complete-question}
  We do not know whether $K(\mathcal{A})$
  is idempotent complete for additive $\mathcal{A}$ (cf.\
  Question~\ref{q:KA-ic}).
\end{remark}

\begin{remark}
  \label{rem:KA-idempotent-complete-wlog-ic}
  If $\mathcal{A}^\ic$ is the idempotent completion of an additive
  category $\mathcal{A}$ then $K(\mathcal{A})$ is idempotent complete
  if and only if $K(\mathcal{A}^\ic)$ is idempotent complete:
  This follows from R.~W.~Thomason's results cited in
  Remark~\ref{rem:thomason-and-balmer-schlichting} (note that
  $K(\mathcal{A}) \subset K(\mathcal{A}^\ic)$ is dense) and
  Proposition~\ref{p:grothendieck-homotopy-category} below.
  Hence in Remark~\ref{rem:KA-idempotent-complete-question} one can
  assume without loss of generality that $\mathcal{A}$ is idempotent
  complete.
\end{remark}

\begin{proof}
  We prove \eqref{enum:hot-idempotent-complete-abelian} in
  Corollary~\ref{c:hot-idempotent-complete-abelian} below. 
  
  Let us show \eqref{enum:hot-idempotent-complete-count-coprod}.
  Assume that $\mathcal{A}$ has countable
  coproducts. Then $K(\mathcal{A})$ has countable coproducts; hence any
  idempotent endomorphism of an object of $K(\mathcal{A})$ splits by
  \cite[Prop.~1.6.8]{neeman-tricat}. 
  Another way to see this is to use the strategy explained in
  Remark~\ref{rem:thomason-and-balmer-schlichting}. More concretely,
  adapt the proof of 
  Theorem~\ref{t:one-side-bounded-hot-idempotent-complete}~\eqref{enum:one-side-bounded-hot-idempotent-complete-neg} in the obvious way: 
  Note that the functor $X \mapsto \bigoplus_{n \in \DN} [2n]X$ is
  well-defined on $K(\mathcal{A})$.

  For the proof of
  \eqref{enum:hot-idempotent-complete-if-idempotent-complete} assume
  now that $\mathcal{A}$ is idempotent complete. 
  Since $K^b(\mathcal{A}) \subset K(\mathcal{A})^b$ is an
  equivalence it is sufficient to show that $K^b(\mathcal{A})$ is
  idempotent complete. 

  Let $C$ be an object of
  $K^b(\mathcal{A})$.
  Let $X:=\bigoplus_{i \in \DZ} C^i$ be the finite direct sum over
  all nonzero components of $C$.
  Let $\add X \subset \mathcal{A}$ be the full subcategory of
  $\mathcal{A}$ that contains  
  $X$ and is closed under finite direct sums and summands.
  Since $\mathcal{A}$ is idempotent complete, 
  $\Hom(X,?): \mathcal{A} \ra \Mod(\End(X))$ induces an equivalence
  \begin{equation}
    \label{eq:project-equiv}
    \Hom(X,?): \add X \sira \proj(R)
  \end{equation}
  where $R=\End(X)$ and $\proj(R)$ is the category of all finitely
  generated projective right $R$-modules (see
  \cite[Prop.~1.3.1]{Krause-krull-schmidt-categories}).

  Note that $K^b(\add X)$ is a full triangulated subcategory of
  $K^b(\mathcal{A})$ containing $C$. Equivalence~\eqref{eq:project-equiv}
  yields an 
  equivalence
  $K^b(\add X) \sira K^b(\proj(R))$.
  The category $K^b(\proj(R))$ of perfect complexes is well known to
  be idempotent complete
  (e.\,g.\ \cite[Exercise~I.30]{KS} or
  \cite[Prop.~3.4]{neeman-homotopy-limits}).
  This implies that any idempotent endomorphism of $C$ splits.
\end{proof}

In the rest of this section we assume that $\mathcal{A}$ is abelian.
We use torsion pairs/t-structures in
order to prove that $K(\mathcal{A})$ is idempotent complete.

Let $\mathcal{K}^{h \geq 0} \subset K(\mathcal{A})$
be the full subcategory of
objects isomorphic to complexes $X$ of the form
\begin{equation*}
  \xymatrix{
    {\dots} \ar[r] &
    {0} \ar[r] &
    {X^{-2}} \ar[r]^{d^{-2}} &
    {X^{-1}} \ar[r]^{d^{-1}} &
    {X^0} \ar[r]^{d^0} &
    {X^1} \ar[r] &
    {\dots}
  }
\end{equation*}
with $X^0$ in degree zero and $d^{-2}$ the kernel of $d^{-1}$. 
(Here $\mathcal{K}$ stands for ``kernel" and ``$h \geq 0$" indicates that the
cohomology is concentrated in degrees $\geq 0$.)

Let $\mathcal{C}^{h \leq 0} \subset K(\mathcal{A})$ 
be the full subcategory of
objects isomorphic to complexes $X$ of the form
\begin{equation*}
  \xymatrix{
    {\dots} \ar[r] &
    {X^{-1}} \ar[r]^{d^{-1}} &
    {X^0} \ar[r]^{d^0} &
    {X^{1}} \ar[r]^{d^{1}} &
    {X^2} \ar[r] &
    {0} \ar[r] &
    {\dots}
  }
\end{equation*}
with $X^0$ in degree zero and $d^{1}$ the cokernel of $d^{0}$. 
(Here $\mathcal{C}$ stands for ``cokernel" and ``$h \leq 0$" indicates that the
cohomology is concentrated in degrees $\leq 0$.)

Define 
$\mathcal{C}^{h\leq n}:= [-n]\mathcal{C}^{h\leq 0}$ 
and
$\mathcal{K}^{h \geq n}:= [-n]\mathcal{K}^{h \geq 0}$. 

\begin{lemma}
  [{cf.~Example after
    \cite[Prop.~I.2.15]{beligiannis-reiten-torsion-theories}}]
  \label{l:torsion-pair-homotopy-abelian}
  Let $\mathcal{A}$ be abelian.
  Then 
  \begin{equation*}
    k:=(K(\mathcal{A})^{w \leq 0}, \mathcal{K}^{h \geq 1})
    \quad\text{and}\quad
    c:=(\mathcal{C}^{h \leq -1}, K(\mathcal{A})^{w \geq 0})
  \end{equation*}
  are torsion pairs on $K(\mathcal{A})$.
\end{lemma}

\begin{remark}
  It is easy to see that the torsion pair $k$ of 
  Lemma~\ref{l:torsion-pair-homotopy-abelian}
  coincides with the torsion pair defined 
  in the Example after
  \cite[Prop.~I.2.15]{beligiannis-reiten-torsion-theories}:
  An object $X \in K(\mathcal{A})$ is in 
  $K(\mathcal{A})^{w \leq n}$
  (resp.\ $\mathcal{K}^{h \geq n}$) if and only if the complex
  $\Hom(A,X)$ of abelian groups is exact in all degrees $>n$ (resp.\
  $<n$) for all $A \in \mathcal{A}$.
  The categories 
  $\mathcal{C}^{h \leq n}$ and $K(\mathcal{A})^{w \geq n}$ can be
  characterized similarly.
\end{remark}

\begin{proof}
  Let $(\mathcal{X}, \mathcal{Y})$ be the pair $k$ or 
  the pair $[1]c:=(\mathcal{C}^{h \leq -2}, K(\mathcal{A})^{w \geq -1})$.

  Let $f$ represent a morphism $X \ra Y$ with $X \in \mathcal{X}$ and $Y \in
  \mathcal{Y}$. Then the diagram
  \begin{equation*}
    \xymatrix{
      {X:} \ar[d]^f &
      {\dots} \ar[r] \ar[d] &
      {X^{-2}} \ar[r]^{d^{-2}} \ar[d] &
      {X^{-1}} \ar[r]^{d^{-1}} \ar[d]^{f^{-1}} &
      {X^0} \ar[r] \ar[d]^{f^0} \ar@{..>}[dl]|-{\exists !} &
      {0} \ar[r] \ar[d] &
      {0} \ar[r] \ar[d] &
      {\dots}\\
      {Y:} &
      {\dots} \ar[r] &
      {0} \ar[r] &
      {Y^{-1}} \ar[r]^{d^{-1}} &
      {Y^0} \ar[r]^{d^0} &
      {Y^1} \ar[r]^{d^1} &
      {Y^2} \ar[r] &
      {\dots}
    }
  \end{equation*}
  shows that $f$ is homotopic to zero; hence
  $\Hom_{K(\mathcal{A})}(\mathcal{X}, \mathcal{Y})=0$. 
  It is obvious that $\mathcal{X}$ is stable under $[1]$ and
  $\mathcal{Y}$ is stable under $[-1]$.

  We need to show that any object $A=(A^j, d^j)$ of $K(\mathcal{A})$ fits into a
  triangle $X \ra A \ra Y \ra [1]X$ with $X \in \mathcal{X}$ and $Y
  \in \mathcal{Y}$.
  
  \begin{itemize}
  \item 
    Case $(\mathcal{X}, \mathcal{Y})=k$:
    Let $f: M \ra A^0$ be the kernel of $d^0:A^0 \ra A^1$; then
    there is a unique morphism $g:A^{-1} \ra M$ such that
    $d^{-1}=f g$.
  \item 
    Case $(\mathcal{X}, \mathcal{Y})=[1]c$:
    Let 
    $g: A^{-1} \ra M$ be the cokernel of $d^{-2}:A^{-2} \ra A^{-1}$; then
    there is a unique morphism $f:M \ra A^{0}$ such that
    $d^{-1}=f g$.
  \end{itemize}
  
  The commutative diagram 
  \begin{equation*}
    \xymatrix{
      {X:} \ar[d] &
      {\dots} \ar[r] &
      {A^{-2}} \ar[r]^{d^{-2}} \ar[d]^1 &
      {A^{-1}} \ar[r]^{g} \ar[d]^1 &
      {M} \ar[r] \ar[d]^{f} &
      {0} \ar[r] \ar[d] &
      {0} \ar[r] \ar[d] &
      {\dots}\\
      {A:} \ar[d] &
      {\dots} \ar[r] &
      {A^{-2}} \ar[r]^{d^{-2}} \ar[d] &
      {A^{-1}} \ar[r]^{d^{-1}} \ar[d]^{g} &
      {A^0} \ar[r]^{d^0} \ar[d]^1 &
      {A^1} \ar[r]^{d^1} \ar[d]^1 &
      {A^2} \ar[r] \ar[d]^1 &
      {\dots}\\
      {Y:} \ar[d] &
      {\dots} \ar[r] &
      {0} \ar[r] \ar[d] &
      {M} \ar[r]^{f} \ar[d]^1 &
      {A^0} \ar[r]^{d^0} \ar[d] &
      {A^1} \ar[r]^{d^1} \ar[d] &
      {A^2} \ar[r] \ar[d] &
      {\dots} \\
      {[1]X:} &
      {\dots} \ar[r] &
      {A^{-1}} \ar[r]^{-g} &
      {M} \ar[r] &
      {0} \ar[r] &
      {0} \ar[r] &
      {0} \ar[r] &
      {\dots}\\
    }
  \end{equation*}
  defines a candidate
  triangle $X \ra A \ra Y \ra [1]X$ in $K(\mathcal{A})$.
  It is easy to check that it is in fact a triangle.
\end{proof}

\begin{corollary}
  \label{c:hot-idempotent-complete-abelian}
  Let $\mathcal{A}$ be an abelian category
  and $n \in \DZ$. 
  Then $K(\mathcal{A})$ is
  idempotent complete, and the same is true for 
  $\mathcal{K}^{h\geq n}$ and $\mathcal{C}^{h \leq n}$.
\end{corollary}

\begin{proof}
  Let $(\mathcal{X}, \mathcal{Y})$ be one of the torsion pairs of 
  Lemma~\ref{l:torsion-pair-homotopy-abelian}.

  Let $e: A \ra A$ be an idempotent endomorphism in $K(\mathcal{A})$.
  The truncation functors $\tau_{\mathcal{X}}: K(\mathcal{A}) \ra
  \mathcal{X}$ and $\tau_{\mathcal{Y}}: K(\mathcal{A}) \ra \mathcal{Y}$
  yield a morphism of triangles
  \begin{equation*}
    \xymatrix{
      {\tau_{\mathcal{X}}(A)} \ar[r]^-u \ar[d]^{\tau_{\mathcal{X}}(e)} &
      {A} \ar[r]^-v \ar[d]^{e} &
      {\tau_{\mathcal{Y}}(A)} \ar[r]^-w \ar[d]^{\tau_{\mathcal{Y}}(e)} &
      {[1]\tau_{\mathcal{X}}(A)} \ar[d]^{[1]\tau_{\mathcal{X}}(e)} \\
      {\tau_{\mathcal{X}}(A)} \ar[r]^-u &
      {A} \ar[r]^-v &
      {\tau_{\mathcal{Y}}(A)} \ar[r]^-w &
      {[1]\tau_{\mathcal{X}}(A).}
    }
  \end{equation*}
  All morphisms ${\tau_{\mathcal{X}}(e)}$, $e$, ${\tau_{\mathcal{Y}}(e)}$ are
  idempotent 
  and 
  ${\tau_{\mathcal{X}}(e)}$ and ${\tau_{\mathcal{Y}}(e)}$ split by
  Theorem~\ref{t:one-side-bounded-hot-idempotent-complete}
  (since $\mathcal{K}^{h \geq n} \subset K(\mathcal{A})^{w\geq n-2}$
  and
  $\mathcal{C}^{h \leq n} \subset K(\mathcal{A})^{w\leq n+2}$).
  Hence $e$ splits by 
  \cite[Prop.~2.3]{Karoubianness}.
  This shows that $K(\mathcal{A})$ is idempotent complete.

  Since $\mathcal{X}=\leftidx{^\perp}{\mathcal{Y}}$
  and $\mathcal{Y}=\mathcal{X}^\perp$ for any torsion pair
  this implies that $\mathcal{X}$ and $\mathcal{Y}$ are idempotent
  complete.
\end{proof}

\section{Weight structures} 
\label{sec:weight-structures}

The following definition of a weight structure is independently due to
M.~Bondarko \cite{bondarko-weight-str-vs-t-str} 
and D.~Pauksztello \cite{pauk-co-t} who calls it a co-t-structure.

\begin{definition}
  \label{d:ws}
  Let $\mathcal{T}$ be a triangulated category. 
  A \define{weight structure}
  (or \define{w-structure})
  on $\mathcal{T}$ is a 
  pair $w=(\mathcal{T}^{w \leq 0},\mathcal{T}^{w \geq 0})$
  of two full subcategories such that:
  (We define
  $\mathcal{T}^{w \leq n} :=[-n]\mathcal{T}^{w \leq 0}$ and
  $\mathcal{T}^{w \geq n} :=[-n]\mathcal{T}^{w \geq 0}$.)
  \begin{enumerate}[label=(ws{\arabic*})]
  \item 
    \label{enum:ws-i} 
    $\mathcal{T}^{w \leq 0}$ and $\mathcal{T}^{w \geq 0}$ are additive
    categories and
    closed under retracts
    in $\mathcal{T}$;
  \item 
    \label{enum:ws-ii} 
    $\mathcal{T}^{w \leq 0} \subset
    \mathcal{T}^{w \leq 1}$ and $\mathcal{T}^{w \geq 1} \subset
    \mathcal{T}^{w \geq 0}$;
  \item 
    \label{enum:ws-iii} 
    $\Hom_\mathcal{T}(\mathcal{T}^{w \geq 1}, \mathcal{T}^{w \leq 0})=0$;
  \item 
    \label{enum:ws-iv} 
    For every $X$ in $\mathcal{T}$ there is a triangle
    \begin{equation*}
      A \ra X \ra B \ra [1]A
    \end{equation*}
    in $\mathcal{T}$ with $A$ in $\mathcal{T}^{w \geq 1}$ and $B$ in
    $\mathcal{T}^{w \leq 0}$. 
  \end{enumerate}

  A weight structure $w=(\mathcal{T}^{w \leq 0},\mathcal{T}^{w \geq
    0})$ is \textbf{bounded above} if $\mathcal{T}=\bigcup_{n \in \DZ}
  \mathcal{T}^{w \leq n}$ and \textbf{bounded below} 
  if $\mathcal{T}=\bigcup_{n \in \DZ} \mathcal{T}^{w \geq n}$.
  It is \textbf{bounded} if it is bounded above and bounded below.

  The \define{heart} of a 
  weight structure $w=(\mathcal{T}^{w \leq 0},\mathcal{T}^{w \geq 0})$
  is
  \begin{equation*}
    \heart(w):= \mathcal{T}^{w=0}:= \mathcal{T}^{w \leq 0}\cap
    \mathcal{T}^{w \geq 0}.   
  \end{equation*}
  
  A \define{weight category} (or \define{w-category}) is a pair
  $(\mathcal{T}, w)$ where 
  $\mathcal{T}$ is a triangulated category and $w$ is a weight
  structure on $\mathcal{T}$. Its \define{heart} is the heart of $w$.
\end{definition}

A triangle 
$A \ra X \ra B \ra [1]A$ with $A$ in $\mathcal{T}^{w \geq n+1}$ and $B$ in
$\mathcal{T}^{w \leq n}$ 
(cf.~\ref{enum:ws-iv})
is called a 
\define{weight decomposition} of $X$, or more precisely a 
\define{$(w\geq n+1, w \leq n)$-weight decomposition} 
or a \define{weight decomposition of type $(w\geq n+1, w \leq n)$} of
$X$. 
It is convenient to write such a weight decomposition as $w_{\geq
n+1}X \ra X \ra w_{\leq n}X \ra [1]w_{\geq n+1}X$ where $w_{\geq
n+1}X$ and $w_{\leq n}X$ are just names for the objects $A$ and
$B$ from above. If we say
that $w_{\geq n+1}X \ra X \ra w_{\leq n}X \ra [1]w_{\geq n+1}X$ is a
weight decomposition without specifying its type explicitly, this type is usually
obvious from the notation.

The heart $\heart(w)$ is a full subcategory of $\mathcal{T}$ and
closed under retracts in $\mathcal{T}$ by
\ref{enum:ws-i}. 

We will use the following notation (for $a, b \in \DZ$):
$\mathcal{T}^{w \in [a,b]} := \mathcal{T}^{w \leq b} \cap
\mathcal{T}^{w \geq a}$, $\mathcal{T}^{w=a}:= \mathcal{T}^{w \in
  [a,a]}$.
Note that $\mathcal{T}^{w \in [a,b]}=0$ if $b<a$ by 
\ref{enum:ws-iii}: For $X \in \mathcal{T}^{w \in [a,b]}$ we have $\id_X=0$.

\begin{definition}
  \label{def:w-exact}
  Let $\mathcal{T}$ and $\mathcal{S}$ be weight categories. A triangulated
  functor $F: \mathcal{T} \ra \mathcal{S}$ is called 
  \define{w-exact} (or \define{weight-exact})  if 
  $F(\mathcal{T}^{w \leq 0}) \subset \mathcal{S}^{w \leq 0}$ and
  $F(\mathcal{T}^{w \geq 0}) \subset \mathcal{S}^{w \geq 0}$.
\end{definition}

\subsection{First properties of weight structures}
\label{sec:first-properties-ws}

Let $\mathcal{T}$ be a triangulated category with a weight structure
$w=(\mathcal{T}^{w \leq 0},\mathcal{T}^{w \geq 0})$.

\begin{lemma}
  [{cf.\ \cite[Prop.~1.3.3]{bondarko-weight-str-vs-t-str} for
    some statements}]
  \label{l:weight-str-basic-properties}
  \rule{0cm}{1mm}
  \begin{enumerate}
  \item 
    \label{enum:triangle-heart-splits}
    Let $X \xra{f} Y \xra{g} Z \xra{h} [1]X$ be a triangle in
    $\mathcal{T}$.
    If $Z \in \mathcal{T}^{w \geq n}$ and $X \in \mathcal{T}^{w \leq n}$
    then this triangle splits.
    
    In particular any triangle $X \xra{f} Y \xra{g} Z \xra{h} [1]X$
    with all objects $X, Y, Z$ in the heart $\heart(w)$ splits.
  \end{enumerate}
  Let
  \begin{equation}
    \label{eq:weight-decomp-X}
    w_{\geq n+1}X \xra{f} X \xra{g} w_{\leq n}X \xra{h} [1]w_{\geq n+1}X  
  \end{equation}
  be
  a
  $(w\geq n+1, w \leq n)$-weight decomposition of $X$, for some $n \in \DZ$.
  \begin{enumerate}[resume]
  \item 
    \label{enum:weight-decomp-with-knowlegde-dir-summand-i}
    If $X$ is in $\mathcal{T}^{w \leq n}$, then $w_{\leq n}X \cong X \oplus
    [1]w_{\geq n+1}X$
  \item 
    \label{enum:weight-decomp-with-knowlegde-dir-summand-ii}
    If $X$ is in $\mathcal{T}^{w \geq n+1}$, then $w_{\geq n+1}X\cong X\oplus
    [-1]w_{\leq n}X$.
  \item 
    \label{enum:weight-perp-prop}
    For every $n \in \DZ$ we have
    \begin{align}
      \label{eq:ws-leq-right-orth-of-geq}
      (\mathcal{T}^{w\geq n+1})^\perp & =\mathcal{T}^{w \leq n}, \\
      \label{eq:ws-geq-left-orth-of-leq}
      \leftidx{^{\perp}}{(\mathcal{T}^{w\leq n})}{} & =\mathcal{T}^{w \geq n+1}.
    \end{align}
    In particular 
    $\mathcal{T}^{w \leq n}$ and $\mathcal{T}^{w \geq n+1}$ are closed
    under extensions.
  \item 
    \label{enum:weights-bounded}
    Assume that $a \leq n < b$ (for $a,b \in \DZ$) and that $X \in
    \mathcal{T}^{w \in [a,b]}$. Then
    $w_{\leq n}X \in \mathcal{T}^{w \in [a,n]}$ and
    $w_{\geq n+1}X \in \mathcal{T}^{w \in [n+1,b]}$.

    More precisely: If $a \leq n$ then $X \in \mathcal{T}^{w \geq a}$
    implies $w_{\leq n}X \in 
    \mathcal{T}^{w \in [a,n]}$ (and obviously $w_{\geq n+1}X \in
    \mathcal{T}^{w \geq n+1} \subset \mathcal{T}^{w \geq a}$). 
    If $n < b$ then 
    $X \in \mathcal{T}^{w \leq b}$
    implies (obviously $w_{\leq n}X \in \mathcal{T}^{w \leq n} \subset
    \mathcal{T}^{w \leq b}$ and) 
    $w_{\geq n+1}X \in \mathcal{T}^{w \in [n+1, b]}$.
  \item 
    \label{enum:bounded-w-decomp}
    Let $a,b,n \in \DZ$.
    For $X \in \mathcal{T}^{w \geq a}$ 
    (resp.\ $X \in \mathcal{T}^{w \leq b}$
    or $X \in \mathcal{T}^{w \in [a,b]}$) 
    there is a 
    $(w\geq n+1, w \leq n)$-weight decomposition 
    \eqref{eq:weight-decomp-X} of $X$ such that 
    both $w_{\geq n+1}X$ and $w_{\leq n}X$ are in 
    $\mathcal{T}^{w \geq a}$ 
    (resp.\ in $\mathcal{T}^{w \leq b}$
    or $\mathcal{T}^{w \in [a,b]}$).
  \end{enumerate}
\end{lemma}

\begin{proof}
  By \ref{enum:ws-iii} we have 
  $h=0$ in \eqref{enum:triangle-heart-splits},
  $f=0$ in 
  \eqref{enum:weight-decomp-with-knowlegde-dir-summand-i} 
  and $g=0$ in
  \eqref{enum:weight-decomp-with-knowlegde-dir-summand-ii};
  use Lemma~\ref{l:zero-triang-cat}.

  We prove \eqref{enum:weight-perp-prop}.
  Axiom \ref{enum:ws-iii} shows that the inclusions $\supset$ in
  \eqref{eq:ws-leq-right-orth-of-geq} and
  \eqref{eq:ws-geq-left-orth-of-leq} are true.   
  Let $X \in \mathcal{T}$ and take a weight decomposition
  \eqref{eq:weight-decomp-X} of $X$.
  If $X \in (\mathcal{T}^{w\geq n+1})^\perp$ then $f=0$ by
  \ref{enum:ws-iii}; hence $X$
  is a summand of $w_{\leq n}X \in \mathcal{T}^{w\leq n}$ and hence in 
  $\mathcal{T}^{w\leq n}$ by \ref{enum:ws-i}.
  Similarly $X \in \leftidx{^\perp}{(\mathcal{T}^{w\leq n})}$ implies
  $g=0$ so $X$ is a summand of $w_{\geq n+1}X \in \mathcal{T}^{w\geq n+1}$
  and hence in $\mathcal{T}^{w\geq n+1}$.

  Let us prove \eqref{enum:weights-bounded}:
  Since $X \in \mathcal{T}^{w \geq a}$ and $[1]w_{\geq n+1}X \in
  \mathcal{T}^{w\geq n} \subset \mathcal{T}^{w \geq a}$ and
  $\mathcal{T}^{w \geq a}$ is closed under extensions 
  by \eqref{enum:weight-perp-prop}
  we have $w_{\leq n}X \in \mathcal{T}^{w \in [a,n]}$.
  Turning the triangle we see that $w_{\geq n+1}X$ is an extension of
  $[-1]w_{\leq n}X \in \mathcal{T}^{w \leq n+1} 
  \subset \mathcal{T}^{w \leq b}$ 
  and $X \in \mathcal{T}^{w \leq b}$, hence $w_{\geq n+1}X \in \mathcal{T}^{w \in
    [n+1,b]}$.

  We prove \eqref{enum:bounded-w-decomp}:
  Assume $X \in \mathcal{T}^{w\geq a}$. If $a \leq n$ any such weight
  decomposition does the job by \eqref{enum:weights-bounded};
  if $a > n$ take $X \xra{\id} X \ra 0 \ra [1]X$.

  Assume $X \in \mathcal{T}^{w\leq b}$. If $n < b$ use
  \eqref{enum:weights-bounded}; 
  if $b \leq n$ take $0 \ra X \xra{\id} X \ra 0$.

  Assume $X \in \mathcal{T}^{w\in [a,b]}$. 
  The case $a > b$ is trivial since then $X=0$.
  So assume $a\leq b$.
  If $a \leq n < b$ use
  \eqref{enum:weights-bounded}; 
  if $a > n$ take $X \xra{\id} X \ra 0 \ra [1]X$.
  if $b \leq n$ take $0 \ra X \xra{\id} X \ra 0$.
\end{proof}

\begin{corollary}
  \label{c:inclusion-w-str}
  Let
  $(\mathcal{D}^{w \leq 0}, \mathcal{D}^{w \geq 0})$ and
  $(\mathcal{T}^{w \leq 0}, \mathcal{T}^{w \geq 0})$ be two
  weight structures on a triangulated category.
  If
  \begin{equation*}
    \mathcal{D}^{w \leq 0} \subset \mathcal{T}^{w \leq 0} 
    \text{ and }
    \mathcal{D}^{w \geq 0} \subset \mathcal{T}^{w \geq 0}, 
  \end{equation*}
  then these two weight structures coincide.
\end{corollary}

\begin{proof}
  Our assumptions and \eqref{eq:ws-leq-right-orth-of-geq} give
  \begin{equation*}
    \mathcal{T}^{w \leq 0} = (\mathcal{T}^{w\geq 1})^\perp 
    \subset
    (\mathcal{D}^{w\geq 1})^\perp =\mathcal{D}^{w \leq 0}.    
  \end{equation*}
  Similarly we obtain 
  $\mathcal{T}^{w \geq 0} \subset \mathcal{D}^{w \geq 0}$.
\end{proof}

The following Lemma is the analog of \cite[1.3.19]{BBD}.

\begin{lemma}
  \label{l:induced-w-str}
  Let $\mathcal{T}'$ be a full triangulated subcategory of a
  triangulated category $\mathcal{T}$.  
  Assume that 
  $w=(\mathcal{T}^{w \leq 0},\mathcal{T}^{w \geq 0})$ is a weight
  structure on $\mathcal{T}$. 
  Let 
  $w'=(\mathcal{T}' \cap \mathcal{T}^{w \leq 0}, \mathcal{T}' \cap
  \mathcal{T}^{w \geq 0})$. 
  Then 
  $w'$ is a weight structure on $\mathcal{T}'$ if and only if for
  any object $X \in \mathcal{T}'$ there is a triangle
  \begin{equation}
    \label{eq:induced-w-str-w-decomp}
    \xymatrix{
      {w_{\geq 1}X} \ar[r] 
      & {X} \ar[r]
      & {w_{\leq 0}X} \ar[r]
      & {[1]w_{\geq 1}X}
    }
  \end{equation}
  in $\mathcal{T}'$ that is a weight
  decomposition 
  of type $(w \geq 1, w \leq 0)$ in $(\mathcal{T},w)$.

  If $w'$ is a weight structure on $\mathcal{T}'$ it is called the
  \define{induced weight structure}. 
\end{lemma}

\begin{proof}
  If $w'$ is a weight structure on $\mathcal{T}'$ weight
  decompositions in $(\mathcal{T}', w')$ are triangles and yield
  weight decompositions in 
  $(\mathcal{T}, w)$.

  Conversely let us show that under the given condition $w'$ is a
  weight structure on $\mathcal{T}'$. 
  This condition obviously says that $w'$ satisfies \ref{enum:ws-iv}. 
  If $Y$ is a retract in $\mathcal{T}'$ of 
  $X \in \mathcal{T}'\cap \mathcal{T}^{w \leq 0}$, it is a retract of
  $X$ in $\mathcal{T}$ and hence $Y \in \mathcal{T}^{w \leq 0}$.
  This proves that $\mathcal{T}'\cap \mathcal{T}^{w \leq 0}$ is
  closed under retracts in $\mathcal{T}'$, cf.\ \ref{enum:ws-i}. 
  The remaining conditions for $w'$ being are weight structure are obvious.
\end{proof}

\subsection{Basic example}
\label{sec:basic-ex}

Let $\mathcal{A}$ be an additive category and $K(\mathcal{A})$ its homotopy category.
We use the notation introduced in Section~\ref{sec:hot-cat-idem-complete}.

\begin{proposition}
  [{cf.\ \cite{bondarko-weight-str-vs-t-str}, \cite{pauk-co-t}}]
  \label{p:ws-hot-additive}
  The pair
  \begin{equation*}
    (K(\mathcal{A})^{w \leq 0}, K(\mathcal{A})^{w \geq 0})  
  \end{equation*}
  is a weight
  structure on $K(\mathcal{A})$.

  It induces (see Lemma~\ref{l:induced-w-str}) weight structures on
  $K^*(\mathcal{A})$  
  for $* \in \{+, -, b\}$
  and on $K(\mathcal{A})^*$ for $* \in \{+, -, b, bw\}$.

  All these weight structures are called the 
  \define{standard weight structure} on the respective category.
\end{proposition}

\begin{remark}
  \label{rem:ws-hot-additive-anti}
  The triangulated equivalence \eqref{eq:KA-antiKA-triequi}
  between $K(\mathcal{A})$ and $K(\mathcal{A})^\anti$
  allows us to transfer the weight structure from
  Proposition~\ref{p:ws-hot-additive} to 
  $K(\mathcal{A})^\anti$. This defines the 
  \define{standard weight structure} 
  \begin{equation*}
    (K(\mathcal{A})^{\anti, w \leq 0}, K(\mathcal{A})^{\anti, w \geq 0})  
  \end{equation*}
  on $K(\mathcal{A})^\anti$.
  We have 
  $K(\mathcal{A})^{\anti, w \leq 0}=K(\mathcal{A})^{w \leq 0}$
  and $K(\mathcal{A})^{\anti, w \geq 0}= K(\mathcal{A})^{w \geq 0}$.
  Similarly one can transfer the induced weight structures.
\end{remark}

\begin{proof}
  Condition \ref{enum:ws-i}: It is obvious that both
  $K(\mathcal{A})^{w \leq 0}$ and 
  $K(\mathcal{A})^{w \geq 0}$ are additive. 
  Since they are strict subcategories of $K(\mathcal{A})$ and 
  idempotent complete by
  Theorem~\ref{t:one-side-bounded-hot-idempotent-complete},
  they are in particular closed under retracts in $K(\mathcal{A})$.
  Conditions \ref{enum:ws-ii} and \ref{enum:ws-iii} are obvious.

  We verify condition \ref{enum:ws-iv} explicitly. 
  Let $X=(X^i, d^i:X^i
  \ra X^{i+1})$ be a complex.
  We give $(w \geq n+1, w \leq n)$-weight decompositions of $X$ for
  any $n \in \DZ$.
  The following diagram
  defines
  complexes $\ul{w}_{\leq n}(X)$, $\ul{w}_{\geq n+1}(X)$ and a mapping
  cone sequence 
  $[-1]\ul{w}_{\leq n}(X) \ra {\ul{w}_{\geq n+1}(X)} \ra X \ra
    \ul{w}_{\leq n}(X)$
  in $C(\mathcal{A})$:
  \begin{equation*}
    \hspace{-0.4cm}
    \xymatrix
    {
      {[-1]{\ul{w}_{\leq n}(X)}:} \ar[d] & 
      {\dots} \ar[r] &
      {X^{{n-2}}} \ar[r]^-{-d^{{n-2}}} \ar[d] &
      {X^{{n-1}}} \ar[r]^-{-d^{{n-1}}} \ar[d] &
      {X^n} \ar[r] \ar[d]^-{d^n} &
      {0} \ar[r] \ar[d] &
      {\dots} \\
      {{\ul{w}_{\geq n+1}(X)}:} \ar[d] &  
      {\dots} \ar[r] &
      {0} \ar[r] \ar[d] &
      {0} \ar[r] \ar[d] &
      {X^{n+1}} \ar[r]^-{d^{n+1}} \ar[d]^-{1} &
      {X^{n+2}} \ar[r] \ar[d]^-1 &
      {\dots} \\
      {X:} \ar[d] & 
      {\dots} \ar[r] &
      {X^{{n-1}}} \ar[r]^-{d^{{n-1}}} \ar[d]^-1 &
      {X^n} \ar[r]^-{d^n} \ar[d]^-1 &
      {X^{n+1}} \ar[r]^-{d^{n+1}} \ar[d] &
      {X^{n+2}} \ar[r] \ar[d] &
      {\dots} \\
      {{\ul{w}_{\leq n}(X)}:} & 
      {\dots} \ar[r] &
      {X^{{n-1}}} \ar[r]^-{d^{{n-1}}} &
      {X^n} \ar[r] &
      {0} \ar[r] &
      {0} \ar[r] &
      {\dots.} 
    }
  \end{equation*}
  Passing to $K(\mathcal{A})$ and rotation of the triangle
  yields the weight decomposition we need:
  \begin{equation}
    \label{eq:hot-wdecomp}
    \hspace{-0.4cm}
    \xymatrix
    {
      {{\ul{w}_{\geq n+1}(X)}:} \ar[d] &  
      {\dots} \ar[r] &
      {0} \ar[r] \ar[d] &
      {0} \ar[r] \ar[d] &
      {X^{n+1}} \ar[r]^-{d^{n+1}} \ar[d]^-{1} &
      {X^{n+2}} \ar[r] \ar[d]^-1 &
      {\dots} \\
      {X:} \ar[d] & 
      {\dots} \ar[r] &
      {X^{{n-1}}} \ar[r]^-{d^{{n-1}}} \ar[d]^-1 &
      {X^n} \ar[r]^-{d^n} \ar[d]^-1 &
      {X^{n+1}} \ar[r]^-{d^{n+1}} \ar[d] &
      {X^{n+2}} \ar[r] \ar[d] &
      {\dots} \\
      {{\ul{w}_{\leq n}(X)}:} \ar[d] & 
      {\dots} \ar[r] &
      {X^{{n-1}}} \ar[r]^-{d^{{n-1}}} \ar[d] &
      {X^n} \ar[r] \ar[d]^-{-d^n} &
      {0} \ar[r] \ar[d] &
      {0} \ar[r] \ar[d] &
      {\dots} \\
      {[1]{\ul{w}_{\geq n+1}(X)}:} &  
      {\dots} \ar[r] &
      {0} \ar[r] &
      {X^{n+1}} \ar[r]^-{-d^{n+1}} &
      {X^{n+2}} \ar[r]^-{-d^{n+2}} &
      {X^{n+3}} \ar[r] &
      {\dots} \\
    }
  \end{equation}

  The statements about the induced weight structures 
  on $K^*(\mathcal{A})$ and $K(\mathcal{A})^*$
  for $* \in \{+,-,b\}$ are obvious
  from \eqref{eq:hot-wdecomp}
  and Lemma~\ref{l:induced-w-str}.
  For $K(\mathcal{A})^{bw}$ use additionally
  Lemma~\ref{l:weight-str-basic-properties}, \eqref{enum:bounded-w-decomp}.
\end{proof}

We continue to use the maps $\ul{w}_{\leq n}$ and $\ul{w}_{\geq n+1}$
on complexes
introduced in the above proof.
Define
$\ul{w}_{>n}:=\ul{w}_{\geq n+1}$, 
$\ul{w}_{<n}:=\ul{w}_{\leq n-1}$,
$\ul{w}_{[a,b]}=\ul{w}_{\geq a}\ul{w}_{\leq b}=\ul{w}_{\leq
  b}\ul{w}_{\geq a}$, and $\ul{w}_a=\ul{w}_{[a,a]}$.
The triangle \eqref{eq:hot-wdecomp} 
will be called the
\define{$\ul{w}$-weight decomposition of $X$}.

\begin{remark}
  Note that $\ul{w}_{\geq n}$ and $\ul{w}_{\leq n}$ are functorial
  on $C(\mathcal{A})$ but not at all on $K(\mathcal{A})$
  (if $\mathcal{A} \not= 0$):
  Take an object $M \in \mathcal{A}$ and consider its mapping cone
  $\Cone(\id_M)$.
  All objects $\Cone(\id_M)$ (for $M \in \mathcal{A}$) are isomorphic
  to the zero 
  object in $K(\mathcal{A})$. If $\ul{w}_{\leq 0}$ were functorial, all
  $\ul{w}_{\leq 0}(\Cone(\id_M))=M$ were isomorphic, hence $\mathcal{A}=0$. 
\end{remark}

\begin{remark}
  \label{rem:t-structures-adjacent-to-standard-w-structure}
  Assume that $\mathcal{A}$ is an abelian category and consider the
  following four subcategories of $K(\mathcal{A})$:

  \begin{equation*}
    (\mathcal{C}^{h \leq 0}, K(\mathcal{A})^{w \geq 1},
    K(\mathcal{A})^{w \leq 0}, \mathcal{K}^{h \geq 1}).
  \end{equation*}
  The two outer pairs are torsion pairs on $K(\mathcal{A})$
  by Lemma~\ref{l:torsion-pair-homotopy-abelian}; the pair in the
  middle is (up two a swap of the two members and a translation) the
  standard w-structure on $K(\mathcal{A})$ 
  from Proposition~\ref{p:ws-hot-additive}.
  In any pair of direct neighbors there are no morphisms from left
  to right; more precisely the left member is the left
  orthogonal of the right member and vice versa.
  
  In the terminology of \cite[Def~4.4.1]{bondarko-weight-str-vs-t-str},
  the t-structure 
  $(K(\mathcal{A})^{w \leq 0}, \mathcal{K}^{h \geq 0})$ 
  and the (standard) w-structure
  $(K(\mathcal{A})^{w \leq 0}, K(\mathcal{A})^{w \geq 0})$
  on $K(\mathcal{A})$ are left adjacent (i.\,e.\ their left aisles coincide)
  to each other.
  Similarly, the t-structure 
  $(\mathcal{C}^{h \leq 0}, K(\mathcal{A})^{w \geq 0})$
  and the (standard) w-structure
  $(K(\mathcal{A})^{w \leq 0}, K(\mathcal{A})^{w \geq 0})$
  on $K(\mathcal{A})$ are right adjacent to each other.
\end{remark}

\begin{lemma}
  \label{l:bd-complex-dir-summand}
  Let $X \in K(\mathcal{A})$ and $a,b \in \DZ$.
  If $X \in K(\mathcal{A})^{w \in [a,b]}$ then 
  $X$ is a summand of $\ul{w}_{[a,b]}(X)$ in $K(\mathcal{A})$.
\end{lemma}

\begin{proof}
  Let $X \in K(\mathcal{A})^{w \in [a,b]}$. If $a>b$ then
  $X\cong 0=\ul{w}_{[a,b]}(X)$. Assume $a \leq b$.
  Lemma~\ref{l:weight-str-basic-properties}~\eqref{enum:weight-decomp-with-knowlegde-dir-summand-i} 
  gives $\ul{w}_{\leq b}(X) \cong X \oplus [1]\ul{w}_{> b}(X)$. Observe that 
  $[1]\ul{w}_{>b}(X) \in K(\mathcal{A})^{w \geq b} \subset
  K(\mathcal{A})^{w \geq a}$. 
  Hence $\ul{w}_{\leq b}(X) \in K(\mathcal{A})^{w \geq a}$.  
  Now 
  Lemma~\ref{l:weight-str-basic-properties}~\eqref{enum:weight-decomp-with-knowlegde-dir-summand-ii}
  shows that $\ul{w}_{[a,b]}(X)\cong \ul{w}_{\leq b}(X) \oplus [-1]
  \ul{w}_{< a}(X)$.
\end{proof}

We view $\mathcal{A}$ as a full subcategory of $K(\mathcal{A})$,
namely given by all complexes concentrated in degree zero.

\begin{corollary}
  \label{c:heart-standard-wstr}
  The heart $\heart$ of the standard weight structure on $K(\mathcal{A})$
  is the idempotent completion of $\mathcal{A}$.

  The same statement is true for the standard weight structure on
  $K(\mathcal{A})^*$ for $* \in \{+,-,bw\}$
  (and for that on 
  $K^*(\mathcal{A})$ for $* \in \{+,-\}$).

  The heart of the standard weight structure on $K(\mathcal{A})^b$
  (resp.\ $K^b(\mathcal{A})$) is 
  the closure under retracts of $\mathcal{A}$ in $K(\mathcal{A})^b$
  (resp.\ $K^b(\mathcal{A})$). 
\end{corollary}

\begin{proof}
  Since $K(\mathcal{A})^{bw}$ is
  idempotent complete
  (Thm.~\ref{t:one-side-bounded-hot-idempotent-complete})
  and the heart $\heart$ is contained in $K(\mathcal{A})^{bw}$
  and closed under retracts in $K(\mathcal{A})$ it follows that
  $\heart$ is idempotent complete.
  Furthermore any object $X \in \heart$ is a summand of
  $X^0=\ul{w}_{0}(X) \in \mathcal{A}$ by Lemma~\ref{l:bd-complex-dir-summand}.
  Since $\mathcal{A}$ is a full subcategory of $\heart$ the claim
  follows from the results at the end of 
  Section~\ref{sec:idemp-completeness}. 
  The proof of the second statement is similar.

  For the proof of the last statement let $\heart$ be the heart of the
  standard weight structure on $\mathcal{T}=K(\mathcal{A})^b$
  (resp.\ $\mathcal{T}=K^b(\mathcal{A})$).
  Since $\mathcal{A} \subset \heart$ and the heart $\heart$ is
  closed under retracts in $\mathcal{T}$,
  any retract of an object of $\mathcal{A}$ in $\mathcal{T}$
  is in $\heart$. Conversely, any object of $\heart$ is
  by Lemma~\ref{l:bd-complex-dir-summand} a retract of an object of
  $\mathcal{A}$.
\end{proof}

\begin{proposition}
  \label{p:grothendieck-homotopy-category}
  Let $\mathcal{A}$ be a small additive category. Then
  the Grothendieck group of its homotopy category $K(\mathcal{A})$ is
  trivial: 
  \begin{equation*}
    K_0(K(\mathcal{A}))=0.
  \end{equation*}
  The Grothendieck groups of $K^-(\mathcal{A})$, $K(\mathcal{A})^{-}$,
  $K^+(\mathcal{A})$ and $K(\mathcal{A})^{+}$ vanish as well. 
\end{proposition}

This result is presumably well-known to the experts 
(cf.~\cite{schlichting-neg-K-theory-dercat} or
\cite{miyachi-grothendieck}).

\begin{proof}
  We prove that $K_0(K(\mathcal{A}))=0$.
  We write $[{X}]$ for the class of an object $X$ in the Grothendieck group.
  Let $X \in K(\mathcal{A})$. Let $A=\ul{w}_{>0}(X)$ and
  $B=\ul{w}_{\leq 0}(X)$. The $\ul{w}$-weight decomposition 
  $A \ra X \ra B \ra [1]A$
  gives $[X]=[A]+[B]$ in the
  Grothendieck group. Since $B$ is a bounded above complex
  \begin{equation*}
    T(B):= B \oplus [2] B \oplus [4]B \oplus \dots = \bigoplus_{n \in \DN}[2n]B
  \end{equation*}
  is a well-defined object of $K(\mathcal{A})$.
  There is an obvious isomorphism $T(B) =B \oplus
  [2]T(B)$ which gives
  \begin{equation*}
    [{T(B)}]=[{B}]+[{[2]T(B)}]=[{B}]+[{T(B)}]
  \end{equation*}
  implying $[{B}]=0$. Considering
  $\bigoplus_{n \in \DN}[-2n]A$
  we similarly find $[A]=0$. Hence $[X]=[A]+[B]=0$.

  The proof that the other Grothendieck groups mentioned in the
  proposition vanish is now obvious.
\end{proof}

\begin{remark}
  Assume that we are in the setting of 
  Proposition~\ref{p:grothendieck-homotopy-category}.
  If we can show that the Grothendieck group of the idempotent
  completion of $K(\mathcal{A})$ vanishes then $K(\mathcal{A})$ is
  idempotent complete by the results cited in
  Remark~\ref{rem:thomason-and-balmer-schlichting}.
\end{remark}

\begin{example}
  Let $\{0\} \subsetneq \Lambda \subset \DN$ be a subset 
  that is
  closed under addition, e.\,g. $\Lambda= 17 \DN+9\DN$.
  Let $\modfin(k)$ be the category of finite dimensional vector spaces
  over a field $k$ and let 
  $\modfin_{\Lambda}(k) \subset \modfin(k)$ be the full subcategory of vector
  spaces whose dimension is in $\Lambda$. 
  We claim that $K(\modfin_{\Lambda}(k))$ is idempotent complete.

  Obviously $K(\modfin_{\Lambda}(k)) \subset
  K(\modfin(k))$ is a full triangulated subcategory. 
  It is easy to see that any object of $K(\modfin(k))$ is isomorphic to a
  stalk complex, i.\,e.\ a complex in $\modfin(k)$ with all differentials
  $d=0$. Hence $K(\modfin(k))$ is idempotent complete 
  (alternatively, this follows from
  Theorem~\ref{t:hot-idempotent-complete}~\eqref{enum:hot-idempotent-complete-abelian})   
  and $K(\modfin_{\Lambda}(k))$ is dense in $K(\modfin(k))$.
  In particular $K(\modfin(k))$ is the idempotent completion of
  $K(\modfin_{\Lambda}(k))$.

  The Grothendieck groups of $K(\modfin_{\Lambda}(k))$ and $K(\modfin(k))$ vanish
  by Proposition~\ref{p:grothendieck-homotopy-category}.
  Hence results of R.\ W.\ Thomason \cite[Section~3]{thomason-class} (as
  explained in Remark~\ref{rem:thomason-and-balmer-schlichting}) show
  that the closure of $K(\modfin_{\Lambda}(k))$ under isomorphisms in
  $K(\modfin(k))$ equals $K(\modfin(k))$. In particular
  $K(\modfin_{\Lambda}(k))$ is 
  idempotent complete.
\end{example}

\section{Weak weight complex functor}
\label{sec:weak-wc-fun}

In this section we construct the weak weight complex
functor essentially following
\cite[Ch.~3]{bondarko-weight-str-vs-t-str} where it is just called
weight complex functor.
We repeat the construction in detail since we need this
for the proof of Theorem~\ref{t:strong-weight-cplx-functor}.

Let $(\mathcal{T}, w=(\mathcal{T}^{w \leq 0}, \mathcal{T}^{w \geq
  0}))$ be a weight category. 
We fix for every object $X$ in
$\mathcal{T}$ and every $n \in \DZ$ a weight decomposition
\begin{equation}
  \label{eq:choice-weight-decomp}
  \xymatrix{
    {T^n_X:} &
    {w_{\geq n+1}X} \ar[r]^-{g^{n+1}_X} & 
    {X} \ar[r]^-{k_X^n} &
    {w_{\leq n}X} \ar[r]^-{v_X^n} &
    {[1]w_{\geq n+1}X;}
  }
\end{equation}
as suggested by the notation we assume
$w_{\geq n+1}X \in \mathcal{T}^{w \geq n+1}$ and $w_{\leq n} \in
\mathcal{T}^{w \leq n}$.  
For any $n$ there is a unique morphism of triangles 
\begin{equation}
  \label{eq:unique-triang-morph-Tnx}
  \xymatrix{
    {T^n_X:} \ar[d]|{(h^n_X,\id_X,l^n_X)} &
    {w_{\geq n+1}X} \ar[r]^-{g^{n+1}_X} \ar[d]^-{h^n_X}
    \ar@{}[rd]|{\triangle} &  
    {X} \ar[r]^-{k^n_X} \ar[d]^-{\id_X} 
    \ar@{}[rd]|{\nabla} &
    {w_{\leq n}X} \ar[r]^-{v^n_X} \ar[d]^-{l^n_X} &
    {[1]w_{\geq n+1}X} \ar[d]^-{[1]h^n_X} \\
    {T^{n-1}_X:} &
    {w_{\geq n}X} \ar[r]^-{g^n_X} & 
    {X} \ar[r]^-{k^{n-1}_X} &
    {w_{\leq n-1}X} \ar[r]^-{v^{n-1}_X} &
    {[1]w_{\geq n}X}
  }  
\end{equation}
extending $\id_X$
(use Prop.~\ref{p:BBD-1-1-9-copied-for-w-str} and
\ref{enum:ws-iii}).
More precisely $h^n_X$ (resp.\ $l^n_X$) is the unique morphism making the square
$\Delta$ (resp.\ $\nabla$) commutative.
We use the square marked with $\triangle$ as the germ cell for the
octahedral axiom and obtain the following diagram:
\begin{equation}
  \label{eq:wc-weak-octaeder}
  O_X^n:=
  \xymatrix@dr{
    && {[1]{w_{\geq n+1}X}} \ar[r]^-{[1]h^n_X} & {[1]{w_{\geq n}X}}
    \ar[r]^-{[1]e^n_X} & {[1]{w_n X}} 
    \\ 
    {T^{\leq n}_X:} &{w_n X} \ar@(ur,ul)[ru]^-{c^n_X}
    \ar@{..>}[r]^-{a^n_X} & 
    {w_{\leq n} X} \ar@{..>}[r]^-{l^n_X} \ar[u]^-{v^n_X} & 
    {w_{\leq n-1} X} 
    \ar@(dr,dl)[ru]^-{b^n_X} \ar[u]_{v^{n-1}_X}\\ 
    {T^{n-1}_X:} &{w_{\geq n}X} \ar[u]^-{e^n_X} \ar[r]^-{g^n_X} 
    & {X} \ar@(dr,dl)[ru]^-{k^{n-1}_X} \ar[u]_{k^n_X} \\  
    {T^n_X:} &{w_{\geq n+1}X} \ar[u]^-{h^n_X} \ar@(dr,dl)[ru]^-{g^{n+1}_X}
    \ar@{}[ru]|{\triangle} \\
    & {T^{\geq n}_X:}
  }
\end{equation}
The octahedral axiom says that after fitting $h^n_X$ into
the triangle $T^{\geq n}_X$ the dotted arrows exist such
that $T^{\leq n}_X$ is a triangle and everything commutes. The lower
dotted arrow is in fact $l^n_X$ by the uniqueness statement given
above.
We fix such octahedral diagrams $O^n_X$ for all objects $X$ and all $n
\in \DZ$. 

The triangles
$T^{\leq n}_X$ and $T^{\geq n}_X$ 
and the fact that
$\mathcal{T}^{w \leq n}$
and $\mathcal{T}^{w \geq n}$ are closed under extensions 
show that $w_n X
\in \mathcal{T}^{w=n}$ and $[n]w_n X \in \heart:=\heart(w)$.

We define the (candidate) weight complex $\candidateWCweak(X) \in C(\heart)$ of
$X$
as follows (the index $c$ stands for ``candidate"): Its $n$-th term is 
\begin{equation*}
  \candidateWCweak(X)^n:=[n]w_n X  
\end{equation*}
and the
differential $d^n_{\candidateWCweak(X)}:[n]w_n X \ra [n+1]w_{n+1}X$ is defined by
\begin{multline}
  \label{eq:differential-weak-WC-complex}
  d^n_{\candidateWCweak(X)}
  := [n](b^{n+1}_X \comp a^n_X)
  =  [n](([1]e^{n+1}_X) \comp v^n_X \comp a^n_X)\\
  =  [n](([1]e^{n+1}_X) \comp c^n_X).
\end{multline}
Note that $d^n_{\candidateWCweak(X)} \comp d^{n-1}_{\candidateWCweak(X)}=0$ 
since 
the composition of two consecutive
maps in a triangle is zero (apply this to \eqref{eq:wc-weak-octaeder}).
Hence $\candidateWCweak(X)$ is in fact a complex in $\heart$.

Now let $f: X \ra Y$ be a morphism in $\mathcal{T}$.
We can extend $f$ to a morphism of triangles 
(use Prop.~\ref{p:BBD-1-1-9-copied-for-w-str} and
\ref{enum:ws-iii}) 
\begin{equation}
  \label{eq:f-w-trunc}
  \xymatrix{
    {T^n_X:} \ar[d]_-{(f_{w \geq n+1}, f, f_{w \leq n})} &
    {w_{\geq n+1}X} \ar[r]^-{g^{n+1}_X} \ar[d]^-{f_{w \geq n+1}} & 
    {X} \ar[r]^-{k_X^n}  \ar[d]^-{f} &
    {w_{\leq n}X} \ar[r]^-{v_X^n} \ar[d]^-{f_{w \leq n}}&
    {[1]w_{\geq n+1}X,} \ar[d]^-{[1]f_{w \geq n+1}} \\
    {T^n_Y:} &
    {w_{\geq n+1}Y} \ar[r]^-{g^{n+1}_Y} & 
    {Y} \ar[r]^-{k_Y^n} &
    {w_{\leq n}Y} \ar[r]^-{v_Y^n} &
    {[1]w_{\geq n+1}Y.}
  }
\end{equation}
This extension is not unique in general; this will be discussed later
on. Nevertheless we fix now such an extension $(f_{w \geq n+1}, f,
f_{w \leq n})$ for any $n \in \DZ$.  

Consider the following diagram in the category of triangles
(objects: triangles; morphisms: morphisms of triangles):
\begin{equation}
  \label{eq:any-trunc-f-comm-hil}
  \xymatrix@C=2.5cm@R=1.5cm{
    {T^n_X} 
    \ar[d]_-{(h^n_X,\id_X,l^n_X)}
    \ar[r]^-{(f_{w \geq n+1}, f, f_{w \leq n})} & 
    {T^n_Y} \ar[d]^-{(h^n_Y,\id_Y,l^n_Y)} \\
    {T^{n-1}_X} 
    \ar[r]^-{(f_{w \geq n}, f, f_{w \leq n-1})} &
    {T^{n-1}_Y} \\
  }
\end{equation}
This square is commutative since a morphism of triangles $T^n_X \ra
T^{n-1}_Y$ extending $f$ is unique 
(use Prop.~\ref{p:BBD-1-1-9-copied-for-w-str} and \ref{enum:ws-iii}).

In particular we have $f_{w \leq n-1} l^n_X =l^n_Y f_{w \leq n}$, so
we can extend the partial morphism $(f_{w \leq n}, f_{w \leq n-1})$ 
by a morphism $f^n:w_nX \ra w_nY$ 
to a morphism of triangles
\begin{equation}
  \label{eq:f-w-trunc-fn}
  \xymatrix{
    {T^{\leq n}_X:} \ar[d]_-{(f^n, f_{w \leq n}, f_{w \leq n-1})} & 
    {w_n X} \ar[r]^-{a^n_X} \ar@{..>}[d]^-{f^n} &
    {w_{\leq n} X} \ar[r]^-{l^n_X} \ar[d]^-{f_{w\leq n}} & 
    {w_{\leq n-1} X} \ar[r]^-{b^n_X} \ar[d]^-{f_{w\leq
        n-1}} &  
    {[1]w_n X} \ar@{..>}[d]^-{[1]f^n}
    \\
    {T^{\leq n}_{Y}:} & 
    {w_n Y} \ar[r]^-{a^n_Y} &
    {w_{\leq n} Y} \ar[r]^-{l^n_Y} & 
    {w_{\leq n-1} Y} \ar[r]^-{b^n_Y} & 
    {[1]w_n Y}
  }
\end{equation}
as indicated.
Again there might be a choice, but we fix for each $n \in \DZ$ such an
$f^n$.  
The commutativity of the squares on the left and right in 
\eqref{eq:f-w-trunc-fn} shows that 
the sequence
$\candidateWCweak(f):=([n]f^n)_{n \in \DZ}$
defines a morphism of
complexes
\begin{equation}
  \label{eq:nearly-weakWC}
  \xymatrix{
    {\candidateWCweak(X):} \ar[d]_{\candidateWCweak(f)} 
    & {\dots} \ar[r] 
    & {[n]w_n X} \ar[r]^-{d^n_{\candidateWCweak(X)}} \ar[d]^{[n]f^n}
    & {[n+1]w_{n+1}X} \ar[r] \ar[d]^{[n+1]f^{n+1}}
    & {\dots}\\
    {\candidateWCweak(Y):}
    & {\dots} \ar[r] 
    & {[n]w_n Y} \ar[r]^-{d^n_{\candidateWCweak(Y)}}
    & {[n+1]w_{n+1}Y} \ar[r]
    & {\dots.}
  }
\end{equation}
Since some morphisms existed but were not unique we cannot expect that
$\candidateWCweak$ defines a functor $\mathcal{T} \ra C(\heart)$,
cf.\ Example~\ref{ex:weak-WC-weak-homotopy} below.

Let $\WCweak$ be the composition of $\candidateWCweak$ with the canonical
functor $C(\heart) \ra K_\weak(\heart)$
(cf.\ \eqref{eq:can-CKKweak}), i.\,e.\ the
assignment mapping an object $X$ of $\mathcal{T}$ to $\WCweak(X):=
\candidateWCweak(X)$ and a morphism $f$ of $\mathcal{T}$ to the class
of $\candidateWCweak(f)$ in $K_\weak(\heart)$. The complex 
$\WCweak(X)$ is called a \define{weight complex of $X$}.

Recall that the assignment $X \mapsto \WCweak(X)$ depends on the
choices made in \eqref{eq:choice-weight-decomp} and
\eqref{eq:wc-weak-octaeder}. For morphisms we have:

\begin{proposition}
  \label{p:dependence-on-choices-weakWCfun}
  Mapping a morphism $f$ in $\mathcal{T}$ to $\WCweak(f)$ does
  not depend on the choices made in \eqref{eq:f-w-trunc} and
  \eqref{eq:f-w-trunc-fn}.   
\end{proposition}

\begin{proof}
  By considering appropriate differences it is easy to see that it is
  sufficient to consider the case that $f=0$. 
  Consider \eqref{eq:f-w-trunc} now for $f=0$ 
  (but we write $f_{w \leq n}$ instead of 
  $0_{w \leq n}$).
  \begin{equation}
    \label{eq:f-w-trunc-null-morph}
    \xymatrix{
      {T^n_X:} \ar[d]_-{(f_{w \geq n+1}, 0, f_{w \leq n})} &
      {w_{\geq n+1}X} \ar[r]^-{g^{n+1}_X} \ar[d]^-{f_{w \geq n+1}} & 
      {X} \ar[r]^-{k_X^n}  \ar[d]^-{f=0} &
      {w_{\leq n}X} \ar[r]^-{v_X^n} \ar[d]^-{f_{w \leq n}}&
      {[1]w_{\geq n+1}X,} \ar[d]^-{[1]f_{w \geq n+1}} \\
      {T^n_Y:} &
      {w_{\geq n+1}Y} \ar[r]^-{g^{n+1}_Y} & 
      {Y} \ar[r]^-{k_Y^n} &
      {w_{\leq n}Y} \ar[r]^-{v_Y^n} &
      {[1]w_{\geq n+1}Y,}\\
    }
  \end{equation}
  Since $f_{w \leq n} \comp k^n_X=0$ there exists
  $s^n: [1]w_{\geq n+1}X \ra
  w_{\leq n}Y$ such that $f_{w \leq n}=s^n v^n_X$.
  Then in the situation
  \begin{equation*}
    \xymatrix@C=1.5cm{
      {[1]w_{\geq n+2}X} \ar[r]^-{[1] h^{n+1}_X}
      & {[1]w_{\geq {n+1}}X} \ar[r]^-{[1] e^{n+1}_X} \ar[d]_(.3){s^n}
      & {[1]w_{n+1}X} 
      \ar@{..>}[lld]^(.3){\tau^{n+1}}
      \\
      {w_n Y} \ar[r]_-{a^n_Y} 
      & {w_{\leq n}Y} \ar[r]_-{l^n_Y} 
      & {w_{\leq n-1}Y} 
    }
  \end{equation*}
  (where both rows are part of triangles in 
  \eqref{eq:wc-weak-octaeder}, the upper row comes up to signs
  from a rotation of $T^{\geq n+1}_{X}$, the lower row is from
  ${T^{\leq n}_{Y}}$)
  the indicated factorization 
  \begin{equation}
    \label{eq:sn-factorization}
    s^n=a^n_Y \tau^{n+1} ([1]e^{n+1}_X)
  \end{equation}
  exists by \ref{enum:ws-iii}. 
  Now \eqref{eq:f-w-trunc-fn} takes the form (the dotted diagonal arrow
  will be explained)
  \begin{equation*}
    \hspace{-1.4cm}
    \xymatrix@=1.8cm{
      {T^{\leq n}_X:} \ar[d]_-{(f^n, f_{w \leq n}, f_{w \leq n-1})} & 
      {w_n X} \ar[r]^-{a^n_X} \ar[d]^-{f^n} \ar@{}[rd]|(0.7){\nabla}&
      {w_{\leq n} X} \ar[r]^-{l^n_X} 
      \ar[d]^-{
        s^n v^n_X} 
      \ar@{..>}[dl]|-{\tau^{n+1} b^{n+1}_X} & 
      {w_{\leq n-1} X} \ar[r]^-{b^n_X} 
      \ar[d]^-{
        s^{n-1} v^{n-1}_X} &  
      {[1]w_n X} \ar[d]^-{[1]f^n}
      \\
      {T^{\leq n}_{Y}:} & 
      {w_n Y} \ar[r]^-{a^n_Y} &
      {w_{\leq n} Y} \ar[r]^-{l^n_Y} & 
      {w_{\leq n-1} Y} \ar[r]^-{b^n_Y} & 
      {[1]w_n Y}
    }
  \end{equation*}
  Equation \eqref{eq:sn-factorization} and $([1]e^{n+1}_X) v^n_X =
  b^{n+1}_X$ (which follows from the octahedron $O^{n+1}_X$, cf.\
  \eqref{eq:wc-weak-octaeder})
  yield
  \begin{equation*}
    s^nv^n_X=a^n_Y \tau^{n+1} ([1]e^{n+1}_X)v^n_X  =a^n_Y \tau^{n+1} b^{n+1}_X.
  \end{equation*}
  This shows that the honest triangle marked $\nabla$ 
  commutes. Consider $f^n-\tau^{n+1} b^{n+1}_X a^n_X$: Since 
  \begin{equation*}
    a^n_Y(f^n-\tau^{n+1} b^{n+1}_X a^n_X)
    =a^n_Y f^n - s^nv^n_X a^n_X=0
  \end{equation*}
  there is a morphism $\nu^n: w_n X \ra [-1]w_{\leq n-1}Y$ such that
  \begin{equation}
    \label{eq:factors-over-nu}
    -([-1]b^n_Y) \nu^n=f^n -\tau^{n+1} b^{n+1}_X a^n_X.  
  \end{equation}
  Now consider the following diagram
  \begin{equation*}
    \hspace{-0.4cm}
    \xymatrix{
      & {w_n X} \ar[d]^{\nu^n} \ar@{..>}[dl]_{\sigma^n}\\
      {[-1]w_{n-1}Y} \ar[r]_-{-[-1]a^{n-1}_Y} 
      & {[-1]w_{\leq {n-1}}Y} \ar[r]_-{-[-1]l^{n-1}_Y} 
      & {[-1]w_{\leq {n-2}}Y} \ar[r]_-{-[-1]b^{n-1}_Y}
      & {w_{n-1}Y}
    }
  \end{equation*}
  where the lower row is the triangle
  obtained from $T^{\leq n-1}_{Y}$ by three rotations.
  The composition $([-1]l^{n-1}_Y) \nu^n$ vanishes by 
  \ref{enum:ws-iii}; hence there is a morphism $\sigma^n: w_n X \ra
  [-1]w_{n-1}Y$ as 
  indicated such that
  $-([-1]a^{n-1}_Y) \sigma^n = \nu^n$. If we plug this into
  \eqref{eq:factors-over-nu} we get
  \begin{equation*}
    f^n- \tau^{n+1} b^{n+1}_X a^n_X=-([-1]b^n_Y) \nu^n =([-1]b^n_Y)
    ([-1]a^{n-1}_Y) \sigma^n. 
  \end{equation*}
  Applying $[n]$ to this equation yields
  (using \eqref{eq:differential-weak-WC-complex})
  \begin{equation}
    \label{eq:well-defined-in-weak-homotopy-cat}
    \candidateWCweak(f)^n=[n]f^n= ([n]\tau^{n+1})
    d^n_{\candidateWCweak(X)} + d^{n-1}_{\candidateWCweak(Y)}
    ([n]\sigma^n).  
  \end{equation}
  This shows that $\candidateWCweak(f)$ is weakly
  homotopic to zero proving the claim (since we assumed that $f=0$). 
\end{proof}

\begin{theorem}
  [{cf.\ \cite[Ch.~3]{bondarko-weight-str-vs-t-str}}]
  \label{t:weakWCfun}
  The assignment $X \mapsto \WCweak(X)$, $f \mapsto \WCweak(f)$,
  depends only on the choices made in
  \eqref{eq:choice-weight-decomp} and 
  \eqref{eq:wc-weak-octaeder}
  and defines an additive functor
  \begin{equation*}
    \WCweak: \mathcal{T} \ra K_\weak(\heart).
  \end{equation*}
  This functor is in a canonical way a functor of additive categories
  with translation:
  There is a canonical isomorphism of functors $\phi: \Sigma \comp \WCweak
  \sira \WCweak \comp [1]$ such that $(\WCweak, \phi)$ is a functor of
  additive categories with translation (the translation on
  $K_\weak(\heart)$ is denoted $\Sigma$ here for clarity).
\end{theorem}

\begin{proof}
  It is obvious that 
  $\WCweak(\id_X)=\id_{\WCweak(X)}$
  and
  $\WCweak(f \comp g)=\WCweak(f) \comp \WCweak(g)$.
  Hence $\WCweak$ is a well-defined functor which is obviously
  additive.
  
  We continue the proof 
  after the following Remark~\ref{rem:choices-weakWC}
\end{proof}

\begin{remark}
  \label{rem:choices-weakWC}
  Let $\WCweak_1:=\WCweak$ be the additive functor from 
  Theorem~\ref{t:weakWCfun} (we do not know yet how it is compatible
  with the respective translations) and let $\WCweak_2$ be another
  such 
  functor constructed from possibly different choices 
  in 
  \eqref{eq:choice-weight-decomp} and 
  \eqref{eq:wc-weak-octaeder}.
  For each $X \in \mathcal{T}$ the identity $\id_X$ gives rise to a
  well-defined morphism $\psi_{21,X}: \WCweak_1(X) \ra \WCweak_2(X)$ in
  $K_\weak(\heart)$ which is constructed in the same manner as $\WCweak(f)$ was
  constructed from $f:X \ra Y$ above. The collection of these $\psi_{21,X}$
  in fact defines an isomorphism $\psi_{21}: \WCweak_1 \sira \WCweak_2$.
  If there is a third functor $\WCweak_3$ of the same type all these
  isomorphisms are compatible in the sense that
  $\psi_{32}\psi_{21}=\psi_{31}$ and $\psi_{ii}=\id_{\WCweak_i}$ for
  all $i \in \{1,2,3\}$.
\end{remark}

\begin{proof}[Proof of Thm.~\ref{t:weakWCfun} continued:]
  Our aim is to define $\phi$. 
  Let $U^{n}_{[1]X}$ be the triangle obtained by three rotations from
  the triangle $T^{n+1}_X$ (see \eqref{eq:choice-weight-decomp}):
  \begin{equation}
    \label{eq:choice-weight-decomp-translation}
    \xymatrix{
      {U^{n}_{[1]X}:} &
      {[1]w_{\geq n+2}X} \ar[r]^-{-[1]g^{n+2}_X} & 
      {[1]X} \ar[r]^-{-[1]k_X^{n+1}} &
      {[1]w_{\leq n+1}X} \ar[r]^-{-[1]v_X^{n+1}} &
      {[2]w_{\geq n+2}X;}
    }
  \end{equation}
  Note that it is a 
  $(w \geq n+1, w \leq n)$-weight decomposition of $[1]X$.
  Using \eqref{eq:wc-weak-octaeder} it is easy to check that 
  \begin{equation}
    \label{eq:wc-weak-octaeder-translation}
    \hspace{-1.0cm}
    \xymatrix@=1.2cm@dr{
      && {[2]{w_{\geq n+2}X}} \ar[r]^-{[2]h^{n+1}_X} & {[2]{w_{\geq {n+1}}X}}
      \ar[r]^-{[2]e^{n+1}_X} & {[2]{w_{n+1} X}} 
      \\ 
      {U^{\leq n}_{[1]X}:} &{[1]w_{n+1} X} \ar@(ur,ul)[ru]^-{-[1]c^{n+1}_X}
      \ar[r]_-{[1]a^{n+1}_X} & 
      {[1]w_{\leq {n+1}} X} \ar[r]^-{[1]l^{n+1}_X} \ar[u]^-{-[1]v^{n+1}_X} & 
      {[1]w_{\leq n} X} 
      \ar@(dr,dl)[ru]^-{-[1]b^{n+1}_X} \ar[u]_{-[1]v^{n}_X}\\ 
      {U^{n-1}_{[1]X}:} & [1]{w_{\geq {n+1}}X} \ar[u]^-{[1]e^{n+1}_X}
      \ar[r]^-{-[1]g^{n+1}_X}  
      & {[1]X} \ar@(dr,dl)[ru]^-{-[1]k^{n}_X} \ar[u]_{-[1]k^{n+1}_X} \\  
      {U^{n}_{[1]X}:} &{[1]w_{\geq n+2}X}
      \ar[u]^-{[1]h^{n+1}_X} \ar@(dr,dl)[ru]^-{-[1]g^{n+2}_X} 
      \\
      & {U^{\geq n}_{[1]X}:}
    }
  \end{equation}
  is an octahedron.
  In the same manner in which the choices
  \eqref{eq:choice-weight-decomp} and 
  \eqref{eq:wc-weak-octaeder} gave rise to the functor $\WCweak:
  \mathcal{T} \ra K_\weak(\heart)$, the choices 
  \eqref{eq:choice-weight-decomp-translation} and
  \eqref{eq:wc-weak-octaeder-translation} 
  give rise to an additive functor 
  $\WCweak': \mathcal{T} \ra K_\weak(\heart)$.
  As seen in Remark~\ref{rem:choices-weakWC} there is a canonical
  isomorphism $\psi:\WCweak' \sira \WCweak$.
  
  We have
  \begin{equation*}
    \WCweak'([1]X)^n=[n][1]w_{n+1}X=\WCweak(X)^{n+1}=\Sigma(\WCweak(X))^{n}  
  \end{equation*}
  and
  \begin{multline*}
    d^{n}_{\WCweak'([1]X)}=[n]((-[1]b^{n+2}_X) \comp [1]a^{n+1}_X)
    =-[n+1](b^{n+2}_X \comp a^{n+1}_X) \\
    = -d^{n+1}_{\WCweak(X)}=
    d^n_{\Sigma(\WCweak(X))}.
  \end{multline*}
  This implies that $\Sigma(\WCweak(X))= \WCweak'([1]X)$
  and it is easy to see that even $\Sigma \comp \WCweak =\WCweak' \comp [1]$.
  Now define $\phi$ as the
  composition 
  \begin{equation*}
    \Sigma \comp \WCweak =\WCweak' \comp [1] \xsira{\psi \comp [1]} \WCweak
    \comp [1].     
  \end{equation*}
\end{proof}

\begin{definition}
  The functor $\WCweak$ (together with $\phi$) of additive categories
  with translation from
  Theorem~\ref{t:weakWCfun}
  is called a \define{weak weight complex functor}.
\end{definition}

A weak weight complex functor depends on the choices
made in \eqref{eq:choice-weight-decomp} and
\eqref{eq:wc-weak-octaeder}.
However it follows from 
the proof of the Theorem~\ref{t:weakWCfun}
(and Remark~\ref{rem:choices-weakWC}) that any two weak weight complex
functors are canonically isomorphic. Hence we allow ourselves to speak
about \emph{the} weak weight complex functor.

\begin{example}
  [{cf.\ \cite[Rem.~1.5.2.3]{bondarko-weight-str-vs-t-str}}]
  \label{ex:weak-WC-weak-homotopy}
  Let $\Mod(R)$ be the category of
  $R$-modules for $R=\DZ/4\DZ$ and consider the standard weight structure on
  $K(\Mod(R))$
  (see Prop.~\ref{p:ws-hot-additive}). Let $X$ be the complex
  $\dots \ra 0 \ra R 
  \xra{\cdot 2} R \ra 0 \ra 
  \dots$ with $R$ in degrees $0$ and $1$.
  If we use the $\ul{w}$-weight decompositions from the proof of
  Proposition~\ref{p:ws-hot-additive}, the only interesting weight
  decomposition is $T^0_X$ of type $(w \geq 1, w \leq 0)$ and has the
  following form (where we draw the complexes vertically and give only
  their components in degrees 0 and 1, and similar for the morphisms):
  \begin{equation*}
    \xymatrix{
      {\ul{w}_{\geq 1}X} \ar[r]^{\svek 10} &
      {X} \ar[r]^{\svek 01} &
      {\ul{w}_{\leq 0}X} \ar[r]^{\svek 0{\cdot(-2)}} &
      {[1]\ul{w}_{\geq 1}X}\\
      {R} \ar[r]^1 &
      {R} \ar[r]^0 &
      {0} \ar[r]^0 &
      {0}\\
      {0} \ar[r]^0 \ar[u] &
      {R} \ar[r]^1 \ar[u]^{\cdot 2}&
      {R} \ar[r]^{\cdot(-2)=\cdot 2} \ar[u] &
      {R.} \ar[u]
    }
  \end{equation*}
  We can choose $w_1X= w_{\geq 1}X$ and $w_0X=w_{\leq 0}X$ and then
  one checks that 
  $\WCweak(X)$ is given by the connecting morphism of this triangle,
  more precisely
  \begin{equation*}
    \WCweak(X) = (\dots \ra 0 \ra R \xra{\cdot 2} R \ra 0 \ra \dots)
  \end{equation*}
  concentrated in degrees 0 and 1.
  Consider the morphism $0=0_X=\svek 00: X \ra X$ and extend it to a
  morphism of triangles
  (cf.\ \eqref{eq:f-w-trunc} or \eqref{eq:f-w-trunc-null-morph})
  \begin{equation*}
    \xymatrix{
      {w_{\geq 1}X} \ar[r] \ar@{..>}[d]^{\svek y0} &
      {X} \ar[r] \ar[d]^{\svek 00} &
      {w_{\leq 0}X} \ar[r] \ar@{..>}[d]^{\svek 0x} &
      {[1]w_{\geq 1}X} \ar@{..>}[d]^{[1]\svek y0 =\svek 0{y}} \\
      {w_{\geq 1}X} \ar[r] &
      {X} \ar[r] &
      {w_{\leq 0}X} \ar[r] &
      {[1]w_{\geq 1}X.}     
    }
  \end{equation*}
  It is an easy exercise to check that the dotted arrows complete $0_X$
  to a morphism of triangles
  for any $x,y \in \{0,2\}$. Now one checks that all four morphism $(0,0),
  (0,2), (2,0), (2,2): \WCweak(X) \ra \WCweak(X)$ are weakly
  homotopic whereas for example $(0,0)$ and $(2,0)$ are not
  homotopic. 

  In particular, this example confirms that mapping an object $X$ to
  $\candidateWCweak(X)$ and a morphism 
  $f$ to $\candidateWCweak(f)$ (or its class in $K(\heart)$) is not a
  well-defined functor: We have to pass to the weak homotopy category.

  (But one can easily find a preferred choice for the morphisms $f^n$
  in this example
  which defines a functor to $K(\heart)$, see Section
  \ref{sec:strong-WC-standard-wstr} below.)
\end{example}

\begin{lemma}
  [{cf.\ \cite[Thm.~3.3.1.IV]{bondarko-weight-str-vs-t-str}}]
  \label{l:WCweak-w-exact}
  Let $a,b \in \DZ$. If 
  $X \in \mathcal{T}^{w \geq a}$ 
  (resp.\ $X \in \mathcal{T}^{w \leq b}$
  or $X \in \mathcal{T}^{w \in [a,b]}$) 
  then
  $\WCweak(X) \in K(\heart)^{w \geq a}$ 
  (resp.\ $\WCweak(X) \in K(\heart)^{w \leq b}$
  or $\WCweak(X) \in K(\heart)^{w \in [a,b]}$).

  In particular, if the weight structure is bounded, then $\WCweak(X) \in
  K(\heart)^b$ for all $X \in \mathcal{T}$.
\end{lemma}

\begin{proof}
  Let $X \in \mathcal{S}$ where $\mathcal{S}$ is one of the categories 
  $\mathcal{T}^{w \geq a}$, $\mathcal{T}^{w \leq b}$, 
  $\mathcal{T}^{w \in [a,b]}$.
  Lemma~\ref{l:weight-str-basic-properties}
  \eqref{enum:bounded-w-decomp} shows that we can assume that in all
  weight decompositions $T^n_X$ 
  (see \eqref{eq:choice-weight-decomp})
  of $X$ the objects $w_{\geq n+1}X$ and $w_{\leq n}X$ are in
  $\mathcal{S}$.
  Choose octahedra \eqref{eq:wc-weak-octaeder} and let
  $\WCweak'(X)$ be constructed using these choices.

  Consider the octahedron
  \eqref{eq:wc-weak-octaeder} again.
  We claim that $w_nX \in \mathcal{S}$.
  We already know that $w_nX \in \mathcal{T}^{w=n}$. In particular the
  triangle $T_X^{\geq n}$ is a $(w\geq n+1, w \leq n)$-weight
  decomposition of $w_{\geq n}X$ and the triangle $T_X^{\leq n}$ is a
  $(w \geq n, w \leq n-1)$-weight decomposition of $w_{\leq n}X$.
  We obtain
  \begin{itemize}
  \item Case $\mathcal{S}= \mathcal{T}^{w \geq a}$:
    If $a \leq n$ then the weight decomposition $T_X^{\geq n}$ and
    Lemma~\ref{l:weight-str-basic-properties} \eqref{enum:weights-bounded}
    show $w_nX \in \mathcal{T}^{w \geq a}=\mathcal{S}$.
    If $a > n$ the triangle $T_X^{\leq n}$ shows that $w_nX$ is an
    extension of $w_{\leq n}X \in \mathcal{T}^{w \geq a}$ and
    $[-1]w_{\leq n-1}X \in \mathcal{T}^{w \geq a+1} \subset
    \mathcal{T}^{w \geq a}$ and 
    hence in $\mathcal{T}^{w \geq a}=\mathcal{S}$.
  \item Case $\mathcal{S}= \mathcal{T}^{w \leq b}$:
    If $n-1 < b$ the weight decomposition $T_X^{\leq n}$ and
    Lemma~\ref{l:weight-str-basic-properties} \eqref{enum:weights-bounded}
    show $w_nX \in \mathcal{T}^{w \leq b}=\mathcal{S}$.
    If $n-1 \geq b$ the triangle $T_X^{\geq n}$ shows that $w_nX$ is an
    extension of $w_{\geq n}X \in \mathcal{T}^{w \leq b}$ and
    $[1]w_{\geq n+1}X \in \mathcal{T}^{w \leq b-1} \subset
    \mathcal{T}^{w \leq b}$ and 
    hence in $\mathcal{T}^{w \leq b}=\mathcal{S}$.
  \item Case $\mathcal{S}= \mathcal{T}^{w \in [a, b]}$: Follows from
    the above two cases.
  \end{itemize}
  This proves the claim $w_nX \in \mathcal{S}$.
  Let $I$ be the one among the intervals $[a,\infty)$, $(-\infty,b]$,
  $[a,b]$ that satisfies $\mathcal{S}=\mathcal{T}^{w \in I}$.
  If $n \not\in I$ then
  $w_nX \in \mathcal{T}^{w \in I} \cap \mathcal{T}^{w=n}=0$ and hence
  $\WCweak'(X)^n=[n]w_nX=0$. This shows 
  $\WCweak'(X) \in K^{w \in I}(\heart)$ and
  $\WCweak(X) \in  K(\heart)^{w \in I}$ since $\WCweak'(X) \cong
  \WCweak(X)$. (Here the categories $K^{w \in I}(\heart)$ and
  $K(\heart)^{w \in I}$ are defined in the obvious way, cf.\ beginning
  of Section~\ref{sec:hot-cat-idem-complete}).
\end{proof}

In the following definition the triangulated category
$K(\heart(w))^\anti$ appears (see Section
\ref{sec:triang-categ} for the definition of $\mathcal{T}^\anti$ for
a triangulated category $\mathcal{T}$). 
This happens naturally as can be seen from
Proposition~\ref{p:strong-WC-for-standard-wstr} below. 
Let us however remark that we could avoid its appearance by 
replacing the definition of a weak weight complex functor $\WCweak$
above with its composition with the functor induced by $(S, \id)$
(see \eqref{eq:KA-antiKA-triequi})
which just changes the signs of all differentials.

\begin{definition}
  [{cf.~\cite[Conj.~3.3.3]{bondarko-weight-str-vs-t-str}}]
  \label{d:strong-WCfun}
  A \define{strong weight complex functor} is a 
  \emph{triangulated} 
  functor $\WCfun:\mathcal{T} \ra K(\heart)^\anti$
  such that the obvious composition
   \begin{equation*}
     {\mathcal{T}} \xra{\WCfun} {K(\heart)^\anti} \ra {K_\weak(\heart)}
   \end{equation*}
   is isomorphic to the/a weak weight complex functor
   as a functor of additive categories with translation.
\end{definition}

Recall the standard weight structure on $K(\heart)^\anti$ from
Remark~\ref{rem:ws-hot-additive-anti}.

\begin{lemma}
  \label{l:WCstrong-w-exact}  
  Any strong weight complex functor $\WCfun:\mathcal{T} \ra
  K(\heart)^\anti$
  is weight-exact.
\end{lemma}

\begin{proof}
  This follows immediately from Lemma~\ref{l:WCweak-w-exact}.
\end{proof}

\subsection{Strong weight complex functor for the standard weight
  structure} 
\label{sec:strong-WC-standard-wstr}

Consider the standard weight structure $w$ on the homotopy category
$K(\mathcal{A})$ of an additive category $\mathcal{A}$
from Proposition~\ref{p:ws-hot-additive}. 
Given $X \in K(\mathcal{A})$ the $\ul{w}$-weight decomposition 
\eqref{eq:hot-wdecomp} is a preferred choice for the
weight decomposition $T_X^n$ in \eqref{eq:choice-weight-decomp}.
Then there is also an obvious preferred choice for the octahedron 
$O_X^n$ in \eqref{eq:wc-weak-octaeder} in which $w_n X$ is just
$[-n]X^n$, the
$n$-th term $X^n$ of the complex $X$ shifted into degree $n$. 
With this choices the complex
$\candidateWCweak(X)=\WCweak(X)$ is 
obtained from $X$ by multiplying all differentials by $-1$, i.\,e.\
$\candidateWCweak(X) = S(X)$ where $S$ is the functor defined in
Section~\ref{sec:homotopy-categories};  
here we view 
$\mathcal{A} \subset \heart(w)$ as a full
subcategory (see Cor.~\ref{c:heart-standard-wstr} for a more
precise statement).

Similarly there are preferred choices for morphisms: Let 
$f: X \ra Y$ be a morphism in $K(\mathcal{A})$. Let $\hat{f}: X \ra Y$
be a morphism in $C(\mathcal{A})$ representing $f$. 
The morphisms $\ol{w}_{\geq n+1}\hat{f}$,
$\ol{w}_{\leq n}\hat{f}$ and $\ol{w}_{n}\hat{f}$ 
gives rise to preferred choices for the morphisms $f_{w \geq n+1}$,
$f_{w\leq n}$ and $f^n$ in 
\eqref{eq:f-w-trunc} and \eqref{eq:f-w-trunc-fn}. 
If we identify $\mathcal{A} \subset \heart(w)$ as above the morphism
$\candidateWCweak(f)$ (see \eqref{eq:nearly-weakWC}) of complexes 
is just $\hat{f}=S(\hat{f}):\candidateWCweak(X)=S(X)
\ra \candidateWCweak(Y)=S(Y)$. Obviously its class in the homotopy
category 
is $f=S(f)$ and hence
does not depend on 
the choice of the representative for $f$. 

\begin{proposition}
  \label{p:strong-WC-for-standard-wstr}
  The composition 
  \begin{equation*}
    K(\mathcal{A}) \xsira{(S,\id)}
    K(\mathcal{A})^\anti \subset K(\heart(w))^\anti
  \end{equation*}
  of the triangulated equivalence $(S, \id)$
  (cf.~\eqref{eq:KA-antiKA-triequi}) and the obvious inclusion is a
  strong weight complex functor.
\end{proposition}

\begin{proof}
  Clear from the above arguments.
\end{proof}

\section{Filtered triangulated categories}
\label{sec:filt-tria-cat}

We very closely follow \cite[App.]{Beilinson}. Let us recall the
definition of a filtered triangulated category.  

\begin{definition}
  \label{d:filt-tria}
  \begin{enumerate}
  \item 
    \label{enum:filt-tria-cat}
    A \define{filtered triangulated category}, or \define{f-cat\-e\-go\-ry}
    for short, is a quintuple
    $(\tildew{\mathcal{T}}, \tildew{\mathcal{T}}(\leq 0),
    \tildew{\mathcal{T}}(\geq 0), s, \alpha)$ 
    where $\tildew{\mathcal{T}}$ is a triangulated category,
    $\tildew{\mathcal{T}}(\leq 0)$ and $\tildew{\mathcal{T}}(\geq 0)$
    are strict full \emph{triangulated} subcategories,
    $s: \tildew{\mathcal{T}} \sira \tildew{\mathcal{T}}$ is a
    triangulated automorphism (called ``shift of filtration") and
    $\alpha: 
    \id_{\tildew{\mathcal{T}}} \ra s$ is a morphism of
    \emph{triangulated}
    functors, such that the 
    following axioms hold 
    (where 
    $\tildew{\mathcal{T}}(\leq n):= s^n(\tildew{\mathcal{T}}(\leq 0))$
    and 
    $\tildew{\mathcal{T}}(\geq n):= s^n(\tildew{\mathcal{T}}(\geq 0))$):
    \begin{enumerate}[label=(fcat{\arabic*})]
    \item 
      \label{enum:filt-tria-cat-shift}
      $\tildew{\mathcal{T}}(\geq 1) \subset \tildew{\mathcal{T}}(\geq 0)$ and
      $\tildew{\mathcal{T}}(\leq 0) \subset \tildew{\mathcal{T}}(\leq 1)$
    \item 
      \label{enum:filt-tria-cat-exhaust}
      $\tildew{\mathcal{T}}=\bigcup_{n \in \DZ} \tildew{\mathcal{T}}(\leq n) =
        \bigcup_{n \in \DZ} \tildew{\mathcal{T}}(\geq n)$.
    \item 
      \label{enum:filt-tria-cat-no-homs}
      $\Hom(\tildew{\mathcal{T}}(\geq 1), \tildew{\mathcal{T}}(\leq 0))=0$.
    \item 
      \label{enum:filt-tria-cat-triang}
      For any $X$ in $\tildew{\mathcal{T}}$ there is a 
      triangle
      \begin{equation*}
        A \ra X \ra B \ra A[1]
      \end{equation*}
      with $A$ in $\tildew{\mathcal{T}}(\geq 1)$ and $B$ 
      in $\tildew{\mathcal{T}}(\leq 0)$.
    \item 
      \label{enum:filt-tria-cat-shift-alpha}
      For any $X \in \tildew{\mathcal{T}}$ one has
      $\alpha_{s(X)}=s(\alpha_X)$ as morphisms
      $s(X) \ra s^2(X)$.
    \item 
      \label{enum:filt-tria-cat-hom-bij}
      For all $X$ in
      $\tildew{\mathcal{T}}(\leq 0)$
      and $Y$ in $\tildew{\mathcal{T}}(\geq 1)$, the map
      \begin{align*}
        \Hom(X,s\inv(Y)) & \sira \Hom(X,Y)\\
        f & \mapsto \alpha_{s\inv(Y)} \comp f
      \end{align*}
      is bijective (equivalently one can require that
      \begin{align*}
        \Hom(s(X), Y) & \sira \Hom(X,Y)\\
        g & \mapsto g \comp \alpha_X
      \end{align*}
      is bijective).
      As diagrams these equivalent conditions look as follows:
      \begin{equation*}
        \xymatrix{
          & {Y}\\
          {X} \ar[ur]^-a \ar@{..>}[r]_-{\exists ! a'} & {s\inv(Y)}
          \ar[u]_-{\alpha_{s\inv(Y)}} 
        }
        \quad
        \text{ and }
        \quad
        \xymatrix{
          {s(X)}  \ar@{..>}[r]^-{\exists ! b'} & {Y}\\
          {X} \ar[u]^-{\alpha_X} \ar[ur]_-b &
        }
      \end{equation*}
    \end{enumerate}
    By abuse of notation we then say that $\tildew{\mathcal{T}}$ is an
    f-category. 
  \item 
    \label{enum:filt-tria-over}
    Let $\mathcal{T}$ be a triangulated category. A \define{filtered
      triangulated category over\footnote
      {
        We do not ask here for a functor
        $\tildew{\mathcal{T}} \ra \mathcal{T}$ as suggested by the word
        ``over"; Proposition~\ref{p:functor-omega} will yield such a
        functor.
      } $\mathcal{T}$}
    (or \define{f-category over $\mathcal{T}$})
    is an f-category $\tildew{\mathcal{T}}$ together with an
    equivalence 
    \begin{equation*}
      i: \mathcal{T} \ra \tildew{\mathcal{T}}(\leq 0)\cap
      \tildew{\mathcal{T}}(\geq 0)  
    \end{equation*}
    of triangulated categories.
  \end{enumerate}
\end{definition}

Let $\tildew{\mathcal{T}}$ be an f-category.
We will use the shorthand notation
$\tildew{\mathcal{T}}([a,b])=\tildew{\mathcal{T}}(\leq b) \cap
\tildew{\mathcal{T}}(\geq a)$ and abbreviate
$[a,a]$ by $[a]$. Similarly $\tildew{\mathcal{T}}(<a)$
etc.\ have the obvious meaning.
For $a<b$ we have $\tildew{\mathcal{T}}(\leq a) \cap \tildew{\mathcal{T}}(\geq b)=0$: If $X$ is in this intersection then axiom \ref{enum:filt-tria-cat-no-homs} implies that $\id_X=0$ hence $X=0$.
Note that $\tildew{\mathcal{T}}$ together with the identity functor
$\id: \tildew{\mathcal{T}}([0]) \ra \tildew{\mathcal{T}}([0])$ is an
f-category over $\tildew{\mathcal{T}}([0])$.

\begin{remark}
  \label{rem:f-cat-vs-t-cat}
  Let $\tildew{\mathcal{T}}$ be a filtered triangulated category. Define 
  $\mathcal{D}^{t \leq 0}:=\tildew{\mathcal{T}}(\geq 1)$ and 
  $\mathcal{D}^{t \geq 0}:=\tildew{\mathcal{T}}(\leq 0)$.
  Then 
  $(\mathcal{D}^{t \leq 0}, 
  \mathcal{D}^{t \geq 0})$ defines a t-structure 
  on
  $\tildew{\mathcal{T}}$.

  Note that in this example all $\mathcal{D}^{t \leq i}$
  coincide since $\tildew{\mathcal{T}}(\geq 1)$ is a triangulated
  subcategory; 
  similarly, all $\mathcal{D}^{t \geq i}$ are equal.
  The heart of this t-structure is zero.

  Of course we can apply the shift of filtration to this example and
  obtain t-structures 
  $(\tildew{\mathcal{T}}(\geq n+1), \tildew{\mathcal{T}}(\leq n))$ for
  all $n \in \DZ$.

  Similarly, define 
  $\mathcal{E}^{w \leq 0} :=\tildew{\mathcal{T}}(\leq 0)$ and
  $\mathcal{E}^{w \geq 0} :=\tildew{\mathcal{T}}(\geq 1)$. Then 
  $(\mathcal{E}^{w \leq 0}, \mathcal{E}^{w \geq 0})$ is a weight
  structure on $\tildew{\mathcal{T}}$.
  Note that \ref{enum:ws-i} is satisfied since 
  $(\tildew{\mathcal{T}}(\geq 1), \tildew{\mathcal{T}}(\leq 0))$ is a
  t-structure. Again all $\mathcal{E}^{w \leq i}$ 
  (resp.\ $\mathcal{E}^{w \leq i}$) coincide and the heart is zero.
\end{remark}

\subsection{Basic example}
\label{sec:basic-example-fcat}

We introduce the filtered derived category of an abelian category,
following \cite[V.1]{illusie-cotan-i}. The reader who is not interested
in this basic example of an f-category can skip this section and
continue with \ref{sec:firstprop-filt-cat}.

Let $\mathcal{A}$ be an abelian category and $CF(\mathcal{A})$ the
category whose objects are complexes in $\mathcal{A}$ with a finite
decreasing filtration and whose morphisms are morphisms of complexes
which respect the filtrations. 
If $L$ is an object of $CF(\mathcal{A})$ and $i,j  \in \DZ$ we denote the
component of $L$ in degree $i$ by $L^i$ and by $F^jL$ the $j$-the step
of the filtration, and by $F^jL^i$ the component of degree $i$ in
$F^jL$.
Pictorially an object $L$ looks as
\begin{equation*}
  \xymatrix{
    {L:} & 
    {\dots} \ar@{}[r]|-{\supset} &
    {F^{-1}L} \ar@{}[r]|-{\supset} &
    {F^{0}L} \ar@{}[r]|-{\supset} &
    {F^{1}L} \ar@{}[r]|-{\supset} &
    {F^{2}L} \ar@{}[r]|-{\supset} &
    {\cdots}
  }
\end{equation*}
or, if we also indicate the differentials, as
\begin{equation*}
  \xymatrix@R1pc{
    {\dots} & 
    {} &
    {\dots} &
    {\dots} &
    {\dots} &
    {} \\
    {L^{1}:} \ar[u] & 
    {\dots} \ar@{}[r]|-{\supset} &
    {F^{-1}L^{1}} \ar@{}[r]|-{\supset} \ar[u] &
    {F^{0}L^{1}} \ar@{}[r]|-{\supset} \ar[u] &
    {F^{1}L^{1}} \ar@{}[r]|-{\supset} \ar[u] &
    {\dots} \\
    {L^{0}:} \ar[u] & 
    {\dots} \ar@{}[r]|-{\supset} &
    {F^{-1}L^{0}} \ar@{}[r]|-{\supset} \ar[u] &
    {F^{0}L^{0}} \ar@{}[r]|-{\supset} \ar[u] &
    {F^{1}L^{0}} \ar@{}[r]|-{\supset} \ar[u] &
    {\dots} \\
    {L^{-1}:} \ar[u] & 
    {\dots} \ar@{}[r]|-{\supset} &
    {F^{-1}L^{-1}} \ar@{}[r]|-{\supset} \ar[u] &
    {F^{0}L^{-1}} \ar@{}[r]|-{\supset} \ar[u] &
    {F^{1}L^{-1}} \ar@{}[r]|-{\supset} \ar[u] &
    {\dots} \\
    {\dots} \ar[u] & 
    {}  &
    {\dots} \ar[u] &
    {\dots} \ar[u] &
    {\dots} \ar[u] &
    {} 
  }
\end{equation*}

This is an additive category having kernels and cokernels, but 
it is
not abelian (if $\mathcal{A}\not=0$).
There is an obvious translation functor $[1]$ on $CF(\mathcal{A})$.

A morphism $f: L \ra M$ between objects of $CF(\mathcal{A})$
is called a \define{quasi-isomorphism} if one of the following
equivalent conditions is satisfied:
\begin{enumerate}
\item $F^nf:F^nL \ra F^nM$ is a quasi-isomorphism for all $n \in \DZ$.
\item $\gr^n(f): \gr^n(L) \ra \gr^n(M)$ is a quasi-isomorphism for
  all $n \in \DZ$.
\end{enumerate}

We localize $CF(\mathcal{A})$ with respect to the class of all
quasi-isomorphisms and obtain the \define{filtered derived category}
$DF(\mathcal{A})$ of $\mathcal{A}$.
This category can equivalently be constructed as the localization of
the filtered homotopy category. The latter category is triangulated
with triangles isomorphic to mapping cone triangles; this
structure of a triangulated category is inherited to
$DF(\mathcal{A})$.

Morphisms in $DF(\mathcal{A})$ are equivalence classes of so-called roofs:
Any morphism $f: L \ra M$ in $DF(\mathcal{A})$ in $DF(\mathcal{A})$ 
can be represented as
\begin{equation*}
  g s\inv: L \xla{s} L' \xra{g} M
\end{equation*}
where $s$ and $g$ are morphisms in $CF(\mathcal{A})$ and $s$ is a
quasi-isomorphism. Similarly, one can also represent $f$ as
\begin{equation*}
  t\inv h: L \xra{h} M' \xla{t} M
\end{equation*}
where $t$ and $h$ are morphisms in $CF(\mathcal{A})$ and $t$ is a
quasi-isomorphism.

Let
$D(\mathcal{A})$ be the derived category of $\mathcal{A}$.
The functor $\gr^n:CF(\mathcal{A}) \ra C(\mathcal{A})$ passes to the
derived categories and yields a triangulated functor $\gr^n:
DF(\mathcal{A}) \ra D(\mathcal{A})$.

Define (strict) full subcategories
\begin{align*}
  DF(\mathcal{A})(\leq n)& :=\{L \in DF(\mathcal{A})\mid \text{$\gr^i (L)
    =0$ for all $i>n$}\},\\
  DF(\mathcal{A})(\geq n)& :=\{L \in DF(\mathcal{A})\mid \text{$\gr^i (L)
    =0$ for all $i<n$}\}.
\end{align*}

Let $s: CF(\mathcal{A}) \ra CF(\mathcal{A})$ the functor which shifts
the filtration: 
Given an object $L$ the object $s(L)$ has the same underlying complex
but filtration $F^n(s(L))=F^{n-1}L$. 
It induces a triangulated automorphism $s:
DF(\mathcal{A}) \ra DF(\mathcal{A})$.
Let $\alpha: \id_{DF(\mathcal{A})} \ra s$ be the obvious morphism of
triangulated functors: We include a picture of $\alpha_L$ where we indicate the $0$-th part of the filtration by a box:
\begin{equation*}
  \xymatrix@R1pc{
    {L:} \ar[d]^-{\alpha_L} & 
    {\dots} \ar@{}[r]|-{\supset} &
    {F^{-1}L} \ar@{}[r]|-{\supset} \ar[d] &
    {\mathovalbox{F^{0}L}} \ar@{}[r]|-{\supset} \ar[d] &
    {F^1L} \ar@{}[r]|-{\supset} \ar[d] &
    {\cdots} \\
    {s(L):} 
    & 
    {\dots} \ar@{}[r]|-{\supset} &
    {F^{-2}L} \ar@{}[r]|-{\supset} 
    &
    {\mathovalbox{F^{-1}L}} \ar@{}[r]|-{\supset} 
    &
    {F^{0}L} \ar@{}[r]|-{\supset} 
    &
    {\cdots} 
  }
\end{equation*}
Note that $s^n(DF(\mathcal{A})(\leq 0))= DF(\mathcal{A})(\leq n)$ and
$s^n(DF(\mathcal{A})(\geq 0))= DF(\mathcal{A})(\geq n)$.

We define a functor 
$i:D(\mathcal{A}) \ra DF(\mathcal{A})$  
by mapping an 
object $L$ of $D(\mathcal{A})$ to $i(L)=(L,\Tr)$, where $\Tr$ is the 
trivial filtration on $L$ defined by $\Tr^i L = L$ for $i \leq 0$ 
and $\Tr^i L =0$ for $i > 0$.
We often consider $i$ as a functor to $DF(\mathcal{A})([0])$.

\begin{proposition}
  [{cf.\ \cite[Example A 2]{Beilinson}}]
  \label{p:basic-ex-f-cat}
  The datum
  \begin{equation}
    \label{eq:filt-der}
    (DF(\mathcal{A}), DF(\mathcal{A})(\leq 0),
    DF(\mathcal{A})(\geq 0), s, \alpha)
  \end{equation}
  defines a filtered triangulated
  category $DF(\mathcal{A})$, and
  the functor $i:D(\mathcal{A}) \ra DF(\mathcal{A})([0])$ makes
  $DF(\mathcal{A})$ into a filtered 
  triangulated category over $D(\mathcal{A})$.
\end{proposition}

\begin{proof}
  We first check that \eqref{eq:filt-der} defines a filtered
  triangulated category.
  Since all $\gr^i: DF(\mathcal{A}) \ra D(\mathcal{A})$ are
  triangulated functors, $DF(\mathcal{A})(\leq 0)$ 
  and $DF(\mathcal{A})(\geq 0)$ are strict full triangulated
  subcategories of $DF(\mathcal{A})$. 
  The conditions \ref{enum:filt-tria-cat-shift}, 
  \ref{enum:filt-tria-cat-exhaust} (we use finite filtrations) and
  \ref{enum:filt-tria-cat-shift-alpha} are obviously satisfied.

  \textbf{Condition \ref{enum:filt-tria-cat-triang}}:
  Let $X$ be any object in $DF(\mathcal{A})$. We define objects
  $X(\geq 1)$ 
  and $X/(X(\geq 1))$ 
  and (obvious) morphisms
$    X(\geq 1) \xra{i} X \xra{p} X/(X(\geq 1))$
  in $CF(\mathcal{A})$
  by the following
  diagram:
  \begin{equation}
    \label{eq:basic-example-fcat-def-Xgeq-triangle}
    \hspace{-1.6cm}
    \xymatrix@R1pc{
      {X(\geq 1):} \ar[d]^-{i} & 
      {\dots} \ar@{}[r]|-{=} &
      {F^{1}X} \ar@{}[r]|-{=} \ar[d] &
      {\mathovalbox{F^{1}X}} \ar@{}[r]|-{=} \ar[d] &
      {F^{1}X} \ar@{}[r]|-{\supset} \ar[d] &
      {F^{2}X} \ar@{}[r]|-{\supset} \ar[d] &
      {\cdots} \\
      {X:} \ar[d] \ar[d]^-{p} & 
      {\dots} \ar@{}[r]|-{\supset} &
      {F^{-1}X} \ar@{}[r]|-{\supset}  \ar[d] &
      {\mathovalbox{F^{0}X}} \ar@{}[r]|-{\supset}  \ar[d] &
      {F^1X} \ar@{}[r]|-{\supset}  \ar[d] &
      {F^2X} \ar@{}[r]|-{\supset}  \ar[d] &
      {\cdots} \\
      {X/(X(\geq 1)):} & 
      {\dots} \ar@{}[r]|-{\supset} &
      {F^{-1}X/F^1X} \ar@{}[r]|-{\supset} &
      {\mathovalbox{F^{0}X/F^1X}} \ar@{}[r]|-{\supset} &
      {0} \ar@{}[r]|-{=} &
      {0} \ar@{}[r]|-{=} &
      {\cdots} \\
    }
  \end{equation}
  There is a morphism $X/(X(\geq 1)) \ra [1](X(\geq 1))$ in
  $DF(\mathcal{A})$ such that 
  \begin{equation}
    \label{eq:triangle-fcat-basic}
    \xymatrix{
      {X(\geq 1)} \ar[r]^-{i} &
      {X} \ar[r]^-{p} &
      {X/(X(\geq 1))} \ar[r] &
      {[1](X(\geq 1))}
    }
  \end{equation}
  is a triangle in $DF(\mathcal{A})$: Use the obvious quasi-isomorphism from
  the mapping cone of $i$ to $X/(X(\geq 1))$.
  Since $X(\geq 1) \in DF(\mathcal{A})(\geq 1)$ and
  $X/(X(\geq 1)) \in DF(\mathcal{A})(\leq 0)$ by definition this
  proves condition \ref{enum:filt-tria-cat-triang}.

  \textbf{Observation:}
  Application of the triangulated functors $\gr^i$ to the 
  triangle \eqref{eq:triangle-fcat-basic} shows:
  If $X$ is in $DF(\mathcal{A})(\geq 1)$ then 
  $X (\geq 1) \ra X$ is an isomorphism in $DF(\mathcal{A})$.
  Similarly, if $X$ is in $DF(\mathcal{A})(\leq 0)$, then $X \ra
  X/(X(\geq 1))$ is an isomorphism.
  We obtain:
  \begin{itemize}
  \item 
    Any object in $DF(\mathcal{A})(\geq a)$
    is isomorphic to an object $Y$ with $Y=F^{-\infty}Y = \dots =F^{a-1}Y
    =F^aY$.
  \item Any object in $DF(\mathcal{A})(\leq b)$ is
    isomorphic to an object $Y$ with $0=F^{b+1}Y=F^{b+2}Y = \dots$.
  \item Any object in $DF(\mathcal{A})([a,b])$ is
    isomorphic to an object $Y$ with 
    $Y= F^{-\infty}Y=\dots 
    =F^aY \supset \dots \supset 0=F^{b+1}Y = \dots$.
  \end{itemize}

  \textbf{Condition \ref{enum:filt-tria-cat-no-homs}}:
  Let $X$ in $DF(\mathcal{A})(\geq 1)$ and $Y$ in
  $DF(\mathcal{A})(\leq 0)$. 
  By the above observation we can assume that $0=F^1Y =F^2Y = \dots$.
  Let a morphism $f: X \ra Y$ be
  represented by a roof $X \xla{s} Z \xra{g} Y$ with $s$ a
  quasi-isomorphism. 
  Then the obvious morphism $\iota: Z(\geq 1) \ra Z$ is a
  quasi-isomorphism and the roof 
  $X \xla{s} Z \xra{g} Y$ is equivalent to the roof
  $X \xla{s\iota} Z(\geq 1) \xra{g \iota} Y$. Since 
  $F^1Y=0$ and $F^1(Z(\geq 1))=Z(\geq 1)$ we obtain $g \iota=0$. Hence
  $f=0$.

  \textbf{Condition \ref{enum:filt-tria-cat-hom-bij}}:
  Let $X$ in $DF(\mathcal{A})(\geq 1)$ and $Y$ in
  $DF(\mathcal{A})(\leq 0)$ as before. 
  As observed above we can assume that $X= \dots =F^0X=F^1X$
  and that $0=F^1Y = F^2Y =\dots$.
  
  We prove that
  \begin{align*}
    \Hom(sY, X) & \ra \Hom(Y,X)\\
    g & \mapsto g \comp \alpha_Y
  \end{align*}
  is bijective. Here is a picture of $\alpha_Y$:
  \begin{equation*}
    \xymatrix@R1pc{
      {Y:} \ar[d]^-{\alpha_Y} & 
      {\dots} \ar@{}[r]|-{\supset} &
      {F^{-1}Y} \ar@{}[r]|-{\supset} \ar[d] &
      {\mathovalbox{F^{0}Y}} \ar@{}[r]|-{\supset} \ar[d] &
      {0} \ar@{}[r]|-{=} \ar[d] &
      {0} \ar@{}[r]|-{=} \ar[d] &
      {\cdots} \\
      {s(Y):} & 
      {\dots} \ar@{}[r]|-{\supset} &
      {F^{-2}Y} \ar@{}[r]|-{\supset} &
      {\mathovalbox{F^{-1}Y}} \ar@{}[r]|-{\supset} &
      {F^{0}Y} \ar@{}[r]|-{\supset} &
      {0} \ar@{}[r]|-{=} &
      {\cdots} 
    }
  \end{equation*}
  We define a map $\Hom(Y, X)\ra \Hom(s(Y), X)$ which will be inverse
  to the above map. 
  Let $f: Y \ra X$ be a morphism, represented by a roof 
  $Y \xra{h} Z \xla{t} X$ where $h$ and $t$ are morphisms in
  $CF(\mathcal{A})$ and $t$ is a quasi-isomorphism. We 
  define an object $\tildew{Z}$ and a morphism $Z \ra \tildew{Z}$ by
  the following picture in which we include $Z \xla{t} X$:
  \begin{equation*}
    \xymatrix@R1pc{
      {\tildew{Z}:} & 
      {\dots} \ar@{}[r]|-{=} &
      {Z} \ar@{}[r]|-{=} &
      {\mathovalbox{Z}} \ar@{}[r]|-{=} &
      {Z} \ar@{}[r]|-{\supset} &
      {F^2Z} \ar@{}[r]|-{\supset} &
      {\cdots} \\
      {Z:} \ar[u]^-{s} & 
      {\dots} \ar@{}[r]|-{\supset} &
      {F^{-1}Z} \ar@{}[r]|-{\supset} \ar[u] &
      {\mathovalbox{F^{0}Z}} \ar@{}[r]|-{\supset} \ar[u] &
      {F^{1}Z} \ar@{}[r]|-{\supset} \ar[u] &
      {F^2Z} \ar@{}[r]|-{\supset} \ar[u] &
      {\cdots} \\
      {X:} \ar[u]^-{t} & 
      {\dots} \ar@{}[r]|-{=} &
      {X} \ar@{}[r]|-{=} \ar[u] &
      {\mathovalbox{X}} \ar@{}[r]|-{=} \ar[u] &
      {X} \ar@{}[r]|-{\supset} \ar[u] &
      {F^{2}X} \ar@{}[r]|-{\supset} \ar[u] &
      {\cdots} 
    }
  \end{equation*}
  Since $t$ is a quasi-isomorphism, all $F^it:F^iX \ra F^iZ$ are
  quasi-isomorphisms. For $i$ small enough we have $Z=F^iZ$; this
  implies that $X \ra Z$ is a quasi-isomorphism in $C(\mathcal{A})$; hence
  $st: X \ra \tildew{Z}$ is a quasi-isomorphism in $CF(\mathcal{A})$. 
  Hence we get the following diagram
  \begin{equation*}
    \xymatrix{
       & {\tildew{Z}}\\
       {Y} \ar[d]_{\alpha_Y} \ar[r]^-h \ar[ru]^-{sh} & {Z} \ar[u]^-s &
       {X} \ar[l]_t \ar[lu]_{st}\\
       {s(Y)}
    }
  \end{equation*}
  Because of the special form of the filtrations on $\tildew{Z}$ and
  on $Y$ it is obvious that $sh:Y \ra \tildew{Z}$ comes from a
  unique morphism $\lambda:s(Y) \ra \tildew{Z}$ in $CF(\mathcal{A})$
  such that $\lambda \alpha_Y =sh$. We map $f$ to the class of
  the roof $(st)\inv \lambda$; it is easy to check that this is
  well-defined and inverse to the map $g \mapsto g \comp \alpha_Y$.

  Now we check that $i$ makes $DF(\mathcal{A})$ into an f-category over
  $D(\mathcal{A})$.
  It is obvious that
  $i:D(\mathcal{A}) \ra DF(\mathcal{A})([0])$ 
  is triangulated. 
  Our observation shows that it is essentially surjective.
  Since $\gr^0 \comp i=\id_{D(\mathcal{A})}$, our functor $i$ is
  faithful. It remains to prove fullness:
  Let $X$ and $Y$ be in $D(\mathcal{A})$ and let $f: i(X) \ra i(Y)$ be
  a morphism in $DF(\mathcal{A})$, represented by a roof
  $i(X) \xla{s} Z \xra{g} i(Y)$ with $s$ a quasi-isomorphism. 
  Consider the morphism
  \begin{equation*}
    \xymatrix@R1pc{
      {i(F^0Z):} \ar[d]^-{t} & 
      {\dots} \ar@{}[r]|-{=} &
      {F^{0}Z} \ar@{}[r]|-{=} \ar[d] &
      {\mathovalbox{F^{0}Z}} \ar@{}[r]|-{\supset} \ar[d] &
      {0} \ar@{}[r]|-{=} \ar[d] &
      {\cdots} \\
      {Z:} & 
      {\dots} \ar@{}[r]|-{\supset} &
      {F^{-1}Z} \ar@{}[r]|-{\supset} &
      {\mathovalbox{F^{0}Z}} \ar@{}[r]|-{\supset} &
      {F^1Z} \ar@{}[r]|-{\supset} &
      {\cdots} 
    }
  \end{equation*}
  Since $F^0s:F^0Z \ra F^0i(X)=X$ is a
  quasi-isomorphism, $st$ is a quasi-isomorphisms (and so is $t$).
  But the roof $i(X) \xla{st} i(F^0Z) \xra{gt} i(Y)$ comes from a roof
  $X \la F^0Z \ra Y$; hence $f$ is in the image of $i$.
\end{proof}

\subsection{First properties of filtered triangulated categories}
\label{sec:firstprop-filt-cat}

We will make heavy use of some results of
\cite[App.]{Beilinson} in Section~\ref{sec:strong-WCfun}.

As no proofs have appeared we give proofs for the more difficult
results we need.

\begin{proposition}
  [{cf.\ \cite[Prop.~A 3]{Beilinson} (without proof)}]
  \label{p:firstprop-filt-cat}
  Let $\tildew{\mathcal{T}}$ be a filtered triangulated category and
  $n \in \DZ$.
  \begin{enumerate}
  \item 
    \label{enum:firstprop-filt-cat-right-and-left}
    The inclusion $\tildew{\mathcal{T}}(\geq n) \subset
    \tildew{\mathcal{T}}$ has a right adjoint $\sigma_{\geq n}:
    \tildew{\mathcal{T}} \ra \tildew{\mathcal{T}}(\geq n)$,
    and the inclusion $\tildew{\mathcal{T}}(\leq n) \subset
    \tildew{\mathcal{T}}$ has a left adjoint $\sigma_{\leq n}:
    \tildew{\mathcal{T}} \ra \tildew{\mathcal{T}}(\leq n)$.
  \end{enumerate}
  We fix all these adjunctions.
  \begin{enumerate}[resume]
  \item 
    \label{enum:firstprop-filt-cat-trunc-triangle}
    For any $X$ in $\tildew{\mathcal{T}}$ there is a unique morphism
    $v^n_X: \sigma_{\leq n}X \ra [1]\sigma_{\geq n+1}X$ in ${\tildew{\mathcal{T}}}$
    such that the candidate triangle 
    \begin{equation}
      \label{eq:sigma-trunc-triangle}
      \xymatrix{
        {\sigma_{\geq n+1}X} \ar[r]^-{g^{n+1}_X} 
        & {X} \ar[r]^-{k^n_X}
        & {\sigma_{\leq n}X} \ar[r]^-{v^n_X}
        & {[1]\sigma_{\geq n+1}X}
      }
    \end{equation}
    is a triangle where the first two
    morphisms are adjunction morphisms. From every triangle $A\ra X
    \ra B \ra[1]A$ with
    $A$ in $\tildew{\mathcal{T}}(\geq n)$ and $B$ in
    $\tildew{\mathcal{T}}(< n)$ there is a 
    unique isomorphism of triangles to the
    above triangle extending $\id_X$.
    We call 
    \eqref{eq:sigma-trunc-triangle}
    the
    \textbf{$\sigma$-truncation triangle (of type $(\geq n+1, \leq n)$)}.
  \item 
    \label{enum:firstprop-filtr-cat-trunc-perps}
    We have
    \begin{align*}
      \tildew{\mathcal{T}}(\leq n) & = (\tildew{\mathcal{T}}(>n))^\perp \quad \text{and}\\
      \tildew{\mathcal{T}}(\geq n) & =
      \leftidx{^\perp}{(\tildew{\mathcal{T}}(< n))}.
    \end{align*}
    In particular if in a triangle $X \ra Y \ra Z
    \ra [1]X$ two out of the three objects $X$, $Y$, $Z$ are in
    $\tildew{\mathcal{T}}(\leq n)$ (resp.\ $\tildew{\mathcal{T}}(\geq
    n)$) then so is the third. 
  \item 
    \label{enum:firstprop-filt-cat-trunc-preserve}
    Let $a, b \in \DZ$.
    All functors $\sigma_{\leq n}$ and $\sigma_{\geq n}$ are
    triangulated and preserve all subcategories
    $\tildew{\mathcal{T}}(\leq a)$ and
    $\tildew{\mathcal{T}}(\geq b)$.
    There is a unique morphism 
    \begin{equation}
      \label{eq:interval-isom-sigma-trunc-morph-functors}
      \sigma^{[b,a]}:\sigma_{\leq a} \sigma_{\geq b} \ra
      \sigma_{\geq b} \sigma_{\leq a}
    \end{equation}
    (which is in fact an isomorphism)
    such that the diagram 
    \begin{equation}
      \label{eq:interval-isom-sigma-trunc}
      \xymatrix{
        {\sigma_{\geq b}X} \ar[r]^-{g^{b}_X}
        \ar[d]_-{k^a_{\sigma_{\geq b}X}}
        & 
        {X} \ar[r]^-{k_X^a}
        &
        {\sigma_{\leq a}X}
        \\
        {\sigma_{\leq a}\sigma_{\geq b}X}
        \ar[rr]^-{\sigma^{[b,a]}_X}
        &&
        {\sigma_{\geq b}\sigma_{\leq a} X} \ar[u]_-{g^{b}_{\sigma_{\leq
              a}X}}
      }
    \end{equation}
    commutes for all $X$ in $\tildew{\mathcal{T}}$.
  \end{enumerate}
\end{proposition}

Our proof of this theorem gives some more canonical isomorphisms, see
Remark~\ref{rem:consequences-from-three-times-three} below. 
If we were only interested in the statements of
Proposition~\ref{p:firstprop-filt-cat} 
we
could do without the
$3 \times 3$-diagram~\eqref{eq:sigma-truncation-two-parameters}.

\begin{proof}
  The statements 
  \eqref{enum:firstprop-filt-cat-right-and-left} and
  \eqref{enum:firstprop-filt-cat-trunc-triangle} 
  follow 
  from the fact that 
  $(\tildew{\mathcal{T}}(\geq n+1), \tildew{\mathcal{T}}(\leq n))$ is
  a t-structure for all $n \in \DZ$
  (see Rem.~\ref{rem:f-cat-vs-t-cat}) and \cite[1.3.3]{BBD}.
  For \eqref{enum:firstprop-filtr-cat-trunc-perps} use
  \cite[1.3.4]{BBD} and the fact that $\tildew{\mathcal{T}}(\geq n)$
  and $\tildew{\mathcal{T}}(\leq n)$ are stable under $[1]$.

  We prove \eqref{enum:firstprop-filt-cat-trunc-preserve}:
  The functors $\sigma_{\leq n}$ and $\sigma_{\geq n}$ are
  triangulated by Proposition~\ref{p:linksad-triang-for-wstr-art}.
  Let $X \in \tildew{\mathcal{T}}$ and $a, b \in \DZ$. Consider the
  following $3 \times 3$-diagram
  \begin{equation}
    \label{eq:sigma-truncation-two-parameters}
    \xymatrix@=1.2cm{
      {[1]\sigma_{\geq b+1}\sigma_{\geq a+1} X} \ar@{..>}[r]^-{[1]\sigma_{\geq b+1}(g^{a+1}_X)} & 
      {[1]\sigma_{\geq b+1} X} \ar@{..>}[r]^-{[1]\sigma_{\geq b+1}({k^a_X})} & 
      {[1]\sigma_{\geq b+1}\sigma_{\leq a}X} \ar@{..>}[r]^-{[1]\sigma_{\geq b+1}(v^a_X)} \ar@{}[rd]|{\anticomm}& 
      {[2]\sigma_{\geq b+1}\sigma_{\geq a+1} X} \\
      {\sigma_{\leq b}\sigma_{\geq a+1}X}
      \ar[u]^{v^b_{\sigma_{\geq a+1}X}} \ar[r]^-{\sigma_{\leq b}(g^{a+1}_X)} & 
      {\sigma_{\leq b}X} \ar[u]^{v^b_{X}}
      \ar[r]^-{\sigma_{\leq b}({k^a_X})} &
      {\sigma_{\leq b}\sigma_{\leq a}X}
      \ar[u]^{v^b_{\sigma_{\leq a}X}} \ar[r]^-{\sigma_{\leq b}(v^a_X)} & 
      {[1]\sigma_{\leq b}\sigma_{\geq a+1}X} \ar@{..>}[u]^{[1]v^b_{\sigma_{\geq a+1}X}} \\
      {\sigma_{\geq a+1}X} \ar[r]^-{g^{a+1}_X} \ar[u]^{k^b_{\sigma_{\geq a+1}X}} & 
      {X} \ar[r]^-{k^a_X} \ar[u]^{k^b_{X}} 
      \ar@{}[ru]|{(2)} 
      &
      {\sigma_{\leq a}X} \ar[r]^-{v^a_X} \ar[u]^{k^b_{\sigma_{\leq a}X}} & 
      {[1]\sigma_{\geq a+1}X} \ar@{..>}[u]^{[1]k^b_{\sigma_{\geq a+1}X}} \\
      {\sigma_{\geq b+1}\sigma_{\geq a+1} X} \ar[u]^{g^{b+1}_{\sigma_{\geq
            a+1}X}} \ar[r]^-{\sigma_{\geq b+1}(g^{a+1}_X)} 
      \ar@{}[ru]|{(1)} 
      &  
      {\sigma_{\geq b+1} X} \ar[u]^{g^{b+1}_{X}} \ar[r]^-{\sigma_{\geq b+1}({k^a_X})} &
      {\sigma_{\geq b+1}\sigma_{\leq a}X} \ar[u]^{g^{b+1}_{\sigma_{\leq a}X}}
      \ar[r]^-{\sigma_{\geq b+1}(v^a_X)} &  
      {[1]\sigma_{\geq b+1}\sigma_{\geq a+1} X} \ar@{..>}[u]^{[1]g^{b+1}_{\sigma_{\geq a+1}X}} 
    }
  \end{equation}
  constructed as follows: All morphisms $g$ and $k$ are
  adjunction morphisms. We start with the third row which is the
  $\sigma$-truncation triangle of $X$ of type $(\geq a+1, \leq a)$.
  We apply the triangulated functors $\sigma_{\leq b}$ and $\sigma_{\geq
    b+1}$ to this triangle and obtain the triangles in the second and
  fourth row. The adjunctions give the morphisms of triangles from
  fourth to third and third to second row.
  Then extend the first three columns to $\sigma$-truncation
  triangles; they can be uniquely connected by morphisms of triangles extending
  $g^{a+1}_X$ and $k^a_X$ respectively 
  using Proposition~\ref{p:BBD-1-1-9-copied-for-w-str}.
  Similarly (multiply the last arrow in the fourth column by $-1$ to get a
  triangle) we obtain the morphism between third and fourth column.

  We prove that $\sigma_{\geq b+1}$ and $\sigma_{\leq b}$ preserve the
  subcategories 
  $\tildew{\mathcal{T}}(\geq a+1)$ and $\tildew{\mathcal{T}}(\leq a)$.

  \begin{itemize}
  \item 
    \textbf{Case $a \geq b$:} Then in the left vertical triangle the
    first two objects are in $\tildew{\mathcal{T}}(\geq b+1)$; hence
    $\sigma_{\leq b}\sigma_{\geq a+1}X \in \tildew{\mathcal{T}}(\geq b+1) \cap
    \tildew{\mathcal{T}}(\leq b)$ is zero 
    (use \eqref{enum:firstprop-filtr-cat-trunc-perps}).
    (This shows that
    $g^{b+1}_{\sigma_{\geq a+1}X}$ and $\sigma_{\leq b}(k^a_X)$ are
    isomorphisms.)
    \begin{itemize}
    \item \textbf{$X \in \tildew{\mathcal{T}}(\geq a+1)$:} Then
      $g^{a+1}_X$ is an isomorphism and the first two vertical
      triangles are isomorphic. This shows $\sigma_{\leq b}X=0 \in
      \tildew{\mathcal{T}}(\geq a+1)$ and that all four morphisms of the
      square $(1)$ are isomorphisms; hence
      $\sigma_{\geq b+1}X
      \in \tildew{\mathcal{T}}(\geq a+1)$.
    \item \textbf{$X \in \tildew{\mathcal{T}}(\leq a)$:} Then
      $\sigma_{\leq b}X \in \tildew{\mathcal{T}}(\leq b) \subset
      \tildew{\mathcal{T}}(\leq a)$. Hence in the second vertical triangle
      two objects are in $\tildew{\mathcal{T}}(\leq a)$; hence
      the third object $\sigma_{\geq b+1}X$ is in
      $\tildew{\mathcal{T}}(\leq a)$.
    \end{itemize}

  \item 
    \textbf{Case $a \leq b$:} Then in the third vertical triangle the
    second and third object are in $\tildew{\mathcal{T}}(\leq b)$; hence
    $\sigma_{\geq b+1}\sigma_{\leq a}X \in \tildew{\mathcal{T}}(\leq b)\cap
    \tildew{\mathcal{T}}(\geq b+1)$ is zero.
    (This shows that
    $k^b_{\sigma_{\leq a}X}$ and $\sigma_{\geq b+1}(g^{a+1}_X)$ are
    isomorphisms.)
    \begin{itemize}
    \item \textbf{$X \in \tildew{\mathcal{T}}(\leq a)$:} Then
      $k^a_X$ is an isomorphism and the second and third vertical
      triangles are isomorphic. This shows $\sigma_{\geq b+1}X=0 \in
      \tildew{\mathcal{T}}(\leq a)$ and that all four morphisms of the
      square $(2)$ are isomorphisms; hence
      $\sigma_{\leq b}X \in \tildew{\mathcal{T}}(\leq a)$.
    \item \textbf{$X \in \tildew{\mathcal{T}}(\geq a+1)$:} Then
      $\sigma_{\geq b+1}X \in \tildew{\mathcal{T}}(\geq b+1) \subset
      \tildew{\mathcal{T}}(\geq a+1)$. Hence in the second vertical triangle
      two objects are in $\tildew{\mathcal{T}}(\geq a+1)$; hence
      the third object $\sigma_{\leq b}X$ is in $\tildew{\mathcal{T}}(\geq
      a+1)$.
    \end{itemize}
  \end{itemize}

  For the last statement consider the diagram 
  \eqref{eq:interval-isom-sigma-trunc} without the arrow
  $\sigma^{[b,a]}_X$ and with $b$ replaced by $b+1$.
  The vertical arrows are part of $\sigma$-truncation
  triangles. 
  Note that we already know that
  $\sigma_{\geq b+1}\sigma_{\leq a}X \in \tildew{\mathcal{T}}(\leq a)$ and
  $\sigma_{\leq a}\sigma_{\geq b+1}X \in \tildew{\mathcal{T}}(\geq b+1)$. 
  Appropriate cohomological functors 
  give the following
  commutative diagram of isomorphisms:
  \begin{equation*}
    \xymatrix@C=1.5cm{
      {\tildew{\mathcal{T}}(\sigma_{\geq b+1}X, \sigma_{\leq a}X)} 
      &
      {\tildew{\mathcal{T}}(\sigma_{\geq b+1}X, \sigma_{\geq b+1}\sigma_{\leq
          a}X)} 
      \ar[l]_-{g^{b+1}_{\sigma_{\leq a}X} \comp ?}^-{\sim}\\
      {\tildew{\mathcal{T}}(\sigma_{\leq a}\sigma_{\geq b+1}X, \sigma_{\leq a}X)} 
      \ar[u]_-{? \comp k^a_{\sigma_{\geq b+1}X}}^-{\sim}
      &
      {\tildew{\mathcal{T}}(\sigma_{\leq a}\sigma_{\geq b+1}X, \sigma_{\geq b+1}\sigma_{\leq
          a}X)} 
      \ar[l]_-{g^{b+1}_{\sigma_{\leq a}X} \comp ?}^-{\sim}
      \ar[u]_-{? \comp k^a_{\sigma_{\geq b+1}X}}^-{\sim}
    }
  \end{equation*}
  It shows that there is a unique morphism 
  \begin{equation*}
    \sigma^{[b+1,a]}_X:
    \sigma_{\leq a}\sigma_{\geq b+1}X \ra \sigma_{\geq b+1}\sigma_{\leq a}X
  \end{equation*}
  such that $g^{b+1}_{\sigma_{\leq a}X} \comp \sigma^{[b+1,a]}_X \comp
  k^a_{\sigma_{\geq b+1}X}= k^a_X \comp g^{b+1}_X$.
  We have to prove that $\sigma^{[b+1,a]}_X$ is an isomorphism.
  
  From \eqref{enum:firstprop-filt-cat-trunc-triangle} we see that
  the lowest horizontal triangle in
  \eqref{eq:sigma-truncation-two-parameters} is uniquely isomorphic
  (by an isomorphism extending the identity) to the corresponding
  $\sigma$-truncation triangle: 
  \begin{equation*}
    \xymatrix@=1.2cm{
      {\sigma_{\geq b+1}\sigma_{\geq a+1} X} 
      \ar[r]^-{\sigma_{\geq b+1}(g^{a+1}_X)} 
      &  
      {\sigma_{\geq b+1} X} 
      \ar[r]^-{\sigma_{\geq b+1}({k^a_X})} 
      &
      {\sigma_{\geq b+1}\sigma_{\leq a}X} 
      \ar[r]^-{\sigma_{\geq b+1}(v^a_X)} 
      &  
      {[1]\sigma_{\geq b+1}\sigma_{\geq a+1} X} 
      \\
      {\sigma_{\geq a+1} \sigma_{\geq b+1}X} 
      \ar[r]^-{g^{a+1}_{\sigma_{\geq b+1}X}}
      \ar[u]^{f}_{\sim}
      &
      {\sigma_{\geq b+1}X} 
      \ar[r]^-{k^a_{\sigma_{\geq b+1}X}}
      \ar[u]^{\id}
      &
      {\sigma_{\leq a} \sigma_{\geq b+1}X} 
      \ar[r]^-{v^a_{\sigma_{\geq b+1}X}}
      \ar[u]^{h}_{\sim}
      &
      {[1]\sigma_{\geq a+1} \sigma_{\geq b+1}X} 
      \ar[u]^{[1]f}_{\sim}
    }
  \end{equation*}
  In combination with \eqref{eq:sigma-truncation-two-parameters} this
  diagram yields
  $g^{b+1}_{\sigma_{\leq a}X} \comp h \comp k^a_{\sigma_{\geq b+1}X} =
  k^a_X \comp g^{b+1}_X$. This shows that $\sigma^{[b+1,a]}_X=h$
  which is hence an isomorphism.
  Similarly it is easy to see that $X \mapsto \sigma_X^{[b,a]}$ in
  fact defines an isomorphism
  \eqref{eq:interval-isom-sigma-trunc-morph-functors}
  of functors.  
\end{proof}

\begin{remark}
  \label{rem:consequences-from-three-times-three}
  Let us just mention some consequences one can now easily deduce from
  the $3\times 3$-diagram~\eqref{eq:sigma-truncation-two-parameters}.
  \begin{itemize}
  \item 
    (As already mentioned in the proof:)
    If $a \geq b$ the object $\sigma_{\leq b}\sigma_{\geq a+1}X$ is zero
    providing two isomorphisms 
    \begin{align}
      \label{eq:a-geq-b-sigma-isos}
      \sigma_{\leq b}(k^a_X): \sigma_{\leq b}X & \sira
      \sigma_{\leq b}\sigma_{\leq a}X 
      \quad \text{(for $a \geq b$), and} \\
      \notag
      g^{b+1}_{\sigma_{\geq a+1}X}: \sigma_{\geq b+1}\sigma_{\geq a+1}X & \sira
      \sigma_{\geq a+1}X \quad \text{(for $a \geq b$).}
    \end{align}
    Similarly $\sigma_{\geq b+1}\sigma_{\leq a}X$ vanishes for $a \leq b$
    and provides two isomorphisms
    \begin{align}
      \label{eq:a-leq-b-sigma-isos}
      \sigma_{\geq b+1}(g^{a+1}_X): \sigma_{\geq b+1}\sigma_{\geq a+1}X & \sira
      \sigma_{\geq b+1}X 
      \quad \text{(for $a \leq b$), and} \\
      \notag
      k^b_{\sigma_{\leq a}X}: \sigma_{\leq a}X & \sira
      \sigma_{\leq b}\sigma_{\leq a}X \quad \text{(for $a \leq b$).}
    \end{align}
  \item 
    In case $a=b$ the two squares marked $(1)$ and $(2)$ consist of
    isomorphisms. Hence 
    \begin{equation}
      \label{eq:a-eq-b-sigma-epsilon-eta-isos}
      g^{a+1}_{\sigma_{\geq a+1}X}= \sigma_{\geq a+1}(g^{a+1}_X)
      \quad \text{and}
      \quad
      k^a_{\sigma_{\leq a}X}= \sigma_{\leq a}(k^a_X).
    \end{equation}
  \item 
    Proposition~\ref{p:BBD-1-1-9-copied-for-w-str} gives several
    uniqueness statements, e.\,g.\ it shows that the
    morphisms connecting the horizontal triangles are also unique
    extending $g^{b+1}_X$, $k^b_X$ and $v^b_X$ respectively
    (in the last case one has to change the sign of the third morphism
    of the top sequence to make it into a triangle).
  \item Application of $\sigma_{\geq a+1} \xra{g^{a+1}} \id
    \xra{k^a} \sigma_{\leq a} \xra{v^a} [1]\sigma_{\geq a+1}$ to the
    second vertical triangle in 
    \eqref{eq:sigma-truncation-two-parameters}
    yields a similar $3\times 3$-diagram
    which is uniquely isomorphic to the $3\times
    3$-diagram~\eqref{eq:sigma-truncation-two-parameters} by an
    isomorphism extending $\id_X$ (and this is functorial in $X$). 
    The four isomorphisms in the corners are 
    the isomorphism $\sigma^{[b+1,a]}_X$ from the Proposition, the
    (inverse of) the isomorphism $\sigma^{[a+1,b]}_X$, and the 
    isomorphisms $\sigma_{\leq a}\sigma_{\leq b}
    \sira \sigma_{\leq b}\alpha_{<a}$ and $\sigma_{\geq a+1}\sigma_{\geq
      b+1}\sira \sigma_{\geq b+1}\sigma_{\geq a+1}$.
  \end{itemize}
  We use the isomorphisms \eqref{eq:a-geq-b-sigma-isos}
  and \eqref{eq:a-leq-b-sigma-isos} and those from the last point
  sometimes tacitly in the following and write them as equalities.
\end{remark}

We introduce some shorthand notation: For $a,b \in \DZ$ define
$\sigma_{[a,b]}:= \sigma_{\leq b} \sigma_{\geq a}$ (which equals
$\sigma_{\geq a}\sigma_{\leq b}$ by the above convention) and
$\sigma_a:=\sigma_{[a]}:=\sigma_{[a,a]}$.

We give some commutation formulas:
Applying the triangulated functor $s$ to the triangle $(\sigma_{\geq
  a+1}X, X, \sigma_{\leq a}X)$ yields a triangle that is uniquely
isomorphic to $(\sigma_{\geq a+2}s(X), s(X), \sigma_{\leq
  a+1}s(X))$. Hence we obtain isomorphisms (that we write as
equalities) 
\begin{equation}
  \label{eq:sgle-eq-s-comm}
  s\sigma_{\geq a} = \sigma_{\geq a+1}s, \quad
  s\sigma_{\leq a} = \sigma_{\leq a+1}s, \quad
  s\sigma_{a} = \sigma_{a+1}s.
\end{equation}

Let $(\tildew{\mathcal{T}}, i)$ be an f-category over a triangulated
category $\mathcal{T}$. Define
\begin{equation*}
  \gr^n:= i\inv s^{-n} \sigma_{n}:
  \tildew{\mathcal{T}} \ra \mathcal{T}
\end{equation*}
where $i\inv$ is a fixed quasi-inverse of $i$.
Note that $\gr^n$ is a triangulated functor. 
From 
$  \gr^{n+1} s 
  = i\inv s^{-n-1}\sigma_{n+1}s
  = i\inv s^{-n}\sigma_{n}
  = \gr^n
$
we obtain
\begin{equation}
  \label{eq:gr-s-comm}
  \gr^{n+1} s= \gr^n.
\end{equation}

Given $X \in \tildew{\mathcal{T}}$ we define its \define{support} 
by
\begin{equation*}
  \supp(X) :=\{n \in \DZ \mid \sigma_n(X) \not=0\}.
\end{equation*}
Note that $\supp(X)$ is a bounded subset of $\DZ$
by axiom \ref{enum:filt-tria-cat-exhaust} and
Proposition~\ref{p:firstprop-filt-cat},
\eqref{enum:firstprop-filt-cat-trunc-preserve}. It is empty if and only
if $X=0$.
The \define{range} $\range(X)$ of $X$ is defined as the smallest
interval (possibly empty) in $\DZ$ containing $\supp X$.
It is the smallest interval $I$ such that $X \in
\tildew{\mathcal{T}}(I)$.
The \define{length} $l(X)$ of $X$ is the number of
elements in $\range(X)$.

\subsection{Forgetting the filtration -- the functor
  \texorpdfstring{$\omega$}{omega}} 
\label{sec:forget-fitlration-functor-omega}

\begin{proposition}
  [{cf.\ \cite[Prop.~A 3]{Beilinson} (without proof)}]
  \label{p:functor-omega}
  Let $(\tildew{\mathcal{T}}, i)$ be an f-category over a triangulated
  category $\mathcal{T}$. There is a unique (up to unique isomorphism)
  triangulated functor
  \begin{equation*}
    \omega: \tildew{\mathcal{T}} \ra \mathcal{T}
  \end{equation*}
  such that
  \begin{enumerate}[label=(om{\arabic*})]
  \item 
    \label{enum:functor-omega-i}
    The restriction $\omega|_{\tildew{\mathcal{T}}(\leq 0)}:
    \tildew{\mathcal{T}}(\leq 0) \ra \mathcal{T}$ is
    left\footnote{typo in \cite[Prop.~A 3]{Beilinson} (and similar in
    \ref{enum:functor-omega-ii})}
    adjoint to
    $i: \mathcal{T} \ra \tildew{\mathcal{T}}(\leq 0)$.
  \item 
    \label{enum:functor-omega-ii}
    The restriction $\omega|_{\tildew{\mathcal{T}}(\geq 0)}:
    \tildew{\mathcal{T}}(\geq 0) \ra \mathcal{T}$ is
    right
    adjoint to
    $i: \mathcal{T} \ra \tildew{\mathcal{T}}(\geq 0)$.
  \item 
    \label{enum:functor-omega-iii}
    For any $X$ in $\tildew{\mathcal{T}}$, the arrow $\alpha_X: X \ra
    s(X)$ is mapped to an isomorphism $\omega(\alpha_X):\omega(X) \sira
    \omega(s(X))$.
  \item 
    \label{enum:functor-omega-iv}
    For all $X$ in $\tildew{\mathcal{T}}(\leq 0)$ and all $Y$ in
    $\tildew{\mathcal{T}}(\geq 0)$ we have an isomorphism
      \begin{equation}
        \label{eq:functor-omega-iv}
        \omega: \Hom_{\tildew{\mathcal{T}}}(X,Y) \sira
        \Hom_{\mathcal{T}}(\omega(X), \omega(Y)).
      \end{equation}
  \end{enumerate}
  In fact $\omega$ is uniquely determined by the properties
  \ref{enum:functor-omega-i} and
  \ref{enum:functor-omega-iii}
  (or by the properties
  \ref{enum:functor-omega-ii} and
  \ref{enum:functor-omega-iii}).
\end{proposition}

The functor $\omega$ is often called the ``forgetting of the filtration
functor"; the
reason for this name is Lemma~\ref{l:omega-in-basic-example} below.
Note that 
$\omega|_{\tildew{\mathcal{T}}([0])}$ is left (and right) adjoint to
the equivalence $i:\mathcal{T} \sira \tildew{\mathcal{T}}([0])$
and hence a quasi-inverse of $i$.

\begin{proof}
  \textbf{Uniqueness:}
  Assume that $\omega: \tildew{\mathcal{T}} \ra
  \mathcal{T}$ satisfies the conditions
  \ref{enum:functor-omega-i} and
  \ref{enum:functor-omega-iii}.

  Let $f: X \ra Y$ be a morphism in $\tildew{\mathcal{T}}$. By 
  \ref{enum:filt-tria-cat-exhaust} there is an $n \in \DZ$ such that
  this is a morphism in $\tildew{\mathcal{T}}(\leq n)$. 
  By \ref{enum:filt-tria-cat-shift} we can assume that $n \geq 0$.
  Consider the commutative diagram
  \begin{equation*}
    \hspace{-1.4cm}
    \xymatrix@C4pc{
      {s^{-n}(X)} \ar[r]^-{\alpha} \ar[d]^-{s^{-n}(f)} &
      {s^{-n+1} (X)} \ar[r]^-{\alpha}
      \ar[d]^-{s^{-n+1}(f)} &
      {\dots} &
      {\dots} \ar[r]^-{\alpha} &
      {s^{-1}(X)}  \ar[r]^-{\alpha}
      \ar[d]^-{s^{-1}(f)} &
      {X}  \ar[d]^-{f} \\
      {s^{-n}(Y)} \ar[r]^-{\alpha} &
      {s^{-n+1} (Y)} \ar[r]^-{\alpha}
      &
      {\dots} &
      {\dots} \ar[r]^-{\alpha} &
      {s^{-1}(Y)}  \ar[r]^-{\alpha} &
      {Y}
    }
  \end{equation*}
  where we omit the indices $s^{-i}(X)$ and $s^{-i}(Y)$ at the various
  maps $\alpha$. 
  If we apply $\omega$, all the horizontal arrows become
  isomorphisms; note that $s^{-n}(f)$ is a morphism in
  $\tildew{\mathcal{T}}(\leq 0)$. This shows that $\omega$ is uniquely
  determined by its restriction to $\tildew{\mathcal{T}}(\leq 0)$ and
  knowledge of all isomorphisms $\omega(\alpha_Z): \omega(Z) \sira
  \omega(s(Z))$ for all $Z$ in $\tildew{\mathcal{T}}$.
  
  If we have two functors $\omega_1$ and $\omega_2$ whose restrictions
  are both adjoint to the inclusion $i: \mathcal{T} \ra 
  \tildew{\mathcal{T}}(\leq 0)$ (i.\,e.\ satisfy
  \ref{enum:functor-omega-i}), these adjunctions give rise to a unique
  isomorphism between $\omega_1|_{\tildew{\mathcal{T}}(\leq 0)}$ and
  $\omega_2|_{\tildew{\mathcal{T}}(\leq 0)}$. If both $\omega_1$ and
  $\omega_2$ satisfy in addition
  \ref{enum:functor-omega-iii}, this isomorphism can be uniquely
  extended to an isomorphism between $\omega_1$ and $\omega_2$ (use
  the above diagram after application of $\omega_1$ and $\omega_2$
  respectively).

  \textbf{Existence:}
  Let $X$ in $\tildew{\mathcal{T}}$. We define objects $X_l$ and $X_r$
  as follows: If $X=0$ let $X_l=X_r=0$. 
  If $X\not= 0$ let $a=a^X, b=b^X \in \DZ$ such that $\range(X)=[a,b]$
  (and $l(X)=b-a+1$);
  let $X_l := s^{-b}(X)$ and $X_r := s^{-a}(X)$ (the indices stand for
  left and right); observe that $X_l \in \tildew{\mathcal{T}}(\leq 0)
  \not\ni s(X_l)$ and $X_r \in \tildew{\mathcal{T}}(\geq 0) \not\ni
  s\inv(X_r)$.

  We denote the composition 
  $X_l \xra{\alpha_{X_l}} s(X_l) \ra \dots \xra{\alpha_{s\inv(X_r)}}
  X_r$ 
  by
  $\alpha_{rl}^X:X_l \ra X_r$.
  Our first goal is to construct for every $X$ in
  $\tildew{\mathcal{T}}$ an object $\Omega_X$ in $\mathcal{T}$ and a
  factorization of $\alpha^X_{rl}$
  \begin{equation*}
    \xymatrix{
      {X_l} \ar[r]^-{\epsilon_X} \ar@/^2pc/[rr]^-{\alpha^X_{rl}} &
      {i(\Omega_X)} \ar[r]^-{\delta_X} &
      {X_r}
    }
  \end{equation*}
  such that 
  \begin{align}
    \label{eq:univ-epsi}
    \Hom(i(\Omega_X), \tildew{A}) \xrightarrow[\sim]{? \comp
      \epsilon_X} \Hom(X_l, \tildew{A}) && \text{for all $\tildew{A}$
      in $\tildew{\mathcal{T}}(\geq 0)$;}\\
    \label{eq:univ-delt}
    \Hom(\tildew{B}, i(\Omega_X)) \xrightarrow[\sim]{\delta_X \comp ?}
    \Hom(\tildew{B}, X_r) && \text{for all $\tildew{B}$
      in $\tildew{\mathcal{T}}(\leq 0)$.}
  \end{align}
  Here is a diagram to illustrate this:
  \begin{equation*}
    \xymatrix{
      & {\tildew{B}} \ar[d]^-b \ar[rd]^-{b'}\\
      {X_l} \ar[r]^-{\epsilon_X} \ar[rd]^-{a'} & 
      {i(\Omega_X)} \ar[r]^-{\delta_X} \ar[d]^-{a} &
      {X_r}\\
      & {\tildew{A}}
    }
  \end{equation*}
  If $X$ is in $\tildew{\mathcal{T}}(\leq 0)$ the composition
  $X \xra{\alpha_X} s(X) \ra \dots \xra{\alpha_{s\inv(X_l)}} X_l$
  is denoted
  $\alpha_l^X:X \ra X_l$. If \eqref{eq:univ-epsi} holds then 
  \ref{enum:filt-tria-cat-hom-bij} shows:
  \begin{align}
    \label{eq:univ-epsi-X}
    \Hom(i(\Omega_X), \tildew{A}) \xrightarrow[\sim]{? \comp
      \epsilon_X\comp \alpha^X_l} \Hom(X, \tildew{A}) &&
    \text{for all $X$ is in $\tildew{\mathcal{T}}(\leq 0)$} \\
    \notag
    && \text{and $\tildew{A}$ 
      in $\tildew{\mathcal{T}}(\geq 0)$.}
  \end{align}
  If $X$ is in
  $\tildew{\mathcal{T}}(\geq 0)$
  and \eqref{eq:univ-delt} holds, we
  similarly have a morphism $\alpha^X_r: X_r \ra X$ and obtain:
  \begin{align}
    \label{eq:univ-delt-X}
    \Hom(\tildew{B}, i(\Omega_X)) \xrightarrow[\sim]{\alpha^X_r \comp
      \delta_X \comp ?} 
    \Hom(\tildew{B}, X) && 
    \text{for all $X \in \tildew{\mathcal{T}}(\geq 0)$}
    \\
    \notag
    && \text{and $\tildew{B}$
      in $\tildew{\mathcal{T}}(\leq 0)$.}
  \end{align}
  
  We construct the triples $(\Omega_X, \epsilon_X, \delta_X)$
  by induction over the length $l(X)$ of $X$. The case $l(X)=0$ is trivial.

  Base case $l(X)=1$.
  Then $X_l=X_r \in \tildew{\mathcal{T}}([0])$.  Let $\kappa$ be a
  quasi-inverse
  of $i$ and fix an isomorphism
  $\theta: \id \sira i\kappa$. 
  We define
  $\Omega_X:= \kappa(X_l)$ and
  \begin{equation*}
    X_l \xrightarrow[\sim]{\epsilon_X:=\theta_{X_l}} i(\Omega_X)
    \xrightarrow[\sim]{\delta_X:=\theta_{X_r}\inv} X_r.
  \end{equation*}
  This is a factorization of $\alpha_{rl}^X=\id_X$ and the conditions
  \eqref{eq:univ-epsi} and
  \eqref{eq:univ-delt} are obviously satisfied.

  Now let $X \in \tildew{\mathcal{T}}$ be given with $l(X) > 1$ and
  assume that we have 
  constructed $(\Omega_Y, \epsilon_Y, \delta_Y)$ as desired for all
  $Y$ with $l(Y) < l(X)$. Let $L:=l(X)-1$. 
  Let $P:=\sigma_{\geq 0}(X_l)$ and $Q:= \sigma_{\leq -1}(X_l)$ and
  let us explain the following diagram:
  \begin{equation*}
    \hspace{-1cm}
    \xymatrix{
      {P=P_l} \ar[r] \ar[dd]^-{\epsilon_{P}}_-{\sim} &
      {X_l} \ar[r] \ar@{..>}[dd]^-{\epsilon_X} &
      {Q} \ar[r]^-h 
      \ar[d]^-{\alpha^{Q}_l}
      &
      {[1] P} \ar[dd]^-{[1]\epsilon_{P}}_-\sim \\
      && {s^{-b_Q}(Q)=Q_l}
      \ar[d]^-{\epsilon_{Q}}\\ 
      {i(\Omega_{P})} \ar[d]^-{\delta_{P}}_-{\sim} \ar@{..>}[r]
      &
      {i(\Omega_X)} \ar@{..>}[r] \ar@{~>}[dd]^-{\delta_X}
      & {i(\Omega_{Q})} \ar@{..>}[r]^-{h'}
      \ar[dd]^-{\delta_{Q}} &
      {[1]i(\Omega_{P})} \ar[d]^-{[1]\delta_{P}}_-{\sim} \\
      {P=P_r} \ar[d]^-{\alpha^L} && 
      &
      {[1]P} \ar[d]^-{[1]\alpha^L} 
      \\
      {s^L(P)} \ar[r] &
      {s^L(X_l)=X_r} \ar[r] &
      {s^L(Q)=Q_r} \ar[r]^-{s^L(h)} 
      &
      {[1] s^L(P)} 
    }
  \end{equation*}
  The first row is the $\sigma$-truncation triangle of type $(\geq 0,
  \leq -1)$ of $X_l \in \tildew{\mathcal{T}}([-L,0])$
  The last row is the image of this triangle under the triangulated
  automorphism $s^L$. 
  There is a morphism between these triangles given by
  $\alpha^L$. 
  Observe that $\range(P)=[0]$ and $\range(Q) = [-L,b_Q] \subset [-L,-1]$; in
  particular the triples 
  $(\Omega_P, \epsilon_P, \delta_P)$ and
  $(\Omega_Q, \epsilon_Q, \delta_Q)$ are constructed.
  Since
  $\delta\comp
  \epsilon$ is a factorization of $\alpha_{rl}$, the
  first, third and fourth column are components of this morphism
  $\alpha^L$ of triangles.
  By \eqref{eq:univ-epsi-X} there is a unique morphism $h'$ as
  indicated making the upper right square commutative, and then
  \eqref{eq:univ-epsi-X} again shows that the lower right square
  commutes.
  We can find some $\Omega_X \in \mathcal{T}$ and a completion of $h'$
  to a triangle as shown in the middle row of the above diagram.
  Then we complete the partial morphism of the two upper triangles
  by $\epsilon_X$ to a morphism of triangles.

  Let us construct the wiggly morphism $\delta_X$.
  Take an arbitrary object $\tildew{A}$ in $\tildew{\mathcal{T}}(\geq
  0)$ and apply 
  $\Hom(?,\tildew{A})$
  to the
  morphism of the upper two triangles. The resulting morphism of long exact
  sequences and the five lemma show that 
  \eqref{eq:univ-epsi} holds for $\epsilon_X$.
  This shows in particular that the morphism $\alpha_{rl}^X:X_l \ra
  X_r$ factorizes 
  uniquely over $\epsilon_X$ by some $\delta_X$. Using 
  \eqref{eq:univ-epsi} again we see that $\delta_X$ defines in fact a
  morphism of the two lower triangles.

  Finally we apply for any $\tildew{B}$ in 
  $\tildew{\mathcal{T}}(\leq 0)$
  the functor 
  $\Hom(\tildew{B}, ?)$ 
  to the
  morphism of the two lower triangles; the five lemma again shows
  that $\delta_X$ satisfies \eqref{eq:univ-delt}.
  This shows that the triple $(\Omega_X, \epsilon_X, \delta_X)$ has
  the properties we want.

  Now we define the functor $\omega: \tildew{\mathcal{T}} \ra
  \mathcal{T}$. On objects we define $\omega(X):=\Omega_X$. Let $f:X
  \ra Y$ be a morphism. Then for some $N \in \DN$ big enough we have
  $s^{-N}(X), s^{-N}(Y) \in \tildew{\mathcal{T}}(\leq 0)$, hence
  $s^{-N}(f)$ is a
  morphism in $\tildew{\mathcal{T}}(\leq 0)$. Similarly for some $M
  \in \DN$
  big enough $s^M(f)$ is a morphism in $\tildew{\mathcal{T}}(\geq 0)$.
  Consider the diagram
  \begin{equation}
    \label{eq:omega-on-morphisms}
    \xymatrix@C4pc{
      {{s^{-N}(X)}} \ar[r]^-{\alpha^?} \ar[d]^-{s^{-N}(f)} &
      {X_l} \ar[r]^-{\epsilon_X} \ar@/^2pc/[rr]^-{\alpha_{rl}^X}
      &
      {i(\omega(X))} \ar[r]^-{\delta_X} \ar@{..>}[d]^-{\Omega_f} & {X_r} 
      \ar[r]^-{\alpha^?} &
      {s^M(X)} \ar[d]^-{s^{M}(f)}
      \\
      {{s^{-N}(Y)}} \ar[r]^-{\alpha^?} &
      {Y_l} \ar[r]^-{\epsilon_Y} \ar@/_2pc/[rr]^-{\alpha_{rl}^Y}
      &
      {i(\omega(Y))} \ar[r]^-{\delta_Y} & {Y_r} 
      \ar[r]^-{\alpha^?} &
      {s^M(Y).}
    }
  \end{equation}
  The morphism from $S^{-N}(X)$ to $i(\omega(Y))$ factors uniquely to
  the dotted arrow by \ref{enum:filt-tria-cat-hom-bij} and
  \eqref{eq:univ-epsi}. Similarly the morphism from $i(\omega(X))$ to
  $s^M(Y)$ factors to the dotted
  arrow. These (a priori) two dotted arrows coincide, since 
  \begin{equation}
    \label{eq:i-omega-isom}
    \Hom(i(\omega(X)),i(\omega(Y))\sira \Hom(s^{-N}(X), s^{M}(Y))
  \end{equation}
  by  \ref{enum:filt-tria-cat-hom-bij} and
  \eqref{eq:univ-epsi} and \eqref{eq:univ-delt} again and since
  both compositions $\delta\comp\epsilon$ are a factorization of
  $\alpha_{rl}$. 
  Note that $\Omega_f$ does not depend on the choice of $N$ and $M$.
  We define $\omega(f)$ to be the unique arrow $\omega(X) \ra
  \omega(Y)$ that is mapped under the equivalence $i$ to $\Omega_f$.
  This respects compositions and
  identity morphisms and hence defines a functor $\omega$.

  We verify \ref{enum:functor-omega-i}-\ref{enum:functor-omega-iv} and
  that $\omega$ can be made into a triangulated functor.

  Let us show \ref{enum:functor-omega-iii} first:
  We can assume that $X\not=0$.
  We draw the left part of diagram
  \eqref{eq:omega-on-morphisms} for the morphism $\alpha_X:X \ra s(X)$
  (note that $b_X+1=b_{s(X)}$ and hence $X_l=(s(X))_l$):
  \begin{equation*}
    \xymatrix@C4pc{
      {{s^{-N}(X)}} \ar[r]^-{\alpha^?}
      \ar[d]_{s^{-N}(\alpha_X)} & 
      {X_l=s^{-{b_X}}(X)} \ar[r]^-{\epsilon_X}
      \ar@{..>}[d]^-{\id}
      &
      {i(\omega(X))} \ar@{~>}[d]^-{\Omega_{\alpha_X}}
      \\
      {{s^{-N+1}(X)}} \ar[r]^-{\alpha^{?-1}} &
      {(s(X))_l=s^{-{b_{s(X)}}}(X)} \ar[r]^-{\epsilon_{s(X)}} 
      &
      {i(\omega(s(X)))}
    }
  \end{equation*}
  By \ref{enum:filt-tria-cat-shift-alpha}
  we have $s^{-N}(\alpha_X)=\alpha_{s^{-N}(X)}$ and hence
  the left square becomes commutative with the dotted arrow.
  Since $\epsilon_X$ and $\epsilon_{s(X)}$ have the same
  universal property \eqref{eq:univ-epsi} the morphism $\Omega_{\alpha_X}$ must
  be an isomorphism.

  Let us show \ref{enum:functor-omega-i}:
  Given $X$ in $\tildew{\mathcal{T}}(\leq 0)$ and $A$ in $\mathcal{T}$
  replace $\tildew{A}$ by $i(A)$ in 
  \eqref{eq:univ-epsi-X} and use $i:\Hom(\omega(X),A)\sira
  \Hom(i(\omega(X)),i(A))$. This gives
  \begin{equation*}
    \Hom_{\mathcal{T}}(\omega(X), A) \xsira{i(?) \comp \epsilon_X
      \comp \alpha_l^X} 
    \Hom_{\tildew{\mathcal{T}}}(X, i(A)).
  \end{equation*}
  From \eqref{eq:omega-on-morphisms} it is easy to see that this
  isomorphism is compatible with morphisms $f:X \ra Y$ in
  $\tildew{\mathcal{T}}(\leq 0)$ and $g:A' \ra A$ in $\mathcal{T}$.

  The proof of \ref{enum:functor-omega-ii} is similar.

  Proposition~\ref{p:linksad-triang-for-wstr-art} shows that
  $\omega|_{\tildew{\mathcal{T}}(\leq 0)}$ is triangulated for a
  suitable isomorphism $\phi: \omega|_{\tildew{\mathcal{T}}(\leq 0)}
  [1] \sira [1] \omega|_{\tildew{\mathcal{T}}(\leq 0)}$. 
  Using the above techniques it is easy to find an isomorphism $\phi:
  \omega[1] \sira [1]\omega$ such that $(\omega, \phi)$ is a
  triangulated functor.
  We leave the details to the reader.

  We finally prove \ref{enum:functor-omega-iv}.
  Let $X$ in
  $\tildew{\mathcal{T}}(\leq 0)$ and $Y$ in $\tildew{\mathcal{T}}(\geq
  0)$. The composition
  \begin{align*}
    \Hom_\mathcal{T}(\omega(X),\omega(Y)) \xsira{i} &
    \, \Hom_{\tildew{\mathcal{T}}}(i(\omega(X)),i(\omega(Y)))  \\
    \xsira{\eqref{eq:i-omega-isom}} &
    \, \Hom_{\tildew{\mathcal{T}}}(s^{-N}(X), s^{M}(Y)) \\
    \xleftarrow[\sim]{\alpha^M \comp ? \comp
      \alpha^N} & 
    \, \Hom_{\tildew{\mathcal{T}}}(X, Y).
  \end{align*}
  of isomorphisms (the last isomorphism comes from
  \ref{enum:filt-tria-cat-hom-bij}) is easily seen to be inverse to
  \eqref{eq:functor-omega-iv}.
\end{proof}

\subsection{Omega in the basic example}
\label{sec:omega-in-basic-example}

Let $\mathcal{A}$ be an abelian category and consider the basic example
of the f-category $DF(\mathcal{A})$ over $D(\mathcal{A})$
(as described in Section~\ref{sec:basic-example-fcat},
Proposition~\ref{p:basic-ex-f-cat}).  
Let $\omega: CF(\mathcal{A}) \ra C(\mathcal{A})$ be the functor
mapping a filtered complex $X=(\ul{X},F)$ to its underlying
non-filtered complex $\ul{X}$.
This functor obviously induces a triangulated functor $\omega:
DF(\mathcal{A}) \ra D(\mathcal{A})$.

\begin{lemma}
  \label{l:omega-in-basic-example}
  The functor $\omega:DF(\mathcal{A}) \ra D(\mathcal{A})$ satisfies
  the conditions 
  \ref{enum:functor-omega-iii} and \ref{enum:functor-omega-i}
  of Proposition~\ref{p:functor-omega}.
\end{lemma}

\begin{proof}
  Condition \ref{enum:functor-omega-iii} is obvious, so let us show
  condition \ref{enum:functor-omega-i}.

  We abbreviate $\omega':= \omega|_{DF(\mathcal{A})(\leq 0)}$ 
  and consider $i$ as a functor to $DF(\mathcal{A})(\leq 0)$.)
  We first define a morphism 
$    \id_{DF(\mathcal{A})(\leq 0)} \xra{\epsilon} i \comp
    \omega'$
  of functors as follows. 
  Let $X$ be in ${DF(\mathcal{A})(\leq 0)}$. We have seen that the obvious
  morphism
  $X \ra X/(X(\geq 1))$
  (cf.\ \eqref{eq:basic-example-fcat-def-Xgeq-triangle})
  is an isomorphism in $DF(\mathcal{A})$.
  As above we denote the filtration of $X$ by 
  $\dots \supset F^i \supset F^{i+1} \supset \dots$
  and its underlying non-filtered
  complex $\omega(X)$ by $\underline X$.
  Define $\epsilon_X$ by the commutativity of the
  following diagram
  \begin{equation*}
    \hspace{-1cm}
    \xymatrix{
      {X} \ar[d]^-{\sim} \ar@{..>}[r]^-{\epsilon_X} & {i(\omega(X))=
        i(\ul{X}) \ar[d]^-{\sim}}
      \ar[d]^-{\sim} 
      \\
      {X/X(\geq 1)} \ar[r] \gar[d] & {i(\omega(X/X(\geq 1)))=
        i(\ul{X}/F^1)} \gar[d]
      \\
      \mathovalbox{\ul{X}/F^1 : F^{-1}/F^1 \supset F^0/F^1 \supset 0} &     
      \mathovalbox{\ul{X}/F^1 : \ul{X}/F^1 = \ul{X}/F^1 \supset 0} &     
    }
  \end{equation*}
  where the lower horizontal map is the obvious one. 
  This is compatible with morphisms and defines $\epsilon$.
  Let $\delta: \omega' \comp i \ra \id_{D(\mathcal{A})}$ be the identity
  morphism.
  We leave it to the reader to check that
  $\delta \omega' \comp \omega' \epsilon = \id_{\omega'}$ and
  $i \delta \comp \epsilon i = \id_{i}$.
  This implies that $\epsilon$ and $\delta$ are unit and counit of an
  adjunction $(\omega', i)$. 
\end{proof}

\subsection{Construction of the functor \texorpdfstring{$c$}{c}}
\label{sec:construction-functor-cfun}

Let $(\tildew{\mathcal{T}},i)$ be an f-category over a triangulated
category $\mathcal{T}$. 
Our aim in this section is to construct a certain functor
$c:\tildew{\mathcal{T}} \ra C^b(\mathcal{T})$.
The construction is in two steps.
It can be found in a very condensed form in 
\cite[Prop.~A 5]{Beilinson} or
\cite[8.4]{bondarko-weight-str-vs-t-str}.

\subsubsection{First step}
\label{sec:first-step-cprime}

We proceed similar as in the
construction of the weak weight complex functor (see
Section~\ref{sec:weak-wc-fun}). 
This may seem a bit involved 
(compare to the second approach explained later on)
but will turn 
out to be convenient 
when showing that the strong weight complex functor is
a lift of the weak one.

For every object $X$ in
$\tildew{\mathcal{T}}$ we have functorial $\sigma$-truncation
triangles (for all $n \in \DZ$)
\begin{equation}
  \label{eq:sigma-trunc-construction-c}
  \xymatrix{
    {S^n_X:} &
    {\sigma_{\geq n+1}X} \ar[r]^-{g^{n+1}_X} & 
    {X} \ar[r]^-{k_X^n} &
    {\sigma_{\leq n}X} \ar[r]^-{v_X^n} &
    {[1]\sigma_{\geq n+1}X.}
  }
\end{equation}
For any $n$ there is a unique morphism of triangles 
$S^n_X \ra S^{n-1}_X$ extending $\id_X$
(use Prop.~\ref{p:BBD-1-1-9-copied-for-w-str} and
\ref{enum:filt-tria-cat-no-homs}): 
\begin{equation}
  \label{eq:unique-triang-morph-Snx}
  \xymatrix{
    {S^n_X:} \ar[d]|{(h^n_X,\id_X,l^n_X)} &
    {\sigma_{\geq n+1}X} \ar[r]^-{g^{n+1}_X} \ar[d]^-{h^n_X}
    \ar@{}[rd]|{\triangle} &  
    {X} \ar[r]^-{k^n_X} \ar[d]^-{\id_X} 
    \ar@{}[rd]|{\nabla} &
    {\sigma_{\leq n}X} \ar[r]^-{v^n_X} \ar[d]^-{l^n_X} &
    {[1]\sigma_{\geq n+1}X} \ar[d]^-{[1]h^n_X} \\
    {S^{n-1}_X:} &
    {\sigma_{\geq n}X} \ar[r]^-{g^n_X} & 
    {X} \ar[r]^-{k^{n-1}_X} &
    {\sigma_{\leq n-1}X} \ar[r]^-{v^{n-1}_X} &
    {[1]\sigma_{\geq n}X}
  }  
\end{equation}
More precisely $h^n_X$ (resp.\ $l^n_X$) is the unique morphism making the square
$\Delta$ (resp.\ $\nabla$) commutative.
It is easy to see that $h^n_X$ corresponds under 
\eqref{eq:a-leq-b-sigma-isos} to the adjunction morphism
$g^{n+1}_{\sigma_{\geq n}X}:\sigma_{\geq
n+1}\sigma_{\geq n}X \ra \sigma_{\geq n}X$. 
Hence we see from
Proposition~\ref{p:firstprop-filt-cat}
\eqref{enum:firstprop-filt-cat-trunc-triangle} that there is 
a
unique morphism  
$c^n_X$
such
that
\begin{equation}
  \label{eq:triang-prime-sigmaX}
  S^{'n}_{\sigma_{\geq n}X}: 
  \xymatrix{
    {\sigma_{\geq n+1}X} \ar[r]^-{h^n_X} &
    {\sigma_{\geq n}X} \ar[rr]^-{e^n_X:=k^n_{\sigma_{\geq n}X}} &&
    {\sigma_{n}X} \ar[r]^-{c^n_X} &
    {[1]\sigma_{\geq n+1}X}
  }
\end{equation}
is a triangle.
We use the 
the tree triangles in \eqref{eq:unique-triang-morph-Snx} and
\eqref{eq:triang-prime-sigmaX} and the square marked with $\triangle$ 
in 
\eqref{eq:unique-triang-morph-Snx}
as the germ cell and obtain (from the octahedral axiom) the following
octahedron: 
\begin{equation}
  \label{eq:octaeder-cprime}
  \tildew{O}_X^n:=
  \xymatrix@dr{
    && {[1]{\sigma_{\geq n+1}X}} \ar[r]^-{[1]h^n_X} & {[1]{\sigma_{\geq n}X}}
    \ar[r]^-{[1]e^n_X} & {[1]{{\sigma_n X}}} 
    \\ 
    {S^{''n}_{\sigma_{\leq n X}}:} 
    &{{\sigma_n X}} \ar@(ur,ul)[ru]^-{c^n_X}
    \ar@{..>}[r]^-{a^n_X} & 
    {\sigma_{\leq n} X} \ar@{..>}[r]^-{l^n_X} \ar[u]^-{v^n_X} & 
    {\sigma_{\leq n-1} X} 
    \ar@(dr,dl)[ru]^-{b^n_X} \ar[u]_{v^{n-1}_X}\\ 
    {S^{n-1}_X:} &{\sigma_{\geq n}X} \ar[u]^-{e^n_X} \ar[r]^-{g^n_X} 
    & {X} \ar@(dr,dl)[ru]^-{k^{n-1}_X} \ar[u]_{k^n_X} \\  
    {S^n_X:} &{\sigma_{\geq n+1}X} \ar[u]^-{h^n_X} \ar@(dr,dl)[ru]^-{g^{n+1}_X}
    \ar@{}[ru]|{\triangle} 
    \\
    & {S^{'n}_{\sigma_{\geq n}X}:}
  }
\end{equation}
Note that the lower dotted morphism 
is in fact $l^n_X$ by
the uniqueness of $l^n_X$ observed above.
From 
Proposition~\ref{p:firstprop-filt-cat}
\eqref{enum:firstprop-filt-cat-trunc-preserve} we obtain that
the upper dotted morphism labeled $a^n_X$ is unique: It is the 
morphism 
$g^{n}_{\sigma_{\leq n}X}$ (more precisely it is the morphism
$g^{n}_{\sigma_{\leq n}X} \comp \sigma_X^{[n,n]}$
in the notation of \eqref{eq:interval-isom-sigma-trunc}).
We see that the triangle
${S^{''n}_{\sigma_{\leq n X}}}$ can be constructed completely
analogous as triangle \eqref{eq:triang-prime-sigmaX} above.

It is now easy to see that $X \mapsto \tildew{O}^n_X$ is in fact functorial.

Let $c'(X)$ be the following complex in $\tildew{\mathcal{T}}$: Its
$n$-th term is 
\begin{equation*}
  c'(X)^n:=[n]{\sigma_n X}  
\end{equation*}
and the 
differential $d^n_{c'(X)}:[n]{\sigma_n X} \ra [n+1]\sigma_{n+1}X$ is defined by
\begin{multline}
  \label{eq:differential-cprime-complex}
  d^n_{c'(X)}
  := [n](b^{n+1}_X \comp a^n_X)
  =  [n](([1]e^{n+1}_X) \comp v^n_X \comp a^n_X)\\
  =  [n](([1]e^{n+1}_X) \comp c^n_X).
\end{multline}
Note that $d^n_{c'(X)} \comp d^{n-1}_{c'(X)}=0$ 
since 
the composition of two consecutive
maps in a triangle is zero (apply this to \eqref{eq:octaeder-cprime}).

Since everything is functorial we have in fact defined a functor
\begin{equation*}
  c':\tildew{\mathcal{T}} \ra C(\tildew{\mathcal{T}}).
\end{equation*}

Let $C^b_{\Delta}(\tildew{\mathcal{T}}) \subset
C(\tildew{\mathcal{T}})$ be the full subcategory consisting of objects
$A \in C(\tildew{\mathcal{T}})$ such that $A^n \in
\tildew{\mathcal{T}}([n])$ for all $n \in \DZ$ and $A^n=0$ for $n\gg 0$.

Axiom \ref{enum:filt-tria-cat-exhaust} and
$c'(X)^n= [n]\sigma_n(X)\in \tildew{\mathcal{T}}([n])$ show that we
can view $c'$ as a functor
\begin{equation*}
  c':\tildew{\mathcal{T}} \ra C^b_{\Delta}(\tildew{\mathcal{T}}).
\end{equation*}

In the proof of the fact that a weak weight complex functor is a
functor of additive categories with translation we have used an
octahedron \eqref{eq:wc-weak-octaeder-translation}. A similar
octahedron with the same distribution of signs yields a canonical
isomorphism 
\begin{equation}
  \label{eq:cprime-translation}
  c' \comp [1]s\inv
  \cong
  \Sigma \comp (s\inv)_{C^b(\tildew{\mathcal{T}})} \comp c'
\end{equation}
where we use notation introduced at the end of
section~\ref{sec:homotopy-categories}.
We leave the details to the reader.
Note that $\Sigma$ and $(s\inv)_{C^b(\tildew{\mathcal{T}})}$ commute.

We describe two other ways to obtain the functor $c'$:
Let $X \in \tildew{\mathcal{T}}$. 
Apply $\sigma_{\leq n+1}$ to the triangle 
\eqref{eq:triang-prime-sigmaX}.
This gives the triangle
\begin{equation}
  \label{eq:triang-prime-sigmaX-trunc}
  \xymatrix{
    {\sigma_{n+1}X} \ar[r]
    &
    {\sigma_{[n,n+1]}X} \ar[r]
    &
    {\sigma_{\leq n+1}\sigma_{n}X} \ar[rr]^-{\sigma_{\leq n+1}(c^n_X)} &&
    {[1]\sigma_{n+1}X.}
  }
\end{equation}
The following commutative diagram is obtained by applying the
transformation $k^{n+1}:\id \ra \sigma_{\leq n+1}$ to its upper
horizontal morphism: 
\begin{equation}
  \label{eq:ce-comm-square}
  \xymatrix{
    {\sigma_n X} \ar[rr]^{c_X^n} 
    \ar[d]^{k^{n+1}_{\sigma_n X}}_{\sim}
    && {[1]\sigma_{\geq n+1}X}
    \ar[d]^{[1]e^{n+1}_{X}}
    \\
    {\sigma_{\leq n+1}\sigma_n X} \ar[rr]^{\sigma_{\leq n+1}(c_X^n)} 
    && {[1]\sigma_{n+1}X}
  }
\end{equation}
We use the isomorphism $k^{n+1}_{\sigma_nX}:\sigma_nX \sira
\sigma_{\leq n+1}\sigma_n (X)$ in order to replace the third term 
in
\eqref{eq:triang-prime-sigmaX-trunc}
and obtain (using \eqref{eq:ce-comm-square}) the following
triangle, where $\tildew{d}^n_X:=[1]e^{n+1}_X \comp c^n_X$:
\begin{equation}
  \label{eq:sigma-truncation-for-tilde-d}
  \xymatrix{
    {\sigma_{n+1}X} \ar[r]
    &
    {\sigma_{[n,n+1]}X} \ar[r]
    &
    {\sigma_{n}X} 
    \ar[r]^-{\tildew{d}^n_X}
    &
    {[1]\sigma_{n+1}X}
  }
\end{equation}

Completely analogous $\sigma_{\geq n}$ applied to the triangle 
${S^{''n+1}_{\sigma_{\leq n+1 X}}}$ and the isomorphism
$g^n_{\sigma_{n+1}X}:\sigma_{\geq n}\sigma_{n+1}X \sira
\sigma_{n+1}X$
provide a triangle of the form
\eqref{eq:sigma-truncation-for-tilde-d} with the same third morphism
$\tildew{d}^n_X=b^{n+1}_X \comp a^n_X$ (cf.\
\eqref{eq:differential-cprime-complex}) which presumably is
\eqref{eq:sigma-truncation-for-tilde-d} under the obvious
identifications.  
Note that 
\begin{equation*}
  \dots \ra [n]\sigma_n(X) \xra{[n]\tildew{d}^n} [n+1]\sigma_{n+1}(X)
  \ra \dots
\end{equation*}
is the complex $c'(X)$.
Hence we have described two slightly different (functorial)
constructions of the functor $c'$.

We will use the construction described after
\eqref{eq:sigma-truncation-for-tilde-d} later on and refer to it as
the second approach to $c'$.

\subsubsection{Second step}
\label{sec:second-step-cdoubleprime}

In the second step we define the functor $c$ via a functor
$c'':C^b_\Delta(\tildew{\mathcal{T}}) \ra
C^b(\tildew{\mathcal{T}}([0]))$.
There is a shortcut to $c$ described in Remark~\ref{rem:quick-def-c}
below.

Let $A = (A^n, d_A^n) \in C^b_\Delta(\tildew{\mathcal{T}})$. We draw
this complex horizontally in the following diagram:
\begin{equation}
  \label{eq:shift-differential-to-diagonal}
  \xymatrix{
    &
    {s(A^{-1})} 
    \ar@{..>}[dr]
    \\
    {\dots} \ar[r]
    &
    {A^{-1}} 
    \ar[r]_-{{d}^{-1}_A} \ar[u]^-{\alpha}
    & 
    {A^0} 
    \ar[r]^-{{d}^0_A}
    \ar@{..>}[dr]
    & 
    {A^1} 
    \ar[r]^-{{d}^1_A}
    & 
    {A^2} 
    \ar[r]^-{{d}^2_A}
    &
    {\dots}
    \\
    &&& 
    {s^{-1}(A^1)} 
    \ar[u]^-{\alpha}
    \ar@{..>}[dr]
    \\
    &&&& 
    {s^{-2}(A^2)} 
    \ar[uu]^-{\alpha^2}
  }
\end{equation}
The dotted arrows making everything commutative are uniquely obtained
using \ref{enum:filt-tria-cat-hom-bij}. 
The diagonal is again a bounded complex (i.\,e.\ $d^2=0$ (again by
\ref{enum:filt-tria-cat-hom-bij})), now even in
$\tildew{\mathcal{T}}([0])$.
We denote this complex by $c''(A)$.
This construction in fact defines a functor
$c'':C^b_\Delta(\tildew{\mathcal{T}}) \ra
C^b(\tildew{\mathcal{T}}([0]))$.

It is easy to check that there is a canonical isomorphism
\begin{equation}
  \label{eq:cdoubleprime-translation}
  c'' \comp \Sigma \comp (s\inv)_{C^b(\tildew{\mathcal{T}})} 
  \cong
  \Sigma \comp c''.
\end{equation}

\subsubsection{Definition of the functor \texorpdfstring{$c$}{c}}
\label{sec:def-functor-c}

By fixing a quasi-inverse $i\inv$ to $i$ we now define $c$ to be the
composition
\begin{equation*}
  c: \tildew{\mathcal{T}} \xra{c'} C^b_\Delta(\tildew{\mathcal{T}})
  \xra{c''} C^b(\tildew{\mathcal{T}}([0]))
  \xra{i\inv_{C^b}} C^b(\mathcal{T}). 
\end{equation*}
This functor maps an object $X\in \tildew{\mathcal{T}}$ to the complex
\begin{equation}
  \label{eq:cfun-gr}
  \dots \ra 
  [-1]\gr^{-1}(X) 
  \ra
  \gr^0(X) 
  \ra
  [1]\gr^1(X) 
  \ra
  [2]\gr^2(X) 
  \ra
  \dots.
\end{equation}

\begin{remark}
  \label{rem:quick-def-c}
  An equivalent shorter definition of $c$ would be to define it as
  $\omega_{C^b} \comp c'$.
\end{remark}

Combining the above canonical isomorphisms 
\eqref{eq:cprime-translation}
and
\eqref{eq:cdoubleprime-translation} we obtain:

\begin{proposition}
  \label{p:cfun-is-functor-of-add-cat-trans}
  The functor $c$ constructed above is 
  a functor 
  \begin{equation}
    \label{eq:functor-c-as-cat-with-translation}
    c: (\tildew{\mathcal{T}}, [1]s\inv) \ra (C^b(\mathcal{T}), \Sigma)
  \end{equation}
  of additive categories with translation: On objects we have a
  canonical isomorphism
  \begin{equation}
    \label{eq:c-Sigma-commute}
    c([1]s\inv(X))\cong \Sigma c(X).
  \end{equation}
\end{proposition}

\section{Strong weight complex functor}
\label{sec:strong-WCfun}

Our aim in this section is to show the following Theorem:

\begin{theorem}
  [{cf.\ \cite[8.4]{bondarko-weight-str-vs-t-str}}]
  \label{t:strong-weight-cplx-functor}
  Let $\mathcal{T}$ be a triangulated category with a 
  bounded
  weight structure 
  $w=(\mathcal{T}^{w \leq 0},\mathcal{T}^{w \geq 0})$.
  Let $(\tildew{\mathcal{T}}, i)$ be an f-category over
  $\mathcal{T}$. 
  Assume that $\tildew{\mathcal{T}}$ satisfies 
  axiom~\ref{enum:filt-tria-cat-3x3-diag} stated below.
  Then there is a strong weight complex functor
  \begin{equation*}
    \WCfun: \mathcal{T} \ra K^b(\heart(w))^\anti.
  \end{equation*}
  In particular $\WCfun$ is a functor of triangulated categories.
\end{theorem}

A proof of (a stronger version of) this theorem is sketched in 
\cite[8.4]{bondarko-weight-str-vs-t-str}, where M.~Bondarko
attributes the argument to A.~Beilinson.
When we tried to understand the details we had to impose the
additional 
conditions that 
$\tildew{\mathcal{T}}$ satisfies
axiom~\ref{enum:filt-tria-cat-3x3-diag} and that the weight structure
is bounded. 
We state axiom~\ref{enum:filt-tria-cat-3x3-diag} in
Section~\ref{sec:additional-axiom} and show 
in Section~\ref{sec:additional-axiom-basic-example}
that it is satisfied in the basic example of a filtered derived
category.
Our proof of 
Theorem~\ref{t:strong-weight-cplx-functor} is an elaboration of the ideas of
A.~Beilinson and M.~Bondarko; we sketch the idea of the proof
in Section~\ref{sec:idea-construction-strong-WC-functor} and give the
details in Section~\ref{sec:existence-strong-WCfun}.

\subsection{Idea of the proof}
\label{sec:idea-construction-strong-WC-functor}

Before giving the details let us explain the strategy of the proof
of Theorem~\ref{t:strong-weight-cplx-functor}.

Let $(\tildew{\mathcal{T}},i)$ be an f-category over a triangulated
category $\mathcal{T}$ and 
assume that $w=(\mathcal{T}^{w \leq 0},\mathcal{T}^{w \geq 0})$ is a weight
structure on $\mathcal{T}$.
Its heart $\heart(w)=\mathcal{T}^{w=0}$ is a full subcategory of $\mathcal{T}$,
and hence $C^b(\heart(w)) \subset C^b(\mathcal{T})$
and $K^b(\heart(w)) \subset K^b(\mathcal{T})$ are
full subcategories.
Let $\tildew{\mathcal{T}}^s$ be the full subcategory of 
$\tildew{\mathcal{T}}$ consisting of objects $X \in 
\tildew{\mathcal{T}}$ such that $c(X) \in
C^b(\heart(w))$
where $c$ is the functor constructed in Section~\ref{sec:construction-functor-cfun}: We have a ``pull-back"
diagram 
\begin{equation}
  \label{eq:pb-subtildet}
  \xymatrix{
    {\tildew{\mathcal{T}}} \ar[r]^-{c} 
    & {C^b(\mathcal{T})}\\
    {\tildew{\mathcal{T}}^s} \ar[r]^-{c} 
    \ar@{}[u]|-{\cup}
    & {C^b(\heart(w))} \ar@{}[u]|-{\cup}
  }
\end{equation}
of categories where we denote the lower horizontal functor also by
$c$.
Note that \eqref{eq:c-Sigma-commute} shows that $\tildew{\mathcal{T}}^s$
is stable under $s\inv[1]=[1]s\inv$.

We expand diagram \eqref{eq:pb-subtildet} to
\begin{equation}
  \label{eq:pb-subtildet-expanded}
  \xymatrix{
    {\tildew{\mathcal{T}}} \ar[r]_-{c} \ar@/^1pc/[rr]^-{h}
    & {C^b(\mathcal{T})} \ar[r]_-{\can} 
    & {K^b(\mathcal{T})^\anti}\\
    {\tildew{\mathcal{T}}^s} \ar[r]^-{c} \ar@/_1pc/[rr]_{h}
    \ar@{}[u]|-{\cup}
    & {C^b(\heart(w))} \ar@{}[u]|-{\cup}
    \ar[r]^-{\can} 
    \ar@{}[u]|-{\cup}
    & {K^b(\heart(w))^\anti} \ar@{}[u]|-{\cup}
  }
\end{equation}
and define $h:=\can \comp c$ as indicated.
We have seen in
Proposition~\ref{p:cfun-is-functor-of-add-cat-trans}
that 
$c$ is a functor of additive categories with translation (where
$\tildew{\mathcal{T}}$ (or $\tildew{\mathcal{T}}^s$) are equipped with
the translation $[1]s\inv$); the same is obviously true for $h$.
For this statement the homotopy categories on the right are just
considered as additive categories with translation. In the following
however we view them as triangulated categories with the class of
triangles described in Sections \ref{sec:anti-triangles}
and \ref{sec:homotopy-categories}.

From now on we assume that $\tildew{\mathcal{T}}$ satisfies 
axiom~\ref{enum:filt-tria-cat-3x3-diag} (stated below) and that the
weight structure is bounded.

Consider the following diagram whose dotted arrows will be explained:
\begin{equation}
  \label{eq:omega-subtildeT-exp}
  \xymatrix{
    {\mathcal{T}} & {\tildew{\mathcal{T}}} \ar[r]^-{h} \ar[l]_{\omega}
    & {K^b(\mathcal{T})^\anti}\\
    {\mathcal{T}} \gar[u] 
    & 
    {\tildew{\mathcal{T}}^s} \ar[r]^-{h}
    \ar@{}[u]|-{\cup} \ar[l]_{\omega|_{\tildew{\mathcal{T}}^s}} \ar@{..>}[d]^-{\can}
    & {K^b(\heart(w))^\anti} \ar@{}[u]|-{\cup}\\
    & 
    {\mathcal{Q}}
    \ar@{..>}[lu]^-{\ol{\omega}}_{\sim}
    \ar@{..>}[ru]_{\ol{h}} 
  }
\end{equation}
We will prove below:
The restriction $\omega|_{\tildew{\mathcal{T}}^s}$ factors over some
quotient functor $\tildew{\mathcal{T}}^s \xra{\can} \mathcal{Q}$ and
induces an equivalence $\ol{\omega}:\mathcal{Q} \sira \mathcal{T}$ of
additive categories with translation
(see Prop.~\ref{p:omega-tildeTs-equiv}) where
the translation functor of $\mathcal{Q}$ is induced by $[1]s\inv$.
Transfer of structure turns $\mathcal{Q}$
into a triangulated category; 
its class of triangles can be
explicitly described (cf.\ Lemma~\ref{l:omega-bar-triang}).

On the other hand the functor
$h:\tildew{\mathcal{T}}^s \ra K^b(\heart(w))^\anti$ factors over
$\mathcal{Q}$ to a functor $\ol{h}$ of triangulated categories
(see Cor.~\ref{c:h-factors-triang}).

Let $\WCfun: \mathcal{T} \ra K^b(\heart(w))^\anti$ be the
composition $\ol{h} \comp 
\ol{\omega}\inv$,
where $\ol{\omega}\inv$ is a
quasi-inverse of
$\ol{\omega}$.
(In diagram \eqref{eq:omega-subtildeT-exp} $\WCfun$ is the
composition of  
the dotted arrows.)
Then $\WCfun$ will turn out to be a strong weight complex functor, proving 
Theorem~\ref{t:strong-weight-cplx-functor}.

\subsection{An additional axiom for filtered triangulated categories}
\label{sec:additional-axiom}

Let $\tildew{\mathcal{T}}$ be a 
filtered triangulated category $\tildew{\mathcal{T}}$.
Let $Y$ in $\tildew{\mathcal{T}}$ be an object and consider the
$\sigma$-truncation triangle
\begin{equation*}
  \label{eq:sigma-trunc}
  \xymatrix{
    {S^0_Y:} &
    {\sigma_{\geq 1}Y} \ar[r]^-{g^1_Y} & 
    {Y} \ar[r]^-{k_Y^0} &
    {\sigma_{\leq 0}Y} \ar[r]^-{v_Y^0} &
    {[1]\sigma_{\geq 1}Y.}
  }
\end{equation*}
Applying the morphism $\alpha$ we obtain a morphism of triangles
\begin{equation*}
  \xymatrix{
    {s(S^0_Y):} &
    {s(\sigma_{\geq 1}Y)} \ar[r]^-{s(g^1_Y)} & 
    {s(Y)} \ar[r]^-{s(k_Y^0)} &
    {s(\sigma_{\leq 0}Y)} \ar[r]^-{s(v_Y^0)} &
    {[1]s(\sigma_{\geq 1}Y)}\\
    {S^0_Y:} &
    {\sigma_{\geq 1}Y} \ar[r]^-{g^1_Y} \ar[u]^{\alpha_{\sigma_{\geq
          1}(Y)}} & 
    {Y} \ar[r]^-{k_Y^0} \ar[u]^{\alpha_Y} &
    {\sigma_{\leq 0}Y} \ar[r]^-{v_Y^0} \ar[u]^{\alpha_{\sigma_{\leq 0}(Y)}}&
    {[1]\sigma_{\geq 1}Y.} \ar[u]^{[1]\alpha_{\sigma_{\geq 1}(Y)}}
  }
\end{equation*}
where we tacitly identify
$s([1]\sigma_{\geq 1}(Y))= [1]s(\sigma_{\geq 1}(Y))$ and 
$\alpha_{[1]\sigma_{\geq 1}(Y)}=[1]\alpha_{\sigma_{\geq 1}(Y)}$.

Given a morphism $f:X \ra Y$ in $\tildew{\mathcal{T}}$,
the morphism of triangles
\begin{equation*}
  \xymatrix{
    {S^0_Y:} &
    {\sigma_{\geq 1}Y} \ar[r]^-{g^1_Y} & 
    {Y} \ar[r]^-{k_Y^0} &
    {\sigma_{\leq 0}Y} \ar[r]^-{v_Y^0} &
    {[1]\sigma_{\geq 1}Y.} \\
    {S^0_X:} &
    {\sigma_{\geq 1}X} \ar[r]^-{g^1_X} \ar[u]^{\sigma_{\geq 1}(f)} & 
    {X} \ar[r]^-{k_X^0} \ar[u]^{f} &
    {\sigma_{\leq 0}X} \ar[r]^-{v_X^0} \ar[u]^{\sigma_{\leq 0}(f)}&
    {[1]\sigma_{\geq 1}X.} \ar[u]^{[1]\sigma_{\geq 1}(f)}
  }
\end{equation*}
is the unique morphism of triangles extending $f$
(use Prop.~\ref{p:BBD-1-1-9-copied-for-w-str} and
\ref{enum:filt-tria-cat-no-homs}).

We denote the composition of this two morphisms of triangles by
$\alpha \comp f: S_X^0 \ra s(S_Y^0)$:
\begin{equation}
  \label{eq:alpha-f-triangle}
  \hspace{-1.6cm}
  \xymatrix@C2cm{
    {s(S^0_Y):} &
    {s(\sigma_{\geq 1}Y)} \ar[r]^-{s(g^1_Y)} & 
    {s(Y)} \ar[r]^-{s(k_Y^0)} &
    {s(\sigma_{\leq 0}Y)} \ar[r]^-{s(v_Y^0)} &
    {[1]s(\sigma_{\geq 1}Y)}\\
    {S^0_X:} \ar[u]^{\alpha \comp f} &
    {\sigma_{\geq 1}X} \ar[r]^-{g^1_X} \ar[u]^{\alpha_{\sigma_{\geq
          1}(Y)} \comp \sigma_{\geq 1}(f)} & 
    {X} \ar[r]^-{k_X^0} \ar[u]^{\alpha_Y \comp f} &
    {\sigma_{\leq 0}X} \ar[r]^-{v_X^0} \ar[u]^{\alpha_{\sigma_{\leq
          0}(Y)} \comp \sigma_{\leq 0}(f)}&
    {[1]\sigma_{\geq 1}X.} \ar[u]^{[1](\alpha_{\sigma_{\geq 1}(Y)}
      \comp \sigma_{\geq 1}(f))}
  }
\end{equation}
(We don't see a reason why this morphism of triangles 
extending
$\alpha_Y \comp f$ should be unique; 
Proposition~\ref{p:BBD-1-1-9-copied-for-w-str} does not apply.
If it were unique axiom \eqref{enum:filt-tria-cat-3x3-diag} below
would be satisfied automatically.)

Now the additional axiom can be stated:
\begin{enumerate}[label=(fcat{\arabic*}),start=7]
\item 
  \label{enum:filt-tria-cat-3x3-diag}
  For any morphism $f: X \ra Y$ in $\tildew{\mathcal{T}}$ the 
  morphism 
  \eqref{eq:alpha-f-triangle}
  of triangles 
  $\alpha \comp f: S_X^0 \ra s(S_Y^0)$ explained above can be
  extended to a $3 \times 3$-diagram 
  \begin{equation}
    \label{eq:filt-tria-cat-nine-diag}
    \hspace{-1.4cm}
    \xymatrix@C2cm{
      &
      {[1]\sigma_{\geq 1}X} \ar@{..>}[r]^-{[1]g^1_X} & 
      {[1]X} \ar@{..>}[r]^-{[1]k_X^0} &
      {[1]\sigma_{\leq 0}X} \ar@{..>}[r]^-{[1]v_X^0} 
      \ar@{}[rd]|-{\anticomm} &
      {[2]\sigma_{\geq 1}X} \\
      &
      {A} \ar[u]^-{b} \ar[r] & 
      {Z} \ar[u] \ar[r] &
      {B} \ar[u] \ar[r] & 
      {[1]A} \ar@{..>}[u]^-{[1]b} \\
      {s(S^0_Y):} &
      {s(\sigma_{\geq 1}Y)} \ar[r]^-{s(g^1_Y)} \ar[u]^-{a} & 
      {s(Y)} \ar[r]^-{s(k_Y^0)} \ar[u]&
      {s(\sigma_{\leq 0}Y)} \ar[r]^-{s(v_Y^0)} \ar[u]&
      {[1]s(\sigma_{\geq 1}Y)} \ar@{..>}[u]^-{[1]a}\\
      {S^0_X:} \ar[u]^{\alpha \comp f} &
      {\sigma_{\geq 1}X} \ar[r]^-{g^1_X} \ar[u]^{\alpha_{\sigma_{\geq
            1}(Y)} \comp \sigma_{\geq 1}(f)} & 
      {X} \ar[r]^-{k_X^0} \ar[u]^{\alpha_Y \comp f} &
      {\sigma_{\leq 0}X} \ar[r]^-{v_X^0} \ar[u]^{\alpha_{\sigma_{\leq
            0}(Y)} \comp \sigma_{\leq 0}(f)} &
      {[1]\sigma_{\geq 1}X} \ar@{..>}[u]^{[1](\alpha_{\sigma_{\geq 1}(Y)}
        \comp \sigma_{\geq 1}(f))}
    }
  \end{equation}
  having the properties described in 
  Proposition~\ref{p:3x3-diagram-copied-for-w-str}.
\end{enumerate}

Instead of taking the morphism $S_X^0 \xra{\alpha \comp f} s(S^0_Y)$
at the bottom of this $3\times 3$-diagram we get similar diagrams with
morphisms  $S_X^n \xra{\alpha \comp f} s(S^n_Y)$ at the bottom (use
the functor $s$ of triangulated categories).

\begin{remark}
  \label{rem:fcat-nine-nearly-needless}
  We first expected that 
  axiom \ref{enum:filt-tria-cat-3x3-diag} is a consequence of
  Proposition~\ref{p:3x3-diagram-copied-for-w-str}.
  In fact this proposition implies that there is a diagram
  \eqref{eq:filt-tria-cat-nine-diag} with nearly all the required
  properties: If we start from the small commutative square in the
  lower left corner, the only thing that is not clear to us is why
  one can assume that the morphism from $\sigma_{\leq 0}X$ to
  $s(\sigma_{\leq 0}Y)$ is 
  ${\alpha_{\sigma_{\leq 0}(Y)} \comp \sigma_{\leq 0}(f)}$.
\end{remark}

\begin{remark}
  \label{rem:fcat-nine-rewritten}
  Some manipulations of diagram~\ref{eq:filt-tria-cat-nine-diag} (mind
  the signs!) show that
  axiom~\ref{enum:filt-tria-cat-3x3-diag} gives:
  For any morphism $f: X \ra Y$ and any $n \in \DZ$ there is a $3
  \times 3$-diagram of the following\footnote{We assume here and in
    similar situations in the following that $[1][-1]=\id$.} form:
  \begin{equation}
    \label{eq:filt-tria-cat-nine-diag-sigma}
    \xymatrix
    {
      {[-1]s({\sigma_{\geq n+1}(Y)})}
      \ar[r] \ar[d]^-{[-1]s(g^{n+1}_Y)} &  
      {A'} \ar[r] \ar[d] & 
      {{\sigma_{\geq n+1}(X)}} \ar[rr]^-{\alpha_{\sigma_{\geq n+1}(Y)}
        \comp \sigma_{\geq n+1}(f)} \ar[d]^-{g^{n+1}_X} & &
      {s({\sigma_{\geq n+1}(Y)})} \ar@{..>}[d]^-{s({g^{n+1}_Y})} \\
      {[-1]s(Y)} \ar[r] \ar[d]^-{[-1]s({k_Y^n})} &
      {Z'} \ar[r] \ar[d] &
      {X} \ar[rr]^-{\alpha_Y \comp f} \ar[d]^-{k_X^n} &&
      {s(Y)} \ar@{..>}[d]^-{s({k_Y^n})} \\
      {[-1]s({\sigma_{\leq n}(Y)})} \ar[r] \ar[d]^-{-[-1]s({v_Y^n})} & 
      {B'} \ar[r] \ar[d] & 
      {{\sigma_{\leq n}(X)}} 
      \ar[rr]^-{\alpha_{\sigma_{\leq n}(Y)} \comp \sigma_{\leq n}(f)}
      \ar[d]^-{v_X^n} \ar@{}[rrd]|-{\anticomm} &&
      {s({\sigma_{\leq n}(Y)})} \ar@{..>}[d]^-{-s({v_Y^n})} \\
      {s({\sigma_{\geq n+1}(Y)})} \ar@{..>}[r] &
      {[1]A'} \ar@{..>}[r] &
      {[1]{\sigma_{\geq n+1}(X)}}
      \ar@{..>}[rr]^-{[1](\alpha_{\sigma_{\geq n+1}(Y)}\comp
        \sigma_{\geq n+1}(f))} & &
      {[1]s({\sigma_{\geq n+1}(Y)})} 
    }
  \end{equation}
\end{remark}

\subsection{The additional axiom in the basic example}
\label{sec:additional-axiom-basic-example}

Let $\mathcal{A}$ be an abelian category and consider the basic example
$DF(\mathcal{A})$ of a filtered triangulated category
(as described in Section~\ref{sec:basic-example-fcat},
Prop.~\ref{p:basic-ex-f-cat}).  

\begin{lemma}
  \label{l:additional-axiom-true-DFA}
  Axiom~\ref{enum:filt-tria-cat-3x3-diag} is true in
  $DF(\mathcal{A})$.
\end{lemma}

\begin{proof}
  Let $f: X \ra Y$ be a morphism in $DF(\mathcal{A})$. We can assume
  without loss of generality that $f$ is (the class of) a morphism
  $f:X \ra Y$ in $CF(\mathcal{A})$. 
  We explain the following diagram:
  \begin{equation*}
    \hspace{-4.1cm}
    \xymatrix@R=40pt@C=18pt{
      {F^n[1]L_X} \ar@{..>}[r]^-{[1]{g_X}} 
      \ar@{}[rd]|(.2){-x} &
      {F^n[1]X} \ar@{..>}[r]^-{\svek 10} 
      \ar@{}[rd]|(.2){-x} &
      {F^n[1]X \oplus F^n[2]L_X} \ar@{..>}[r]^-{\zvek 01}
      \ar@{}[rd]|-{\anticomm} 
      \ar@{}[rd]|(.2){\tzmat {-x}{-{g_X}}0{x}} & 
      {F^n[2]L_X} 
      \ar@{}[rd]|(.2){x} &
      \\
      {F^{n-1}L_Y\oplus F^n[1]L_X} \ar[u]^-{\zvek 01} \ar[r]^-{\tzmat
        {g_Y}00{g_X}} 
      \ar@{}[rd]|(.2){\tzmat y\beta 0{-x}} & 
      {F^{n-1}Y\oplus F^n[1]X} \ar[u]^-{\zvek 01} \ar[r]^-{
        \Big[
        \begin{smallmatrix}
          1&0\\ 0&0\\ 0&1\\ 0&0
        \end{smallmatrix}
        \Big]
      } 
      \ar@{}[rd]|(.2){\tzmat y\gamma 0{-x}} &
      {F^{n-1}Y \oplus F^{n-1}[1]L_Y \oplus F^n[1]X \oplus F^n[2]L_X}
      \ar[u]^-{
        \Big[
        \begin{smallmatrix}
          0&0&1&0\\ 0&0&0&1
        \end{smallmatrix}
        \Big]
      } \ar[r]^-{
        \Big[
        \begin{smallmatrix}
          0&1&0&0\\ 0&0&0&{-1}
        \end{smallmatrix}
        \Big]
      } 
      \ar@{}[rd]|(.35){
        \Big[
        \begin{smallmatrix}
          y&{g_Y}&\gamma&0\\
          0&-y&0&\beta\\
          0&0&-x&-{g_X}\\
          0&0&0&x
        \end{smallmatrix}
        \Big]
      }  
      & 
      {F^{n-1}[1]L_Y \oplus F^n[2]L_X} \ar@{..>}[u]^-{\zvek 01} 
      \ar@{}[rd]|(.2){\tzmat {-y}{-\beta} 0{x}} & 
      { } \\
      {F^{n-1}L_Y} \ar[u]^-{\svek 10} \ar[r]^-{{g_Y}} 
      \ar@{}[rd]|(.2){y} & 
      {F^{n-1}Y} \ar[u]^-{\svek 10} \ar[r]^-{\svek 10} 
      \ar@{}[rd]|(.2){y} &
      {F^{n-1}Y\oplus F^{n-1}[1]L_Y} \ar[u]^-{
        \Big[
        \begin{smallmatrix}
          1&0\\ 0&1\\ 0&0\\ 0&0
        \end{smallmatrix}
        \Big]
      } \ar[r]^-{\zvek 01} 
      \ar@{}[rd]|(.2){\tzmat y{g_Y}0{-y}} & 
      {F^{n-1}[1]L_Y} \ar@{..>}[u]^-{\svek 10} 
      \ar@{}[rd]|(.2){-y} &
      { } \\
      {F^nL_X} \ar[u]^-{\beta} \ar[r]^-{g_X} 
      \ar@{}[rd]|(.2){x} & 
      {F^nX} \ar[u]^-{\gamma} \ar[r]^-{\svek 10} 
      \ar@{}[rd]|(.2){x} &
      {F^nX \oplus F^n[1]L_X} \ar[u]^-{\tzmat \gamma 0 0 \beta} 
      \ar[r]^-{\zvek 01} 
      \ar@{}[rd]|(.2){\tzmat x{g_X}0{-x}} & 
      {F^n[1]L_X} \ar@{..>}[u]^-{[1]\beta} 
      \ar@{}[rd]|(.2){-x} 
      & { } \\
      { } &
      { } &
      { } &
      { } &
      { } 
    }
  \end{equation*}
  This diagram is the $n$-th filtered part of the $3\times 3$ diagram 
  we need. 
  For simplicity we will in the rest of this description not
  distinguish between a morphism and its $n$-th filtered part.

  The lowest row is 
  constructed as follows: 
  Let $L_X:= X(\geq 1)$ be as defined in
  \ref{eq:basic-example-fcat-def-Xgeq-triangle}
  and let $g_X: L_X \ra X$ be the obvious morphism (called $i$
  there). The lowest row is the ($n$-th filtered part of)
  the mapping cone triangle of this morphism. This triangle is
  isomorphic to the triangle \eqref{eq:triangle-fcat-basic} and hence
  a possible choice 
  for the $\sigma$-truncation triangle of $X$, cf.\ the proof of 
  Proposition~\ref{p:basic-ex-f-cat}.

  The second row from below is the
  corresponding triangle for $s(Y)$.

  The lower right ``index" at each object indicates its differential,
  e.\,g.\ the differential of $X$ is $x$.
  
  The morphism of triangles between the lower two rows 
  is constructed as described before \eqref{eq:alpha-f-triangle},
  e.\,g.\ $\gamma = \alpha_Y \comp f:X \ra s(Y)$.

  Then we fit the morphisms $\beta$, $\gamma$ and $\tzmat \gamma
  00\beta$ into mapping cone triangles. 
  Then consider the horizontal arrows in the second row from above:
  They are morphisms of complexes and make all small squares
  (anti-)commutative as required.

  We only need to show that this second row is a triangle.
  This is a consequence of the following diagram which gives an
  isomorphism of this row to the mapping cone triangle of
  $\tzmat{g_Y}00{g_X}$. 
  \begin{equation*}
    \hspace{-4.4cm}
    \xymatrix@R=60pt@C=23pt{
      {F^{n-1}L_Y\oplus F^n[1]L_X} \gar[d] \ar[r]^-{\tzmat
        {g_Y}00{g_X}} 
      \ar@{}[rd]|(.2){\tzmat y\beta 0{-x}} & 
      {F^{n-1}Y\oplus F^n[1]X} \gar[d] \ar[r]^-{
        \Big[
        \begin{smallmatrix}
          1&0\\ 0&0\\ 0&1\\ 0&0
        \end{smallmatrix}
        \Big]
      } 
      \ar@{}[rd]|(.2){\tzmat y\gamma 0{-x}} &
      {F^{n-1}Y \oplus F^{n-1}[1]L_Y \oplus F^n[1]X \oplus F^n[2]L_X}
      \ar[d]^-{\sim}_-{
        \Big[
        \begin{smallmatrix}
          1&0&0&0\\ 
          0&0&1&0\\ 
          0&1&0&0\\ 
          0&0&0&-1
        \end{smallmatrix}
        \Big]
      } \ar[r]^-{
        \Big[
        \begin{smallmatrix}
          0&1&0&0\\ 0&0&0&{-1}
        \end{smallmatrix}
        \Big]
      } 
      \ar@{}[rd]|(.35){
        \Big[
        \begin{smallmatrix}
          y&{g_Y}&\gamma&0\\
          0&-y&0&\beta\\
          0&0&-x&-{g_X}\\
          0&0&0&x
        \end{smallmatrix}
        \Big]
      }  
      & 
      {F^{n-1}[1]L_Y \oplus F^n[2]L_X} \gar[d]
      \ar@{}[rd]|(.2){\tzmat {-y}{-\beta} 0{x}} & 
      { } \\
      {F^{n-1}L_Y\oplus F^n[1]L_X} \ar[r]^-{\tzmat
        {g_Y}00{g_X}} 
      \ar@{}[rd]|(.2){\tzmat y\beta 0{-x}} & 
      {F^{n-1}Y\oplus F^n[1]X} \ar[r]^-{
        \Big[
        \begin{smallmatrix}
          1&0\\ 0&1\\ 0&0\\ 0&0
        \end{smallmatrix}
        \Big]
      } 
      \ar@{}[rd]|(.2){\tzmat y\gamma 0{-x}} &
      {F^{n-1}Y \oplus F^n[1]X \oplus F^{n-1}[1]L_Y \oplus F^n[2]L_X}
      \ar[r]^-{
        \Big[
        \begin{smallmatrix}
          0&0&1&0\\ 0&0&0&1
        \end{smallmatrix}
        \Big]
      } 
      \ar@{}[rd]|(.35){
        \Big[
        \begin{smallmatrix}
          y&\gamma&{g_Y}&0\\
          0&-x&0&{g_X}\\
          0&0&-y&-\beta\\
          0&0&0&x
        \end{smallmatrix}
        \Big]
      }  
      & 
      {F^{n-1}[1]L_Y \oplus F^n[2]L_X}
      \ar@{}[rd]|(.2){\tzmat {-y}{-\beta} 0{x}} & 
      { } \\
      { } &
      { } &
      { } &
      { } &
      { } 
    }
  \end{equation*}
\end{proof}

\subsection{Existence of a strong weight complex functor}
\label{sec:existence-strong-WCfun}

Let
$\mathcal{T}$ be a triangulated category with a 
bounded
weight structure 
$w=(\mathcal{T}^{w \leq 0},\mathcal{T}^{w \geq 0})$ and let
$(\tildew{\mathcal{T}}, i)$ be an f-category over
$\mathcal{T}$.

An object $X \in \tildew{\mathcal{T}}$ is by definition in
$\tildew{\mathcal{T}}^s$ if and only if $c(X) \in C^b(\heart(w))$.
Using \eqref{eq:cfun-gr} we obtain:
\begin{align}
  \label{eq:subtildeT-equiv}
  X \in \tildew{\mathcal{T}}^s 
  & \Leftrightarrow \forall a \in \DZ: [a]\gr^a(X) \in
  \heart(w)=\mathcal{T}^{w=0}\\
  \notag
  & \Leftrightarrow \forall a \in \DZ: \gr^a(X) = i\inv s^{-a} \sigma_a (X) \in
  [-a]\mathcal{T}^{w=0}=\mathcal{T}^{w=a}\\
  \notag
  & \Leftrightarrow \forall a \in \DZ: \sigma_a (X) \tildew{\in}
  s^a(i(\mathcal{T}^{w=a})),
\end{align}
where ``$\tildew{\in}$" stands for ``is isomorphic to some object in".

\begin{remark}
  \label{rem:tildeTs-and-heart-comp-w-str}
  Note the difference between $\tildew{\mathcal{T}}^s$ and
  the heart \eqref{eq:heart-comp-w-str} 
  of
  the unique compatible w-structure on $\tildew{\mathcal{T}}$ 
  described in Proposition~\ref{p:compatible-w-str}.
\end{remark}

Observe that $\tildew{\mathcal{T}}^s$ is not a triangulated subcategory
of $\tildew{\mathcal{T}}$: It is not closed under the translation $[1]$.

However $\tildew{\mathcal{T}}^s$ is closed under $[1]s\inv$
(use \eqref{eq:c-Sigma-commute}) and under
extensions: If 
$(A,X,B)$ is a triangle with $A, B \in \tildew{\mathcal{T}}^s$,
apply the triangulated functor $\gr^a$ and obtain the triangle 
$(\gr^a(A), \gr^a(X), \gr^a(B))$; now use that $\mathcal{T}^{w=a}$ is closed
under extensions 
(Lemma~\ref{l:weight-str-basic-properties}~\eqref{enum:weight-perp-prop}).

It is obvious from \eqref{eq:subtildeT-equiv}
that 
$\tildew{\mathcal{T}}^s$ 
is stable under all $\sigma$-truncations.

\begin{lemma}
  \label{l:omega-range-versus-weights-on-tildeTs}
  Let $X \in \tildew{\mathcal{T}}^s([a,b])$ for $a,b \in \DZ$. Then
  $\omega(X) \in \mathcal{T}^{w \in [a,b]}$.
\end{lemma}

\begin{proof}
  We can build up $X$ as indicated in the
  following diagram in the case $[a,b]=[-2,1]$:
  \begin{equation}
    \label{eq:build-up-X}
    \xymatrix@C-1.3cm{
      {X=\sigma_{\leq 1}X} \ar[rr] 
      && {\sigma_{\leq 0}X} \ar[rr] \ar@{~>}[ld]
      && {\sigma_{\leq -1}X} \ar[rr] \ar@{~>}[ld]
      && {\sigma_{\leq -2}X} \ar[rr] \ar@{~>}[ld]
      && {\sigma_{\leq -3}X=0} \ar@{~>}[ld]
      \\
      & {\sigma_1(X)} \ar[lu]
      && {\sigma_0(X)} \ar[lu]
      && {\sigma_{-1}(X)} \ar[lu]
      && {\sigma_{-2}(X)} \ar[lu]
    }
  \end{equation}
  All triangles are isomorphic to $\sigma$-truncation triangles with
  the wiggly arrows of degree one.
  Since $\sigma_n(X)\cong s^n(i(\gr^n(X)))$ 
  Proposition~\ref{p:functor-omega} \ref{enum:functor-omega-iii}
  yields $\omega(\sigma_n(X)) \cong \gr^n(X)$.
  If we apply $\omega$ to diagram \eqref{eq:build-up-X}
  we obtain a diagram that is isomorphic to
  \begin{equation*}
    \xymatrix@C-1.3cm{
      {\omega(X)} \ar[rr] 
      && {\omega(\sigma_{\leq 0}X)} \ar[rr] \ar@{~>}[ld]
      && {\omega(\sigma_{\leq -1}X)} \ar[rr] \ar@{~>}[ld]
      && {\omega(\sigma_{\leq -2}X)} \ar[rr] \ar@{~>}[ld]
      && {0} \ar@{~>}[ld]
      \\
      & {\gr^1(X)} \ar[lu]
      && {\gr^0(X)} \ar[lu]
      && {\gr^{-1}(X)} \ar[lu]
      && {\gr^{-2}(X).} \ar[lu]
    }
  \end{equation*}
  Since $\gr^n(X) \in \mathcal{T}^{w=n}$ and all $\mathcal{T}^{w \leq
    n}$, $\mathcal{T}^{w\geq n}$ are closed under extensions 
  (Lemma \ref{l:weight-str-basic-properties}
  \eqref{enum:weight-perp-prop}) we obtain the claim.
\end{proof}

\begin{remark}
  \label{rem:omega-range-versus-weights-on-tildeTs}
  The converse statement of
  Lemma~\ref{l:omega-range-versus-weights-on-tildeTs} is not true in
  general (but cf.
  Prop.~\ref{p:omega-on-subtildeT}):
  Take a non-zero object $X \in \tildew{\mathcal{T}}^s([0])$ (so
  $X\cong \sigma_0(X) \tildew{\in} i(\mathcal{T}^{w=0})$) and fit the
  morphism $\alpha_X$ into a triangle
  \begin{equation*}
    \xymatrix{
      {[-1]s(X)} \ar[r]
      &
      {E} \ar[r]
      &
      {X} \ar[r]^-{\alpha_X}
      &
      {s(X).}
    }
  \end{equation*}
  Then $E \in \tildew{\mathcal{T}}^s$. Since $\omega(\alpha_X)$ is an
  isomorphism we have $\omega(E)=0$ but $\range(E)=[0,1]$.
\end{remark}

In the rest of this section we additionally assume:
The weight structure
$w=(\mathcal{T}^{w \leq 0},\mathcal{T}^{w \geq 0})$ is bounded,
and $\tildew{\mathcal{T}}$ satisfies 
axiom~\ref{enum:filt-tria-cat-3x3-diag}.

\begin{lemma}
  [{\cite[Lemma 8.4.2]{bondarko-weight-str-vs-t-str}}]
  \label{l:alpha-hom-subtildeT}
  For all $M, N \in \tildew{\mathcal{T}}^s$ the map
  \begin{equation*}
    (\alpha_{N})_*=(\alpha_{N} \comp ?): \Hom_{\tildew{\mathcal{T}}}(M, N) 
    \ra \Hom_{\tildew{\mathcal{T}}}(M, s(N))
  \end{equation*}
  is surjective and for all $a>0$ the map
  \begin{equation*}
    (\alpha_{s^a(N)})_*=(\alpha_{s^a(N)} \comp ?): \Hom_{\tildew{\mathcal{T}}}(M, s^a(N)) 
    \ra \Hom_{\tildew{\mathcal{T}}}(M, s^{a+1}(N))
  \end{equation*}
  is bijective.
\end{lemma}

\begin{proof}
  We complete $\alpha_N$ to a triangle 
  \begin{equation}
    \label{eq:alpha-N-tri}
    \xymatrix{
      {N} \ar[r]^{\alpha_N} & 
      {s(N)} \ar[r] &
      {Q} \ar[r] &
      {[1]N,}
    }
  \end{equation}
  apply $s^a$ and obtain the triangle
  \begin{equation*}
    \xymatrix@C=1.5cm{
      {s^a(N)} \ar[r]^{s^a(\alpha_N)=\alpha_{s^a(N)}} &
      {s^{a+1}(N)} \ar[r] &
      {s^a(Q)} \ar[r] &
      {[1]s^a(N).}
    }    
  \end{equation*}
  Applying $\tildew{\mathcal{T}}(M,?)$
  to this 
  triangle yields an exact sequence
  \begin{equation*}
    \tildew{\mathcal{T}}(M,[-1]s^a(Q)) \ra 
    \tildew{\mathcal{T}}(M,s^a(N)) \xra{(\alpha_{s^a(N)})_*} 
    \tildew{\mathcal{T}}(M,s^{a+1}(N)) \ra 
    \tildew{\mathcal{T}}(M,s^a(Q))
  \end{equation*}
  Hence we have to prove:
  \begin{align*}
    \tildew{\mathcal{T}}(M,s^a(Q)) & = 0 && \text{for $a \geq 0$, and}\\
    \tildew{\mathcal{T}}(M,[-1]s^a(Q)) & = 0 && \text{for $a > 0$.}
  \end{align*}
  This is clearly implied by
  \begin{align*}
    \tildew{\mathcal{T}}(M,[b]s^a(Q)) & = 0 && \text{for all $a, b \in
      \DZ$ with $a+b\geq 0$.}
  \end{align*}
  
  We claim more generally that
  \begin{align}
    \label{eq:m-prime-cq-claim}
    \tildew{\mathcal{T}}(M',[c]Q) & = 0 && \text{for all $c \in \DN$
      and all $M' \in \tildew{\mathcal{T}}^s$;}
  \end{align}
  the above special case is obtained by setting $M':=s^{-a}[a]M$ (note
  that $\tildew{\mathcal{T}}^s$ is $[1]s\inv$-stable) and $c=a+b$ using
  $\tildew{\mathcal{T}}(M,[b]s^a(Q)) \xsira{s^{-a}[a]}
  \tildew{\mathcal{T}}(M',[c]Q)$.

  We claim that we can assume in 
  \eqref{eq:m-prime-cq-claim} that the support of $M'$ and $N$
  (the object $N$ determines $Q$ up to isomorphism) is a singleton (or empty):
  For $M'$ this is obvious; for $N$ we use axiom
  \ref{enum:filt-tria-cat-3x3-diag} for $\id_N: N \ra N$ 
  and obtain the following $3 \times 3$-diagram:
  \begin{equation*}
    \xymatrix{
      {[1]\sigma_{\geq d+1} (N)} \ar@{..>}[r] & 
      {[1]s(\sigma_{\geq d+1} (N))} \ar@{..>}[r] &
      {[1]Q'} \ar@{..>}[r] \ar@{}[rd]|{\anticomm}& 
      {[2]\sigma_{\geq d+1} (N)} \\
      {\sigma_{\leq d}(N)} \ar[u] \ar[r]^-{\alpha_{\sigma_{\leq d}(N)}} & 
      {s(\sigma_{\leq d}(N))} \ar[u] \ar[r] &
      {Q''} \ar[u] \ar[r] & 
      {[1]\sigma_{\leq d}(N)} \ar@{..>}[u] \\
      {N} \ar[r]^-{\alpha_N} \ar[u] & 
      {s(N)} \ar[r] \ar[u] &
      {Q} \ar[r] \ar[u] & 
      {[1]N} \ar@{..>}[u] \\
      {\sigma_{\geq d+1} (N)} \ar[u] \ar[r]^-{\alpha_{\sigma_{\geq d+1}(N)}} & 
      {s(\sigma_{\geq d+1} (N))} \ar[u] \ar[r] &
      {Q'} \ar[u] \ar[r] & 
      {[1]\sigma_{\geq 1} (N)} \ar@{..>}[u] 
    }
  \end{equation*}
  This shows that knowing \eqref{eq:m-prime-cq-claim} for $Q'$ and
  $Q''$ implies \eqref{eq:m-prime-cq-claim} for $Q$, proving the claim.

  Assume now that the support of $M'$ is $[x]$ and that of $N$
  is $[y]$ for some $x,y \in \DZ$.
  This means that we can assume 
  (cf.\ \eqref{eq:subtildeT-equiv})
  that 
  \begin{align*}
    M' & = s^x i(X) && \text{for some $X \in \mathcal{T}^{w=x}$, and}\\
    N & = s^y i(Y) && \text{for some $Y \in \mathcal{T}^{w=y}$.}
  \end{align*}
  Since $M'$ is in $\tildew{\mathcal{T}}(\geq x)$, the triangle 
  $(\sigma_{\geq x}([c]Q), [c]Q, \sigma_{\leq x-1}([c]Q))$ and
  \ref{enum:filt-tria-cat-no-homs} show the first isomorphism in
  \begin{equation*}
    \tildew{\mathcal{T}}(M', [c] Q) \sila
    \tildew{\mathcal{T}}(M', \sigma_{\geq x}([c] Q)) 
    \xsira{\omega}
    {\mathcal{T}}(\omega(M'), \omega(\sigma_{\geq x}([c] Q))),
  \end{equation*}
  the second isomorphism is a consequence of
  \ref{enum:functor-omega-iv} and $M' \in \tildew{\mathcal{T}}(\leq
  x)$.
  Note that $Q$ and $[c]Q$ is in $\tildew{\mathcal{T}}([y,y+1])$ by
  \eqref{eq:alpha-N-tri}.
  In order to show that 
  ${\mathcal{T}}(\omega(M'), \omega(\sigma_{\geq x}([c] Q)))$
  vanishes, we consider three cases:
  \begin{enumerate}
  \item $x > y+1$: Then $\sigma_{\geq x}([c] Q)=0$.
  \item $x \leq y$: Then $\sigma_{\geq x}([c] Q)=[c]Q$. 
    Applying the triangulated functor $\omega$ to
    \eqref{eq:alpha-N-tri} and using \ref{enum:functor-omega-iii}
    shows that $\omega(\sigma_{\geq x}([c] Q)=\omega([c]Q)=0$.
  \item $x=y+1$: Applying the triangulated functor $\sigma_{\geq x}$
    to \eqref{eq:alpha-N-tri} shows that 
    $\sigma_{\geq x}([c] Q) \cong [c]s(N)=[c]s^{y+1}i(Y)=[c]s^x i(Y)$.
    Hence we have
    \begin{equation*}
      {\mathcal{T}}(\omega(M'), \omega(\sigma_{\geq x}([c] Q)))
      = {\mathcal{T}}(\omega(s^xi(X)), \omega([c]s^xi(Y)))
      \cong {\mathcal{T}}(X, [c]Y)
    \end{equation*}
    where we use \ref{enum:functor-omega-iii} ($i(X)$
    and $s^xi(X)$ are connected by a sequence of morphisms
    $\alpha_{s^?i(X)}$, and similarly for $i(Y)$) and the fact that
    $\omega|_{\tildew{\mathcal{T}}([0])}$ is a quasi-inverse of
    $i$. Since $X \in \mathcal{T}^{w \geq 
      x}=\mathcal{T}^{w \geq y+1}$ and $[c]Y
    \in \mathcal{T}^{w \leq y-c} \subset \mathcal{T}^{w \leq y}$ we have
    ${\mathcal{T}}(X, [c]Y)=0$ by \ref{enum:ws-iii}.
  \end{enumerate}
\end{proof}

\begin{proposition}
  \label{p:omega-on-subtildeT}
  The restriction
  $\omega|_{\tildew{\mathcal{T}}^s}:\tildew{\mathcal{T}}^s \ra
  \mathcal{T}$ of 
  $\omega:\tildew{\mathcal{T}} \ra \mathcal{T}$ 
  is full (i.\,e.\ induces
  epimorphisms on morphism spaces) and essentially surjective (i.\,e.\
  surjective on isoclasses of objects).

  More precisely
  $\omega|_{\tildew{\mathcal{T}}^s([a,b])}:\tildew{\mathcal{T}}^s([a,b]) \ra
  \mathcal{T}^{w \in [a,b]}$ 
  (cf.\ Lemma~\ref{l:omega-range-versus-weights-on-tildeTs})
  has these properties.
\end{proposition}

\begin{proof}
  We first prove that $\omega|_{\tildew{\mathcal{T}}^s}$ induces
  epimorphisms on morphism spaces. Let $M, N \in
  \tildew{\mathcal{T}}^s$.
  By \ref{enum:filt-tria-cat-exhaust} we find $m, n \in \DZ$ such that
  $M \in \tildew{\mathcal{T}}(\leq m)$ and $N \in
  \tildew{\mathcal{T}}(\geq n)$. 
  Choose $a \in \DZ$ satisfying $a \geq m-n$.

  We give a pictorial proof:
  Consider the following diagrams
  \begin{equation}
    \label{eq:pictorial-omegas-subtilde}
    \xymatrix{
      & {s^a(N)} \\
      & {\vdots} \ar[u] \\
      & {s(N)} \ar[u] \\
      {M} 
      \ar[r]^-{f_0}
      \ar[ru]^-{f_1}
      \ar[ruuu]^-{f_a}
      & {N,} \ar[u]_{\alpha_N}
    }
    \quad
    \xymatrix{
      & {\omega(s^a(N))} \\
      & {\vdots} \ar[u]_\sim \\
      & {\omega(s(N))} \ar[u]_\sim \\
      {\omega(M)} 
      \ar[r]^-{g_0}
      \ar[ru]^-{g_1}
      \ar[ruuu]^-{g_a}
      & {\omega(N).} \ar[u]_{\omega(\alpha_N)}^\sim
    }
  \end{equation}
  Assume that we are given a morphism $g_0$. 
  Since $\omega$ maps every $\alpha_?$ to an isomorphism (cf.\
  \ref{enum:functor-omega-iii}), $g_0$ uniquely determines $g_1,\dots,
  g_a$ such that the diagram on the right commutes.
  
  Since 
  $M \in \tildew{\mathcal{T}}(\leq m)$ and $s^a(N) \in
  \tildew{\mathcal{T}}(\geq n+a) \subset \tildew{\mathcal{T}}(\geq m)$,
  \ref{enum:functor-omega-iv} implies that there is a unique $f_a$
  satisfying $\omega(f_a)=g_a$. This $f_a$ yields
  $f_{a-1}, \dots, f_1$ (uniquely) and $f_0$ (possibly non-uniquely)
  such that the diagram on the left commutes 
  (use Lemma \ref{l:alpha-hom-subtildeT}). Taking $\omega$ of the
  diagram on the left it is clear that $\omega(f_0)=g_0$.

  Now we prove that $\omega|_{\tildew{\mathcal{T}}^s}$ is surjective
  on isoclasses of objects. Let $X \in \mathcal{T}$ be given.
  Since the given weight structure is bounded there are $a,b \in \DZ$
  such that $X \in \mathcal{T}^{w \in [a,b]}$. We prove the statement
  by induction on $b-a$. If $a > b$ then $X=0$ and the statement is
  obvious. Assume $a=b$. Then $s^a(i(X))$ is in $\tildew{\mathcal{T}}^s([a])$
  by \eqref{eq:subtildeT-equiv} and we have $\omega(s^a(i(X)))\cong
  \omega(i(X))\cong X$.

  Now assume $a<b$. Choose $c \in \DZ$ with $a \leq c < b$ and
  take a weight decomposition
  \begin{equation}
    \label{eq:omega-subtildeT-wdec}
    w_{\geq c+1}X \ra X \ra w_{\leq c}X \ra [1]w_{\geq c+1}X.
  \end{equation}
  Lemma~\ref{l:weight-str-basic-properties}
  \eqref{enum:weights-bounded} shows that
  $w_{\leq c}X \in \mathcal{T}^{w \in [a,c]}$
  and $w_{\geq c+1}X \in \mathcal{T}^{w \in [c+1,b]}$.
  By induction we can hence lift
  $w_{\leq c}X \in \mathcal{T}^{w\in[a,c]}$ 
  and $[1]w_{\geq c+1}X \in \mathcal{T}^{w \in [c,b-1]}$ to
  objects
  $\tildew{A} \in \tildew{\mathcal{T}}^s([a,c])$, 
  $\tildew{B} \in \tildew{\mathcal{T}}^s([c,b-1])$.
  This shows that the triangle
  \eqref{eq:omega-subtildeT-wdec} is isomorphic to 
  a triangle 
  \begin{equation}
    \label{eq:omega-subtildeT-unten}
    [-1]\omega(\tildew{B}) \ra X \ra \omega(\tildew{A}) \xra{g}
    \omega(\tildew{B}). 
  \end{equation}
  Since we have already proved fullness there exists $f: \tildew{A} \ra
  \tildew{B}$ such that $\omega(f)=g$. We complete the composition 
  $\tildew{A} \xra{f} \tildew{B} \xra{\alpha_{\tildew{B}}}
  s(\tildew{B})$ to a triangle
  \begin{equation*}
    [-1]s(\tildew{B}) \ra \tildew{X} \ra \tildew{A} \xra{\alpha\comp
      f} s(\tildew{B}).
  \end{equation*}
  The image of this triangle under $\omega$ is isomorphic to triangle
  \eqref{eq:omega-subtildeT-unten}. 
  In particular $X \cong \omega(\tildew{X})$.
  Since $[-1]s(\tildew{B}) \in \tildew{\mathcal{T}}^s([c+1,b])$ and
  $\tildew{A} \in \tildew{\mathcal{T}}^s([a,c])$ we have
  $\tildew{X} \in \tildew{\mathcal{T}}^s([a,b])$ since
  $\tildew{\mathcal{T}}^s([a,b])$ is closed under extensions.
\end{proof}

In the following we view $\tildew{\mathcal{T}}^s$ as an additive
category with translation $[1]s\inv$. Proposition~\ref{p:functor-omega},
\ref{enum:functor-omega-iii} and the fact that $\omega$ is
triangulated turn 
$\omega|_{\tildew{\mathcal{T}}^s}:\tildew{\mathcal{T}}^s \ra
\mathcal{T}$ into a functor of additive categories with translation.

Recall (for example from \cite[A.3.1]{assem-simson-skowronski-1})
that a two-sided ideal $\mathcal{I}$ in an additive category
$\mathcal{A}$ is a subclass of
the class of all morphism satisfying some obvious properties. Then the quotient
$\mathcal{A}/\mathcal{I}$ has the same objects as $\mathcal{A}$ but
morphisms are identified if their difference is in the ideal. Then
$\mathcal{A}/\mathcal{I}$ is again an additive category and the
obvious \textbf{quotient functor} $\can: \mathcal{A} \ra
\mathcal{A}/\mathcal{I}$ has an obvious universal property.
If $F:\mathcal{A} \ra \mathcal{B}$ is an additive functor of additive
categories, its kernel $\Kern F$ is the two-sided ideal given by 
\begin{equation*}
  (\Kern F)(A,A') = \Kern(F:\mathcal{A}(A,A') \ra \mathcal{B}(FA,FA'))
\end{equation*}
for $A, A' \in \mathcal{A}$.

Define 
$\mathcal{Q}:={\tildew{\mathcal{T}}^s/(\Kern
  \omega|_{\tildew{\mathcal{T}}^s})}$
and let $\can: \tildew{\mathcal{T}}^s \ra \mathcal{Q}$ be the quotient
functor.

The functor $[1]s\inv: \tildew{\mathcal{T}}^s \ra
\tildew{\mathcal{T}}^s$ descends to a functor $\mathcal{Q} \ra
\mathcal{Q}$
denoted by the same symbol: $(\mathcal{Q}, [1]s\inv)$ is an additive
category with translation and $\can$ is a functor of such categories.

\begin{proposition}
  \label{p:omega-tildeTs-equiv}  
  The functor $\omega|_{\tildew{\mathcal{T}}^s}$ factors
  uniquely 
  to an equivalence
  \begin{equation*}
    \ol{\omega}: (\mathcal{Q},[1]s\inv) \sira (\mathcal{T},[1])
  \end{equation*}
  of additive categories with translation
  (as indicated in diagram~\eqref{eq:omega-subtildeT-exp}).
\end{proposition}

\begin{proof}
  This is clear from
  Proposition~\ref{p:omega-on-subtildeT}.
\end{proof}

In any additive category with translation we have the notion of
candidate triangles and morphisms between them. Let
$\Delta_\mathcal{Q}$ be the class of all candidate triangles in
$\mathcal{Q}$ that are isomorphic to the image\footnote
{
  We
  assume that $[1]s\inv [-1]s =\id$.
}
\begin{equation*}
  \xymatrix{
    {[-1]s(X)} \ar[r]^-{\ol{f}} &
    {Y} \ar[r]^-{\ol{g}} &
    {Z} \ar[r]^-{\ol{h}} &
    {X} &
  }
\end{equation*}
in $\mathcal{Q}$ of a sequence
\begin{equation*}
  \xymatrix{
    {[-1]s(X)} \ar[r]^-{{f}} &
    {Y} \ar[r]^-{{g}} &
    {Z} \ar[r]^-{{h}} &
    {X} &
  }
\end{equation*}
in $\tildew{\mathcal{T}}^s$ such that
\begin{equation*}
  \xymatrix{
    {[-1]s(X)} \ar[r]^-{f} &
    {Y} \ar[r]^-{{g}} &
    {Z} \ar[r]^-{{\alpha_X \comp h}} &
    {s(X)} &
  }
\end{equation*}
is a triangle in $\tildew{\mathcal{T}}$.

\begin{lemma}
  \label{l:omega-bar-triang}
  $(\mathcal{Q}, [1]s\inv, \Delta_\mathcal{Q})$ is a triangulated
  category and $\ol{\omega}: \mathcal{Q} \ra \mathcal{T}$ is a functor
  of triangulated categories.
\end{lemma}

\begin{proof}
  Since we already know that $\ol{\omega}$ is an equivalence of
  additive categories with translation it is
  sufficient to show that a candidate triangle in $\mathcal{Q}$ is in
  $\Delta_{\mathcal{Q}}$ if and only if its image under $\ol{\omega}$
  is a triangle in $\mathcal{T}$. This is an easy exercise left to the reader.
\end{proof}

\begin{remark}
  There is another description of $\Delta_\mathcal{Q}$:
  Let $\Delta'_\mathcal{Q}$ be the class of all candidate triangles in
  $\mathcal{Q}$ that are isomorphic to the image
  \begin{equation*}
    \xymatrix{
      {X} \ar[r]^-{\ol{f}} &
      {Y} \ar[r]^-{\ol{g}} &
      {Z} \ar[r]^-{\ol{h}} &
      {[1]s\inv(X)} &
    }
  \end{equation*}
  in $\mathcal{Q}$ of a sequence
  \begin{equation*}
    \xymatrix{
      {X} \ar[r]^-{{f}} &
      {Y} \ar[r]^-{{g}} &
      {Z} \ar[r]^-{{h}} &
      {[1]s\inv(X)} &
    }
  \end{equation*}
  in $\tildew{\mathcal{T}}^s$ such that
  \begin{equation*}
    \xymatrix{
      {s\inv(X)} \ar[r]^-{{f\comp \alpha_{s\inv(X)}}} &
      {Y} \ar[r]^-{{g}} &
      {Z} \ar[r]^-{{h}} &
      {[1]s\inv(X)} &
    }
  \end{equation*}
  is a triangle in $\tildew{\mathcal{T}}$.
  Lemma~\ref{l:omega-bar-triang} is also true for the class
  $\Delta'_\mathcal{Q}$ of triangles with essentially the same
  proof. This shows that $\Delta_\mathcal{Q}=\Delta'_\mathcal{Q}$.
\end{remark}

Recall that $c$ and $h$ are functors of additive categories with
translation if we equip $\tildew{\mathcal{T}}$ (or
$\tildew{\mathcal{T}}^s$) with the 
translation $[1]s\inv$. 
This is used implicitly in the next
proposition.

\begin{proposition}
  \label{p:cone-alpha-homotopic-to-zero}
  Let $m:M\ra N$ be a morphism in $\tildew{\mathcal{T}}^s$. Assume
  that $\alpha_N \comp m$ appears as the last morphisms in a triangle
  \begin{equation}
    \label{eq:alpha-triangle-MN}
    \xymatrix{
      {[-1]s(N)} \ar[r]^-{u} &
      {Q} \ar[r]^-{v} &
      {M} \ar[r]^-{\alpha_N \comp m} &
      {s(N)}
    }
  \end{equation}
  in $\tildew{\mathcal{T}}$.
  Then $Q \in \tildew{\mathcal{T}}^s$, 
  the functor 
  $c:\tildew{\mathcal{T}}^s \ra C^b(\heart(w))$  
  maps the sequence
  \begin{equation}
    \label{eq:alpha-lift-triangle-MN}
    \xymatrix{
      {[-1]s(N)} \ar[r]^-{u} &
      {Q} \ar[r]^-{v} &
      {M} \ar[r]^-{m} &
      {N}
    }
  \end{equation}
  in $\tildew{\mathcal{T}}^s$ 
  to a sequence that is isomorphic to 
  \begin{equation}
    \label{eq:h-triangulated}
    \xymatrix{
      {\Sigma\inv (c(N))} \ar[r]^-{\svek {1}0} &
      {\Sigma\inv (\Cone(-c(m)))} \ar[r]^-{\zvek 0{1}} &
      {c(M)} \ar[r]^-{c(m)} &
      {c(N)}      
    }
  \end{equation}
  in $C^b(\heart(w))$, and 
  the  
  functor $h:\tildew{\mathcal{T}}^s \ra K^b(\heart(w))^\anti$ maps
  \eqref{eq:alpha-lift-triangle-MN} to a triangle in ${K^b(\heart(w))}^\anti$.

  In the particular case that $m=\id_N:N \ra N$ we have $h(Q)=0$ in
  $K^b(\heart(w))^\anti$, i.\,e.\ 
  $\id_{c(Q)}$ is homotopic to zero.
\end{proposition}

\begin{proof}
  Since $\tildew{\mathcal{T}}^s$ is stable under $[-1]s$ and closed
  under extensions, the first three objects in 
  the triangle \eqref{eq:alpha-triangle-MN} 
  are in $\tildew{\mathcal{T}}^s$.

  When computing $c=c''\comp c'$ in this proof we use the second
  approach
  to $c'$ described at the end of
  Section~\ref{sec:first-step-cprime}.

  \textbf{First case: $m$ is a morphism in $\tildew{\mathcal{T}}^s([a])$:}
  We start with the case that $m:M \ra N$ is in
  $\tildew{\mathcal{T}}^s([a])$ for some $a \in \DZ$.
  Assume that we are given a triangle
  \begin{equation}
    \label{eq:alpha-triangle-MN-gr-a}
    \xymatrix{
      {[-1]s(N)} \ar[r]^-{u'} & 
      {L} \ar[r]^-{v'} &
      {M} \ar[r]^-{\alpha_N \comp m} &
      {s(N).} 
    }
  \end{equation}
  in $\tildew{\mathcal{T}}$ (cf.\ \eqref{eq:alpha-triangle-MN}; we
  write $L$ instead of $Q$ here for 
  notational reasons to become clear later on).
  Note that $\range(L) \subset [a,a+1]$ and $\range([-1]s(N)) \subset [a+1]$.
  Let 
  \begin{equation*}
    \xymatrix{
      {\sigma_{a+1}(L)} \ar[r]^-{u''} &
      {L} \ar[r]^-{v''} &
      {\sigma_{a}(L)} \ar[r]^-{\tildew{d}^a_{L}} &
      {[1]\sigma_{a+1} ({L}).}
    }
  \end{equation*}
  be the triangle\footnote{
    We assume in this proof without loss of generality that $\sigma_{\geq n}$ 
    (more precisely $g^n:\sigma_{\geq n} \ra \id$) is the identity 
    on objects of
    $\tildew{\mathcal{T}}(\geq n)$, and similarly for $k^n$. 
    This gives for example $\sigma_{[a,a+1]}(L)=L$.
  }
  constructed 
  in the second approach to $c'(L)$, cf.\
  \eqref{eq:sigma-truncation-for-tilde-d}.   
  This triangle is uniquely isomorphic to
  the 
  triangle
  \eqref{eq:alpha-triangle-MN-gr-a} by an isomorphism extending
  $\id_{L}$ 
  (use Prop.~\ref{p:BBD-1-1-9-copied-for-w-str}):
  \begin{equation}
    \label{eq:isom-trunc-extended-alpha}
    \xymatrix{
      {\sigma_{a+1}(L)} \ar[r]^-{u''} \ar[d]^{p}_{\sim} &
      {L} \ar[r]^-{v''} \gar[d] &
      {\sigma_{a}(L)} \ar[r]^-{\tildew{d}^a_{L}}
      \ar[d]^{q}_{\sim} &
      {[1]\sigma_{a+1} (L)} \ar[d]^{[1]p}_{\sim}\\
      {{[-1]s(N)}} \ar[r]^-{u'} & 
      {{L}} \ar[r]^-{v'} &
      {{M}} \ar[r]^-{\alpha_N \comp m} &
      {{s(N)}.} 
    }
  \end{equation}
  This proposition even characterizes $p$ as the unique morphism
  making the square on the left commutative (and similarly for
  $p\inv$), and $q$ as the unique 
  morphism making the square 
  in the middle commutative.

  In order to compute the differential of $c'(L)$ in terms of $m$ and
  to identify $p$ and $q$ we follow the second approach to $c'$
  described in  
  Section~\ref{sec:construction-functor-cfun}:
  We apply to the
  triangle \eqref{eq:alpha-triangle-MN-gr-a} the sequence of functors
  in the left column of the following diagram 
  (cf.\ \eqref{eq:sigma-truncation-for-tilde-d})
  and obtain in this way the rest of the diagram:
  \begin{equation}
    \label{eq:alpha-triangle-MN-gr-a-sigma}
    \xymatrix{
      \sigma_{a+1} \ar[d] & 
      {[-1]s(N)} \ar[r]^{\sigma_{a+1}(u')} \gar[d]&
      {\sigma_{a+1}(L)} \ar[r] \ar[d]^-{u''} &
      {0} \ar[r] \ar[d] &
      {s(N)} \gar[d] \\
      \sigma_{[a,a+1]} \ar[d] &
      {[-1]s(N)} \ar[r]^-{u'} \ar[d]&
      {L} \ar[r]^-{v'} \ar[d]^-{v''} &
      {M} \ar[r]^-{\alpha_N \comp m} \gar[d] &
      {s(N)} \ar[d] \\
      \sigma_{a} \ar[d]^-{\tildew{d}^a} &
      {0} \ar[r] \ar[d]&
      {\sigma_a(L)} \ar[r]^-{\sigma_a(v')} \ar[d]^-{\tildew{d}^a_{L}} &
      {M} \ar[r] \ar[d] &
      {0} \ar[d] \\
      [1]\sigma_{a+1} &
      {s(N)} 
      \ar[r]^{[1]\sigma_{a+1}(u')} &
      {[1]\sigma_{a+1}(L)} \ar[r]  &
      {0} \ar[r] &
      {[1]s(N)} 
    }
  \end{equation}
  From the above remarks we see that $\sigma_{a+1}(u')=p\inv$ and
  $\sigma_a(v')=q$
  (this can also be seen directly from
  \eqref{eq:isom-trunc-extended-alpha} using the adjunctions but we
  wanted to include the above diagram). 
  If we also use the commutativity of the square on the right in 
  \eqref{eq:isom-trunc-extended-alpha} we obtain
  \begin{equation}
    \label{eq:differential-eL-gr-supp-singleton}
    ([1]\sigma_{a+1}(u'))\inv \comp \tildew{d}^a_{L} \comp
    (\sigma_a(v'))\inv
    = ([1]p) \comp \tildew{d}^a_{L} \comp q\inv
    = \alpha_N \comp m.
  \end{equation}
  We will need precisely this formula later on. 
  It means that (up to unique isomorphisms) the differential
  $\tildew{d}^a_{L}$ is $\alpha_N \comp m$. 

  Let us finish the proof of the proposition in this special case
  in order to convince the
  reader that we got all signs correct.
  
  These results show that the image of the sequence 
  \begin{equation}
    \label{eq:lift-of-triangle-MN-gr-a}
    \xymatrix{
      {{[-1]s(N)}} \ar[r]^-{u'} & 
      {{L}} \ar[r]^-{v'} &
      {{M}} \ar[r]^-{m} &
      {{N}.} 
    }
  \end{equation}
  in $\tildew{\mathcal{T}}^s$
  under the functor $c'$ 
  is the sequence 
  \begin{equation*}
    \xymatrix{
      {0} \ar[r] \ar[d]&
      {\sigma_a(L)} \ar[r]^-{q} \ar[d]^-{\tildew{d}^a_{L}} &
      {M} \ar[r]^-{m} \ar[d] &
      {N} \ar[d] \\
      {s(N)} 
      \ar[r]^{[1]p\inv} &
      {[1]\sigma_{a+1}(L)} \ar[r]  &
      {0} \ar[r] &
      {0} 
    }
  \end{equation*}
  in $C^b_\Delta(\tildew{\mathcal{T}})$ where we draw the complexes
  vertically and only draw their components of degree $a$ in the upper
  and of degree $a+1$ in the lower row (all other components are
  zero).
  These sequence is isomorphic by the isomorphism $(\id, \svek
  q{[1]p}, \id, \id)$ to the sequence (use
  \eqref{eq:differential-eL-gr-supp-singleton}) 
  \begin{equation*}
    \xymatrix{
      {0} \ar[r] \ar[d]&
      {M} \ar[r]^-{\id_M} \ar[d]^-{\alpha_N \comp m} &
      {M} \ar[r]^-{m} \ar[d] &
      {N} \ar[d] \\
      {s(N)} 
      \ar[r]^{\id_{s(N)}} &
      {s(N)} \ar[r]  &
      {0} \ar[r] &
      {0.} 
    }
  \end{equation*}
  If we now apply the functor $c''$ we see that 
  the image of \eqref{eq:lift-of-triangle-MN-gr-a} under $c$ is
  isomorphic to
  \begin{equation*}
    \xymatrix{
      {0} \ar[r] \ar[d]&
      {s^{-a}(M)} \ar[r]^-{\id} \ar[d]^-{s^{-a}(m)} &
      {s^{-a}(M)} \ar[r]^-{s^{-a}(m)} \ar[d] &
      {s^{-a}(N)} \ar[d] \\
      {s^{-a}(N)} 
      \ar[r]^{\id} &
      {s^{-a}(N)} \ar[r]  &
      {0} \ar[r] &
      {0} 
    }
  \end{equation*}
  which is precisely
  \eqref{eq:h-triangulated}.
  If we pass to the homotopy category we obtain a candidate triangle 
  that has the same shape as
  \eqref{eq:mapcone-minus-m-rot2} but all morphisms are multiplied
  by $-1$: It is a triangle in $K^b(\heart(w))^\anti$.
  This finishes the proof in the case that
  $m:M \ra N$ is in $\tildew{\mathcal{T}}^s([a])$ for some $a \in \DZ$.

  \textbf{Observation:}
  Recall that the first three objects of triangle
  \eqref{eq:alpha-triangle-MN}  are in $\tildew{\mathcal{T}}^s$. If we
  apply the triangulated 
  functor $\sigma_b$ (for any $b \in \DZ$), we obtain a triangle with
  first three terms isomorphic to objects in $s^b(i(\mathcal{T}^{w=b})$
  by \eqref{eq:subtildeT-equiv}; 
  since there are only trivial extensions in the heart of a weight
  structure
  (see Lemma~\ref{l:weight-str-basic-properties}), this triangle
  is isomorphic to the obvious direct sum triangle.
  We will construct explicit splittings.
  
  \textbf{The general case.}
  Let $m:M \ra N$ be a morphism in $\tildew{\mathcal{T}}^s$.

  Let $a \in \DZ$ be arbitrary. 
  Axiom \ref{enum:filt-tria-cat-3x3-diag}
  (we use the variant \eqref{eq:filt-tria-cat-nine-diag-sigma}
  for the morphism $\sigma_{\leq a}(m): X=\sigma_{\leq a}(M) \ra Y =
  \sigma_{\leq a}(N)$ and $n=a-1$) gives a $3\times 3$-diagram
  \begin{equation}
    \label{eq:sigma-leq-a-N-slice-off}
    \xymatrix{
      {[-1]s({\sigma_{a}(N)})} \ar[r]^-{u'_a} 
      \ar[d]^-{[-1]s(g^a)}
      & 
      {L_a} \ar[r]^-{v'_a} \ar[d]^-{f'_a} & 
      {{\sigma_{a}(M)}} \ar[rr]^-{\alpha_{\sigma_{a}(N)} \comp \sigma_a(m)} 
      \ar[d]^-{g^a}
      && 
      {s({\sigma_{a}(N)})} 
      \ar@{..>}[d]^-{s(g^a)}
      \\
      {[-1]s({\sigma_{\leq a}(N)})} \ar[r]^-{u_a} \ar[d]^-{[-1]s({k_N^{a-1,a}})} &
      {Q_a} \ar[r]^-{v_a} \ar[d]^-{f_{a-1,a}} &
      {{\sigma_{\leq a}(M)}} \ar[rr]^-{\alpha_{\sigma_{\leq a}(N)}
        \comp \sigma_{\leq a}(m)} \ar[d]^-{k_M^{a-1,a}} &&
      {s({\sigma_{\leq a}(N)})} \ar@{..>}[d]^-{s({k_N^{a-1,a}})} \\
      {[-1]s({\sigma_{\leq a-1}(N)})} \ar[r]^-{u_{a-1}} 
      \ar[d]
      & 
      {Q_{a-1}} \ar[r]^-{v_{a-1}} \ar[d] & 
      {{\sigma_{\leq a-1}(M)}} \ar[rr]^-{\alpha_{\sigma_{\leq a-1}(N)}
        \comp \sigma_{\leq a-1}(m)}
      \ar[d]
      \ar@{}[rrd]|-{\anticomm} & &
      {s({\sigma_{\leq a-1}(N)})} 
      \ar@{..>}[d]
      \\
      {s({\sigma_{a}(N)})} \ar@{..>}[r] &
      {[1]L_a} \ar@{..>}[r] &
      {[1]{\sigma_{a}(M)}}
      \ar@{..>}[rr]^-{[1](\alpha_{\sigma_{a}(N)}\comp \sigma_a(m))} &&
      {[1]s({\sigma_{a}(N)})} 
    }
  \end{equation}
  where we write $k_?^{a-1,a}$ for the adjunction morphism
  $k^{a-1}_{\sigma_{\leq a}?}$ and $g^a$ for the adjunction morphisms
  $g^{a}_{\sigma_{\leq a}M}$ and $g^{a}_{\sigma_{\leq a}N}$. 
  We fix such a $3 \times 3$-diagram for any $a \in \DZ$.

  By \ref{enum:filt-tria-cat-exhaust} we find $A \in \DZ$ such that
  $M, N \in \tildew{\mathcal{T}}(\leq A)$. Then the top and bottom row in
  diagram 
  \eqref{eq:sigma-leq-a-N-slice-off} are zero for $a>A$, and the two
  rows in the middle are connected by an isomorphism of triangles.
  It is easy to see (first define $f_A$, then $f_{A-1}=f_{A-1,A}f_A$
  etc.) that there are morphisms 
  $f_a: Q \ra [1]Q_a$ 
  (for any $a \in \DZ$) such that 
  $f_{a-1,a} f_a = f_{a-1}$ holds and such that
  \begin{equation}
    \label{eq:N-to-sigma-leq-a-second}
    \xymatrix{
      {{[-1]s(N)}} \ar[r]^-u \ar[d]^-{[-1]s(k_N^a)} & 
      {{Q}} \ar[r]^-v \ar[d]^-{f_a} &
      {{M}} \ar[r]^-{\alpha_N \comp m} \ar[d]^-{k_M^a} &
      {{s(N)}} \ar[d]^-{s(k_N^a)} \\
      {{[-1]s(\sigma_{\leq a}(N))}} \ar[r]^-{u_a} & 
      {{Q_a}} \ar[r]^-{v_a} &
      {{\sigma_{\leq a}(M)}} \ar[r]^-{\alpha_{\sigma_{\leq a}(N)}\comp
      \sigma_{\leq a(m)}} &
      {{s(\sigma_{\leq a}(N))}} 
    }
  \end{equation}
  is a morphism of triangles, where
  the top triangle is \eqref{eq:alpha-triangle-MN}, the bottom triangle
  is the second horizontal triangle from above in
  \eqref{eq:sigma-leq-a-N-slice-off} and $k_?^a$ are the adjunction
  morphisms. 
  Since $k_?^{a-1,a} k_?^a = k_?^{a-1}$ holds
  (under the usual identifications, cf.\
  \eqref{eq:sigma-truncation-two-parameters}),  
  the morphisms of triangles
  \eqref{eq:N-to-sigma-leq-a-second} (for $a \in \DZ$)
  are compatible with the morphisms between the two middle rows of 
  \eqref{eq:sigma-leq-a-N-slice-off}.

  We claim that $\sigma_b(f_a)$ is an isomorphism for all
  $b \leq a$: Applying the triangulated functor $\sigma_b$ to
  \eqref{eq:N-to-sigma-leq-a-second}
  gives a morphism of triangles
  with two components isomorphisms
  by \eqref{eq:a-geq-b-sigma-isos}
  (and \eqref{eq:sgle-eq-s-comm}); hence it is an isomorphisms and
  $\sigma_b(f_a)$ is an isomorphism. 
  This shows that any $f_a$ induces an isomorphism between
  the parts of the complexes 
  $c(Q)$ and $c(Q_a)$ in degrees $\leq a$:
  Apply $\tildew{d}^b: \sigma_b  \ra [1] \sigma_{b+1}$ to $f_a$ for
  all $b < a$.

  Note that the three horizontal triangles 
  in \eqref{eq:sigma-leq-a-N-slice-off}
  are of the type considered
  in the observation. Applying $\sigma_b$ (any $b \in \DZ$) to them
  yields triangles with vanishing connecting morphisms.
  Hence we omit these morphisms in the
  following diagram which is
  $\sigma_a$ applied to the nine upper left entries of
  \eqref{eq:sigma-leq-a-N-slice-off} (we use again some canonical
  identifications) 
  \begin{equation}
    \label{eq:sigma-leq-a-N-slice-off-sigma-a}
    \xymatrix@=1.5cm{
      {0} \ar[r] \ar[d] & 
      {{\sigma_a(L_a)}} \ar[r]^-{\sigma_a(v'_a)}_-\sim
      \ar[d]_-{\sigma_a(f'_a)} & 
      {{\sigma_{a}(M)}} \ar[d]^-{\id} 
      \ar@/_1pc/@{..>}[dl]^-{\delta_a} \\
      {{\sigma_a([-1]s(N))}} \ar[r]^-{\sigma_a(u_a)}
      \ar[d]^-{\id}  
      &  
      {{\sigma_a(Q_a)}} \ar[r]^-{\sigma_a(v_a)} 
      \ar[d]^-{\sigma_a(f_{a-1,a})}
      \ar@/^1pc/@{..>}[dl]_-{\epsilon_a}
      &
      {{\sigma_{a}(M)}} \ar[d] \\
      {{\sigma_a([-1]s(N))}} \ar[r]_-{\sigma_a(u_{a-1})}^-\sim & 
      {{\sigma_a(Q_{a-1})}} \ar[r] &
      {0}
    }
  \end{equation}
  and where
  the dotted arrows are defined by
  \begin{align*}
    \epsilon_a & := (\sigma_a(u_{a-1}))\inv \comp \sigma_a(f_{a-1,a}),\\
    \delta_a   & := \sigma_a(f'_a) \comp (\sigma_a(v'_a))\inv.
  \end{align*}
  So the four honest triangles with one dotted arrow commute, in
  particular 
  $\epsilon_a \comp \sigma_a(u_a)=\id$ and $\sigma_a(v_a)\comp
  \delta_a =\id$.
  Note that $\epsilon_a \comp \delta_a=0$ since the middle column in
  \eqref{eq:sigma-leq-a-N-slice-off-sigma-a} 
  is part of a triangle
  (this comes from axiom~\ref{enum:filt-tria-cat-3x3-diag} used
  above).
  Lemma~\ref{l:zero-triang-cat}
  gives an explicit splitting 
  of the middle
  row of 
  \eqref{eq:sigma-leq-a-N-slice-off-sigma-a}:
  \begin{equation}
    \label{eq:explicit-iso-sigma-a-Q-a}
    \xymatrix{
      {{\sigma_a([-1]s(N))}} \ar[r]^-{\sigma_a(u_a)}
      \gar[d] 
      &  
      {{\sigma_a(Q_a)}} \ar[r]^-{\sigma_a(v_a)} 
      \ar[d]^-{\svek{\epsilon_a}{\sigma_a(v_a)}}_-{\sim}
      &
      {{\sigma_{a}(M)}} \gar[d] \\
      {{\sigma_a([-1]s(N))}} \ar[r]^-{\svek 10}
      &  
      {{\sigma_a([-1]s(N))} \oplus \sigma_{a}(M)}
      \ar[r]^-{\zvek 01} 
      &
      {{\sigma_{a}(M)}}
    }
  \end{equation}
  and states that $\zvek{\sigma_a(u_a)}{\delta_a}$ is inverse to
  $\svek{\epsilon_a}{\sigma_a(v_a)}$.

  Our aim now is to compute the morphisms which will yield the
  differential of $c(Q)$ using this explicit direct sum
  decomposition.

  We explain the following diagram (which is commutative without the
  dotted arrows).
  \begin{equation}
    \label{eq:start-analysis-differential-Q}
    \xymatrix@=1.5cm{
      {\sigma_{a-1}([-1]s(N))} \ar[r]^-{\sigma_{a-1}(u_{a-1})} 
      & 
      {\sigma_{a-1}(Q_{a-1})} \ar[r]^-{\sigma_{a-1}(v_{a-1})} 
      \ar@/^2pc/@{..>}[l]^-{\epsilon_{a-1}}
      & 
      {{\sigma_{a-1}(M)}} 
      \ar@/_2pc/@{..>}[l]_-{\delta_{a-1}}
      \\
      {\sigma_{a-1}([-1]s(N))} \ar[r]^-{\sigma_{a-1}(u_a)}
      \gar[u]
      \ar[d]^{\tildew{d}^{a-1}_{[-1]s(N)}}
      & 
      {\sigma_{a-1}(Q_a)} \ar[r]^-{\sigma_{a-1}(v_a)}
      \ar[u]_-{\sigma_{a-1}(f_{a-1,a})}^-{\sim} 
      \ar[d]^{\tildew{d}^{a-1}_{Q_a}} & 
      {{\sigma_{a-1}(M)}} 
      \gar[u]
      \ar[d]^{\tildew{d}^{a-1}_{M}}
      \\
      {{[1]\sigma_a([-1]s(N))}} \ar[r]^-{[1]\sigma_a(u_a)}
       &  
      {{[1]\sigma_a(Q_a)}} \ar[r]^-{[1]\sigma_a(v_a)} 
      \ar@/^2pc/@{..>}[l]^-{[1]\epsilon_a}
      &
      {{[1]\sigma_{a}(M)}} 
      \ar@/_2pc/@{..>}[l]_-{[1]\delta_{a}}
    }
  \end{equation}
  The first two rows are (the horizontal mirror image of)
  $\sigma_{a-1}$ applied to the two middle rows of 
  \eqref{eq:sigma-leq-a-N-slice-off}.
  The last 
  row is $[1]$ applied to the middle row of
  \eqref{eq:sigma-leq-a-N-slice-off-sigma-a}. The morphism between
  second and third row is that from
  \eqref{eq:sigma-truncation-for-tilde-d}. 
  Note that we have split the first and third row explicitly before.
  
  This diagram shows that 
  \begin{equation*}
    \tildew{d}^{a-1}_{Q}: \sigma_{a-1}(Q) \ra [1]\sigma_a (Q)
  \end{equation*}
  has the form
  \begin{equation}
    \label{eq:partial-differential-Q}
    \begin{bmatrix}
      {\tildew{d}^{a-1}_{[-1]s(N)}} & {\kappa^{a-1}} \\ {0} &
      {{\tildew{d}^{a-1}_{M}}} 
    \end{bmatrix}
  \end{equation}
  if we identify 
  \begin{equation*}
    \xymatrix@C=2cm{
      {\sigma_{a-1}(Q)} \ar[r]^-{\sigma_{a-1}(f_{a-1})}_-{\sim} &
      {\sigma_{a-1}(Q_{a-1})}
      \ar[r]^-{\svek{\epsilon_{a-1}}{\sigma_{a-1}(v_{a-1})}}_-{\sim} & 
      {\sigma_{a-1}([-1]s(N)) \oplus \sigma_{a-1}(M)}
    }
  \end{equation*}
  and 
  \begin{equation*}
    \xymatrix@C=2cm{
      {[1]\sigma_a(Q)} \ar[r]^-{[1]\sigma_a(f_{a})}_-{\sim} &
      {[1]\sigma_a(Q_a)}
      \ar[r]^-{\svek{[1]\epsilon_{a}}{[1]\sigma_{a}(v_a)}}_-{\sim} & 
      {{{[1]\sigma_a([-1]s(N))} \oplus [1]\sigma_{a}(M)}}
    }
  \end{equation*}
  along 
  \eqref{eq:explicit-iso-sigma-a-Q-a}, for some morphism 
  \begin{equation*}
    \kappa^{a-1}: \sigma_{a-1}(M) \ra 
    [1]\sigma_a([-1]s(N));
  \end{equation*}
  which we will determine now; our aim is to prove
  \eqref{eq:kappa-identified} below.
 
  We apply to the commutative diagram
  \begin{equation*}
    \xymatrix{
      {L_{a-1}} \ar[r]^-{f'_{a-1}} &
      {Q_{a-1}} &
      {Q_a} \ar[l]^-{f_{a-1,a}} &
      {Q} \ar[l]^-{f_{a}} \ar@/_1.5pc/[ll]_-{f_{a-1}}
    }
  \end{equation*}
  the morphism ${\tildew{d}^{a-1}}: \sigma_{a-1} \ra
  [1]\sigma_a$ of functors and obtain the 
  middle part of
  the following commutative diagram (the rest will
  be explained below):
  \begin{equation*}
  \hspace{-0.8cm}
    \xymatrix@=1.5cm{
      {{\sigma_{a-1}(M)}} \ar[dr]^-{\delta_{a-1}}\\
      {\sigma_{a-1}(L_{a-1})} \ar[r]^-{\sigma_{a-1}(f'_{a-1})} 
      \ar[d]^-{\tildew{d}^{a-1}_{L_{a-1}}} 
      \ar[u]^-{\sigma_{a-1}(v'_{a-1})}_-{\sim} &
      {\sigma_{a-1}(Q_{a-1})} 
      \ar[d]^-{\tildew{d}^{a-1}_{Q_{a-1}}} &
      {\sigma_{a-1}(Q_a)} \ar[l]_-{\sigma_{a-1}(f_{a-1,a})}^-{\sim} 
      \ar[d]^-{\tildew{d}^{a-1}_{Q_{a}}} &
      {\sigma_{a-1}(Q)} \ar[l]_-{\sigma_{a-1}(f_{a})}^-{\sim} 
      \ar@/_2.5pc/[ll]_-{\sigma_{a-1}(f_{a-1})}^-{\sim}
      \ar[d]^-{\tildew{d}^{a-1}_{Q}} \\
      {[1]\sigma_a(L_{a-1})} \ar[r]^-{[1]\sigma_a(f'_{a-1})}_-{\sim} &
      {[1]\sigma_a(Q_{a-1})} &
      {[1]\sigma_a(Q_a)} \ar[l]_-{[1]\sigma_a(f_{a-1,a})} 
      \ar[dl]^-{[1]\epsilon_a} &
      {[1]\sigma_a(Q)} \ar[l]_-{[1]\sigma_a(f_{a})}^-{\sim}\\
      {{[1]\sigma_a([-1]s(N))}} \ar[u]_-{[1]\sigma_a(u'_{a-1})}^-\sim
      \gar[r] &
      {{[1]\sigma_a([-1]s(N))}} \ar[u]_-{[1]\sigma_a(u_{a-1})}^-\sim 
    }
  \end{equation*}
  The honest triangles with diagonal sides $\delta_{a-1}$ and $[1]\epsilon_a$
  respectively are commutative by
  \eqref{eq:sigma-leq-a-N-slice-off-sigma-a}. The 
  lower left square commutes since it is (up to rotation)
  $[1]\sigma_a$ applied to 
  the upper left square 
  in \eqref{eq:sigma-leq-a-N-slice-off} (after substituting $a-1$ for
  $a$ there). Note that $[1]\sigma_a(u'_{a-1})$ is an isomorphism
  since $\sigma_{a}(\sigma_{a-1}(M))=0$.
  It is immediate from this diagram that $\kappa^{a-1}$ is the
  downward vertical composition in the left column of this diagram,
  i.\,e.
  \begin{equation}
    \label{eq:kappa-nearly-identified}
    \kappa^{a-1}= ([1]\sigma_a(u'_{a-1}))\inv \comp
    \tildew{d}^{a-1}_{L_{a-1}} \comp (\sigma_{a-1}(v'_{a-1}))\inv.
  \end{equation}

  Recall that the considerations 
  in the first case
  gave formula 
  \eqref{eq:differential-eL-gr-supp-singleton};
  if we apply them to the morphism $\sigma_{a-1}(m):\sigma_{a-1}(M)
  \ra \sigma_{a-1}(N)$ 
  in $\tildew{\mathcal{T}}^s([a-1])$ and the top horizontal
  triangle in  
  \eqref{eq:sigma-leq-a-N-slice-off} (with $a$ replaced by $a-1$)
  (which plays the role of
  \eqref{eq:alpha-triangle-MN-gr-a}), 
  this formula describes the right hand side of 
  \eqref{eq:kappa-nearly-identified}
  and hence yields
  \begin{equation}
    \label{eq:kappa-identified}
    \kappa^{a-1}= \alpha_{\sigma_{a-1}(N)} \comp \sigma_{a-1}(m).  
  \end{equation}
  
  Let us sum up what we know:
  We use from now on tacitly the  
  identifications
  \begin{equation*}
    \sigma_a(Q)\sira 
    \sigma_a([-1]s(N)) \oplus \sigma_{a}(M)
  \end{equation*}
  given by $\sigma_a(f_a)$ and \eqref{eq:explicit-iso-sigma-a-Q-a}. 
  Then the morphism
  $\tildew{d}^{a}_{Q}:  \sigma_{a}(Q) \ra [1]\sigma_{a+1} (Q)$
  is given by the matrix
  \begin{equation*}
    \tildew{d}^{a}_{Q} = 
    \begin{bmatrix}
      {\tildew{d}^{a}_{[-1]s(N)}} & {\alpha_{\sigma_a(N)} \comp
        \sigma_{a}(m)} \\  
      {0} & {{\tildew{d}^{a}_{M}}} 
    \end{bmatrix}.
  \end{equation*}
  Shifting accordingly this describes the complex $c'(Q)$
  completely. Moreover, the 
  morphisms $c'(u): c'([-1]s(N)) \ra c'(Q)$ 
  and $c'(v):c'(Q) \ra c'(M)$ become identified with the inclusion
  $\svek 10$ of the first summand and the projection 
  $\svek 01$ onto the second summand respectively.

  Then it is clear that $c(Q)=c''(c'(Q))$ 
  is given by (we assume that $i$
  is just the inclusion $\tildew{\mathcal{T}}([0]) \subset
  \tildew{\mathcal{T}}$):
  \begin{align*}
    c(Q)^a & =c([-1]s(N))^{a} \oplus c(M)^{a},\\
    d_{c(Q)}^a & =
    \begin{bmatrix}
      {d^{a}_{c([-1]s(N))}} & {c(m)^a}\\ 
      0 & {d^{a}_{c(M)}}
    \end{bmatrix}.
  \end{align*}

  Using the canonical identification $c([-1]s(N))\cong \Sigma\inv
  c(N)$ we see that $c$ maps
  \eqref{eq:alpha-lift-triangle-MN} to
  \eqref{eq:h-triangulated}.
  As before this becomes a triangle
  in $K^b(\heart(w))^\anti$, 
  cf.\ \eqref{eq:mapcone-minus-m-rot2}.
\end{proof}

\begin{corollary}
  \label{c:e-prime-zero-on-kernel}
  For all $M, N \in \tildew{\mathcal{T}}^s$, 
  \begin{equation*}
    h:\Hom_{\tildew{\mathcal{T}}^s}(M,N) \ra
    \Hom_{K^b(\heart(w))^\anti}(h(M), h(N))
  \end{equation*}
  vanishes
  on
  the kernel of 
  $\omega|_{\tildew{\mathcal{T}}^s}:\Hom_{\tildew{\mathcal{T}}^s}(M,N) \ra
  \Hom_{\mathcal{T}}(\omega(M), \omega(N))$.
\end{corollary}

\begin{proof}
  (We argue similarly as in the proof of Proposition
  \ref{p:omega-on-subtildeT}.)
  Let $M, N \in \tildew{\mathcal{T}}^s$.
  By \ref{enum:filt-tria-cat-exhaust} we find $m, n \in \DZ$ such that
  $M \in \tildew{\mathcal{T}}(\leq m)$ and $N \in
  \tildew{\mathcal{T}}(\geq n)$. 
  Choose $a \in \DZ$ satisfying $a \geq m-n$.
  Consider the following commutative diagram 
  where the epimorphisms and isomorphisms in the upper row come
  from Lemma~\ref{l:alpha-hom-subtildeT}, the isomorphisms in the
  lower row from Proposition~\ref{p:functor-omega},
  \ref{enum:functor-omega-iii}, and the vertical morphisms are
  application of $\omega$;
  the vertical morphism on the right is an isomorphism by
  Proposition~\ref{p:functor-omega}, \ref{enum:functor-omega-iv}:
  \begin{equation*}
    \xymatrix{
      {\tildew{\mathcal{T}}(M,N)} \sar[r]^{\alpha_N \comp ?} \ar[d]^{\omega} &
      {\tildew{\mathcal{T}}(M,s(N))} \ar[r]^-{\sim} \ar[d]^{\omega} &
      {\tildew{\mathcal{T}}(M,s^a(N))} \ar[d]^{\omega}_{\sim} \\
      {{\mathcal{T}}(\omega(M),\omega(N))} \ar[r]^-{\sim} &
      {{\mathcal{T}}(\omega(M),\omega(s(N)))} \ar[r]^-{\sim} &
      {{\mathcal{T}}(\omega(M),\omega(s^a(N)))} 
    }
  \end{equation*}
  This diagram shows that the kernel of the vertical arrow $\omega$ on
  the left coincides with the kernel of the horizontal arrow
  $(\alpha_N \comp ?)$.

  Let $f: M \ra N$ be a morphism in $\tildew{\mathcal{T}}^s$ and
  assume that $\omega(f)=0$. 
  Since $\alpha_N \comp f=0$ the morphism $f$ factors to $f'$ as
  indicated in the following commutative diagram
  \begin{equation*}
    \xymatrix{
      &&
      {M} \ar[d]^f \ar@{..>}[ld]_{f'}
      \\
      {[-1]s(N)} \ar[r]^-{u} &
      {Q} \ar[r]^-{v} &
      {N} \ar[r]^-{\alpha_N} &
      {s(N),}
    }
  \end{equation*}
  where the lower horizontal row is a completion of $\alpha_N$ into a triangle.
  Proposition~\ref{p:cone-alpha-homotopic-to-zero} shows that
  $h(Q)=0$ and hence $h(f)=0$.
\end{proof}

\begin{corollary}
  \label{c:h-factors-triang}
  As indicated in diagram~\eqref{eq:omega-subtildeT-exp},
  the functor $h$ of additive categories with translation factors
  uniquely to a triangulated functor
  \begin{equation*}
    \ol{h}: (\mathcal{Q},[1]s\inv, \Delta_\mathcal{Q}) \ra K^b(\heart(w))^\anti.
  \end{equation*}
\end{corollary}

\begin{proof}
  Corollary~\ref{c:e-prime-zero-on-kernel} shows that $h$ factors
  uniquely to a functor $\ol{h}$ of additive categories with
  translation. 
  Proposition~\ref{p:cone-alpha-homotopic-to-zero}
  together with the description of the class of triangles
  $\Delta_\mathcal{Q}$ before Lemma~\ref{l:omega-bar-triang} show that
  $\ol{h}$ is a triangulated functor.
\end{proof}

\begin{proof}
  [Proof of Thm.~\ref{t:strong-weight-cplx-functor}]

  We first use Proposition~\ref{p:omega-on-subtildeT}:
  For any object $X \in \mathcal{T}$ we fix an object $\tildew{X} \in
  \tildew{\mathcal{T}}^s$ and an isomorphism $X \cong
  \omega(\tildew{X})$. Then we fix for any
  morphism $f:X \ra Y$ in 
  $\mathcal{T}$ a morphism $\tildew{f}:\tildew{X} \ra \tildew{Y}$ in
  $\tildew{\mathcal{T}}^s$ such that $f$ corresponds
  to $\omega(\tildew{f})$ under the isomorphisms 
  $X \cong \omega(\tildew{X})$ and
  $Y \cong \omega(\tildew{Y})$.
  Mapping $X$ to $\tildew{X}$ and $f$ to the class of $\tildew{f}$
  in $\mathcal{Q}$ defines a quasi-inverse $\ol{\omega}\inv$ to
  $\ol{\omega}$
  (and any quasi-inverse is of this form). 
  We claim that $\WCfun:=\ol{h} \comp \ol{\omega}\inv$ 
  (cf.\ diagram~\eqref{eq:omega-subtildeT-exp})
  is a strong weight complex functor.
  Lemma~\ref{l:omega-bar-triang} and
  Corollary~\ref{c:h-factors-triang} show that it is a triangulated
  functor. We have to show that its 
  composition 
  with the canonical functor $K(\heart(w))^\anti \ra
  {K_\weak(\heart(w))}$ is isomorphic to a weak weight complex functor.

  Observe that the constructions of the weak weight complex functor
  (see Section~\ref{sec:weak-wc-fun})
  and of $c'$ 
  (see Section~\ref{sec:first-step-cprime})
  are almost parallel under $\omega$:
  Lemma~\ref{l:omega-range-versus-weights-on-tildeTs} shows that
  $\omega$ maps $\sigma$-truncation
  triangles of objects $\tildew{X} \in \tildew{\mathcal{T}}^s$ to
  weight decompositions. This means that we can take the image
  $\omega(S^n_{\tildew{X}})$ of 
  $S^n_{\tildew{X}}$ (see \eqref{eq:sigma-trunc-construction-c})
  under $\omega$ as 
  a preferred choice for the triangle $T^n_X$ (see
  \eqref{eq:choice-weight-decomp}); more precisely we have to replace
  $\omega(\tildew{X})$ in $\omega(S^n_{\tildew{X}})$ by $X$. 
  Similarly 
  the octahedron $\tildew{O}^n_{\tildew{X}}$ (see \eqref{eq:octaeder-cprime})
  yields $\omega(\tildew{O}^n_{\tildew{X}})$ as a preferred choice for
  the octahedron 
  $O^n_X$ (see \eqref{eq:wc-weak-octaeder}): We have
  $w_nX=\omega(\sigma_n(X))$. Since the octahedron
  $\tildew{O}^n_X$ 
  is functorial we immediately get preferred choices for the
  morphisms $f^n$ in \eqref{eq:f-w-trunc-fn}, namely
  $\omega(\sigma_n(\tildew{f}))$. 
  These preferred choices define 
  the assignment $X \mapsto \candidateWCweak(X)$, $f \mapsto 
  \candidateWCweak(f)=([n]\omega(\sigma_n(\tildew{f})))_{n \in \DZ}$.
  Passing to the weak homotopy category defines 
  a weak weight complex functor
  $\WCweak:\mathcal{T} \ra K_\weak(\heart(w))$ (see Thm.~\ref{t:weakWCfun}).

  Let $\can:C^b(\heart) \ra K^b(\heart(w))^\anti$ be the obvious functor.
  Then we have (using Rem.~\ref{rem:quick-def-c})
  \begin{equation*}
    \can \comp \candidateWCweak \comp \omega = \can \comp \omega_{C^b}
    \comp c' \cong \can 
    \comp c'' \comp c' =\can \comp c = h
  \end{equation*}
  on objects and morphisms and, since $h$ is a functor,
  as functors $\tildew{\mathcal{T}}^s \ra K^b(\heart(w))^\anti$.
  This implies that $\can \comp \candidateWCweak \cong \ol{h} \comp
  \ol{\omega}\inv = \WCfun$.
  The composition of $\can \comp \candidateWCweak$
  with $K^b(\heart(w))^\anti \ra K_\weak(\heart(w))$ 
  is the weak weight complex functor $\WCweak$ from above. Hence
  $\WCfun$
  is a strong weight complex functor.
\end{proof}

\begin{remark}
  \label{rem:special-choice-omega-inverse}
  At the beginning of the proof of Theorem
  \ref{t:strong-weight-cplx-functor} we chose
  some objects $\tildew{X}$.
  Let us take some more care there:
  For $X=0$ choose $\tildew{X}=0$. Assume $X\not=0$
  Then let $a,b \in \DZ$ be such
  that $X \in \mathcal{T}^{w \in [a,b]}$ and $b-a$ is minimal. Then
  Proposition~\ref{p:omega-on-subtildeT} allows us to 
  find an object $\tildew{X} \in \tildew{\mathcal{T}}^s([a,b])$
  and an isomorphism $X \cong \omega(\tildew{X})$. 
  Taking these choices and proceeding as in the above proof it is
  obvious that $\WCfun:\mathcal{T} \ra K^b(\heart(w))^\anti$ maps
  $\mathcal{T}^{w \in [a,b]}$ to $K^{[a,b]}(\heart(w))$.
\end{remark}

\section{Lifting weight structures to f-categories} 
\label{sec:weight-structures-and-filtered-triang}

We show some statements about compatible weight structures. 
For the corresponding and motivating results for t-structures see
\cite[Prop.~A 5 a]{Beilinson}. 

\begin{definition}
  Let $(\tildew{\mathcal{T}},i)$ be an f-category over a triangulated
  category $\mathcal{T}$. Assume that both $\tildew{\mathcal{T}}$ and
  $\mathcal{T}$ are weight categories (see Def.~\ref{d:ws}), i.\,e.\
  they are equipped with weight structures. Then these weight structures are
  \define{compatible}  
  if\footnote
  {
    Condition \ref{enum:wstr-ft-i} is natural whereas 
    \ref{enum:wstr-ft-s} is perhaps a priori not clear. Here is a
    (partial) justification. Take an object $0\not= X \in \heart(w)$
    (where $w$ is the w-structure on $\mathcal{T}$) and
    denote $i(X)$ also by $X$. Then we have the nonzero morphism
    $\alpha_X: X \ra s(X)$. Now $X \in \tildew{\mathcal{T}}^{w \geq 0}$
    and $s(X) \in s(\tildew{\mathcal{T}}^{w \leq 0})$. So we cannot have
    $s(\tildew{\mathcal{T}}^{w \leq 0})= \tildew{\mathcal{T}}^{w \leq -1}$.
  }
  \begin{enumerate}[label=(wcomp-ft{\arabic*})]
  \item 
    \label{enum:wstr-ft-i}
    $i: \mathcal{T} \ra \tildew{\mathcal{T}}$ is w-exact and
  \item
    \label{enum:wstr-ft-s}
    $s(\tildew{\mathcal{T}}^{w \leq 0}) = \tildew{\mathcal{T}}^{w \leq
      1}$.
  \end{enumerate}
  Note that the at first sight asymmetric condition
  \ref{enum:wstr-ft-s} implies its counterpart
  \begin{enumerate}[resume,label=(wstr-ft{\arabic*})]
  \item
    \label{enum:wstr-ft-s-symm}
    $s(\tildew{\mathcal{T}}^{w \geq 0})
    = \tildew{\mathcal{T}}^{w \geq 1}$,
  \end{enumerate}
  as we prove in Remark~\ref{rem:wstr-ft-symm} below.
\end{definition}

\begin{remark}
  \label{rem:wstr-ft-symm}
  Assume that 
  $\tildew{\mathcal{T}}$ together with 
  $i:\mathcal{T} \ra \tildew{\mathcal{T}}$ is an f-category over the
  triangulated category $\mathcal{T}$, and that both
  $\tildew{\mathcal{T}}$ and $\mathcal{T}$ are equipped with
  compatible w-structures.
  Then \eqref{eq:ws-geq-left-orth-of-leq} and 
  \ref{enum:wstr-ft-s} 
  show that
  \begin{equation*}
    \tildew{\mathcal{T}}^{w \geq 2} 
    = \leftidx{^{\perp}}{(\tildew{\mathcal{T}}^{w \leq 1})}{}
    = \leftidx{^{\perp}}{(s(\tildew{\mathcal{T}}^{w \leq 0}))}{}
    = s(\leftidx{^{\perp}}{(\tildew{\mathcal{T}}^{w \leq 0})}{})
    = s(\tildew{\mathcal{T}}^{w \geq 1})
  \end{equation*}
  Now apply the translation $[1]$ to get
  \ref{enum:wstr-ft-s-symm}
\end{remark}

The statement of the following proposition appears independently in
\cite[Prop.~3.4 (1)]{achar-kitchen-koszul-mixed-arXiv} where the proof is
essentially left to the reader.

\begin{proposition}
  \label{p:compatible-w-str}
  Let $(\tildew{\mathcal{T}},i)$ be an f-category over a triangulated
  category $\mathcal{T}$. Given a w-structure $w=(\mathcal{T}^{w \leq
    0}, \mathcal{T}^{w \geq 0})$ on $\mathcal{T}$, there
  is a unique w-structure $\tildew{w}=(\tildew{\mathcal{T}}^{w \leq
    0}, \tildew{\mathcal{T}}^{w \geq 0})$ on $\tildew{\mathcal{T}}$ that
  is compatible 
  with the given w-structure on $\mathcal{T}$. It is given by
  \begin{align}
    \label{eq:tildeT-wstr}
    \tildew{\mathcal{T}}^{w \leq 0} & =\{X \in \tildew{\mathcal{T}}\mid
    \text{$\gr^j(X) \in \mathcal{T}^{w \leq -j}$ for all $j \in
      \DZ$}\},\\
    \notag
    \tildew{\mathcal{T}}^{w \geq 0} & =\{X \in
    \tildew{\mathcal{T}}\mid
    \text{$\gr^j(X) \in \mathcal{T}^{w \geq -j}$ for all $j \in \DZ$}\}.
  \end{align}
\end{proposition}
From 
\eqref{eq:tildeT-wstr} we get
\begin{align}
  \label{eq:tildeT-wstr-shifts}
  \tildew{\mathcal{T}}^{w \leq n} & 
  := [-n]\tildew{\mathcal{T}}^{w \leq 0} 
  =\{X \in \tildew{\mathcal{T}}\mid
  \text{$\gr^j(X) \in \mathcal{T}^{w \leq n-j}$ for all $j \in
    \DZ$}\},\\
  \notag
  \tildew{\mathcal{T}}^{w \geq n} & 
  := [-n]\tildew{\mathcal{T}}^{w \geq 0} 
  =\{X \in \tildew{\mathcal{T}}\mid
  \text{$\gr^j(X) \in \mathcal{T}^{w \geq n-j}$ for all $j \in \DZ$}\}.
\end{align}
The heart of the w-structure of the proposition is
\begin{equation}
  \label{eq:heart-comp-w-str}
  \heart(\tildew{w}) = \tildew{\mathcal{T}}^{w =0} 
  =\{X \in \tildew{\mathcal{T}}\mid
  \text{$\gr^j(X) \in \mathcal{T}^{w =-j}$ for all $j \in \DZ$}\}.
\end{equation}

\begin{proof}
  \textbf{Uniqueness:}
  We assume that we already know that 
  \eqref{eq:tildeT-wstr}
  is a compatible w-structure.
  Let $X$ be any object in $\tildew{\mathcal{T}}$. 
  Then $X$ is in $\tildew{\mathcal{T}}([a,b])$ for some $a, b \in
  \DZ$. 
  We can build up $X$ from its graded pieces as indicated in the
  following diagram 
  in the case $[a,b]=[-2,1]$:
  \begin{equation}
    \label{eq:build-up-X2}
    \xymatrix@C-1.3cm{
      {X=\sigma_{\leq 1}X} \ar[rr] 
      && {\sigma_{\leq 0}X} \ar[rr] \ar@{~>}[ld]
      && {\sigma_{\leq -1}X} \ar[rr] \ar@{~>}[ld]
      && {\sigma_{\leq -2}X} \ar[rr] \ar@{~>}[ld]
      && {\sigma_{\leq -3}X=0} \ar@{~>}[ld]
      \\
      & {s(i(\gr^1(X)))} \ar[lu]
      && {i(\gr^0(X))} \ar[lu]
      && {s^{-1}(i(\gr^{-1}(X)))} \ar[lu]
      && {s^{-2}(i(\gr^{-2}(X)))} \ar[lu]
    }
  \end{equation}
  All triangles are isomorphic to $\sigma$-truncation triangles with the wiggly
  arrows of degree one.
 
  Assume that $(\mathcal{C}^{w \leq 0},
  \mathcal{C}^{w \geq 0})$ is a compatible w-structure on
  $\tildew{\mathcal{T}}$.
  We claim that
  \begin{equation*}
    \tildew{\mathcal{T}}^{w \leq 0} \subset \mathcal{C}^{w \leq 0} 
    \text{ and }
    \tildew{\mathcal{T}}^{w \geq 0} \subset \mathcal{C}^{w \geq 0}
  \end{equation*}
  which implies equality of the two w-structures 
  (see Lemma \ref{c:inclusion-w-str}). 
  Assume that $X$ is in $\tildew{\mathcal{T}}^{w \leq 0}$. Then
  $\gr^j(X) \in \mathcal{T}^{w \leq -j}$ implies
  $i(\gr^j(X)) \in \mathcal{C}^{w \leq -j}$ and we obtain
  $s^j(i(\gr^j(X))) \in \mathcal{C}^{w \leq 0}$.
  Now an induction using \eqref{eq:build-up-X2} (respectively its
  obvious generalization to arbitrary $[a,b]$) shows that $X$ is in
  $\mathcal{C}^{w \leq 0}$. Analogously we show
  $\tildew{\mathcal{T}}^{w \geq 0} \subset \mathcal{C}^{w \geq 0}$.

  \textbf{Existence:}
  If 
  \eqref{eq:tildeT-wstr} defines a w-structure 
  compatibility is obvious:
  \ref{enum:wstr-ft-i} is clear
  and \ref{enum:wstr-ft-s} follows from
  $\gr^j \comp s= \gr^{j-1}$ (see
  \eqref{eq:gr-s-comm}).

  We prove that 
  \eqref{eq:tildeT-wstr} defines a w-structure on $\tildew{\mathcal{T}}$.
  
  Condition \ref{enum:ws-i} holds:
  All functors $\gr^j$ are triangulated and in particular
  additive. 
  Since all $\mathcal{T}^{w\leq j}$ and $\mathcal{T}^{w \geq j}$ are
  additive categories and closed under retracts in $\mathcal{T}$,
  $\tildew{\mathcal{T}}^{w \leq 0}$
  and $\tildew{\mathcal{T}}^{w \geq 0}$ are additive categories and
  closed under retracts in $\tildew{\mathcal{T}}$.

  Condition \ref{enum:ws-ii} follows from 
  \eqref{eq:tildeT-wstr-shifts}.
  
  Condition \ref{enum:ws-iii}:
  Let $X \in \tildew{\mathcal{T}}^{w \geq 1}$ and 
  $Y \in \tildew{\mathcal{T}}^{w \leq 0}$.
  It is obvious from \eqref{eq:tildeT-wstr-shifts} that 
  all $\tildew{\mathcal{T}}^{w \geq n}$ 
  and $\tildew{\mathcal{T}}^{w \leq n}$ are stable under all
  $\sigma$-truncations.
  Hence we can first $\sigma$-truncate $X$ and reduce to the case that
  $l(X)\leq 1$ and then similarly reduce to $l(Y) \leq 1$.
  So it is sufficient to prove $\Hom(X,Y)=0$ for 
  $X \in \tildew{\mathcal{T}}([a])$
  and $Y \in \tildew{\mathcal{T}}([b])$ for arbitrary $a, b \in \DZ$. 
  Let $f: X\ra Y$ be a morphism.
  \begin{itemize}
  \item 
    If $b<a$ then $f=0$ by \ref{enum:filt-tria-cat-no-homs}.
  \item 
    Let $b=a$. Then $f=\sigma_a(f)$ (under the obvious identification) and
    $\sigma_a(f) \cong s^a(i(\gr^a(f)))$.
    But $\gr^a(f):\gr^a(X) \ra \gr^a(Y)$ is zero since 
    $\gr^a(X) \in \mathcal{T}^{w \geq 1-a}$ and
    $\gr^a(Y) \in \mathcal{T}^{w \leq -a}$.

  \item 
    Let $b > a$. Then $f$ can be factorized as $X \xra{f'}
    s^{-(b-a)}(Y)\xra{\alpha^{b-a}} Y$ for a unique $f'$ using
    \ref{enum:filt-tria-cat-hom-bij}.
    Then 
    \begin{equation*}
      s^{-(b-a)}(Y) \in s^{-(b-a)}(\tildew{\mathcal{T}}^{w \leq
        0}([b]))=\tildew{\mathcal{T}}^{w \leq a-b}([a])
      \subset \tildew{\mathcal{T}}^{w \leq 0}([a])
    \end{equation*}
    and the case $b=a$ imply that $f'=0$ and hence $f=0$.
  \end{itemize}
  
  Condition \ref{enum:ws-iv}: 
  By induction on $b-a$ we prove the following statement
  (which is sufficient by \ref{enum:filt-tria-cat-exhaust}):
  Let $X$ be in $\tildew{\mathcal{T}}([a,b])$. Then for each $n \in
  \DZ$ there are triangles
  \begin{equation}
    \label{eq:taun-comp}
    w_{\geq n+1} X \ra X \ra w_{\leq n}X \ra [1] w_{\geq n+1} X
  \end{equation}
  with 
  $w_{\geq n+1} X \in \tildew{\mathcal{T}}^{w \geq n+1}([a,b])$ and
  $w_{\leq n}X \in \tildew{\mathcal{T}}^{w \leq n}([a,b])$
  and satisfying
  \begin{align}
    \label{eq:omega-tau}
    \omega (w_{\geq n+1} X) & \in \mathcal{T}^{w \geq n+1-b},\\
    \notag
    \omega (w_{\leq n} X) & \in \mathcal{T}^{w \leq n-a}.
  \end{align}

  For $b-a<0$ we choose everything to be zero.
  Assume $b-a=0$. Then the object $s^{-a}(X)$ is isomorphic to $i(Y)$
  for some $Y$ in $\mathcal{T}$. 
  Let $m\in \DZ$ and let 
  \begin{equation*}
    w_{\geq m+1}Y \ra Y \ra w_{\leq m}Y \ra [1]w_{\geq m+1}Y
  \end{equation*}
  be a $(w \geq m+1, w\leq m)$-weight decomposition of $Y$ in
  $\mathcal{T}$ with respect to $w$.
  Applying the triangulated functor $s^a \comp i$ we obtain (using an
  isomorphism 
  $s^{-a}(X)\cong i(Y)$) a triangle
  \begin{equation}
    \label{eq:taun-comp-aa}
    s^a(i(w_{\geq m+1}Y)) \ra X \ra s^a(i(w_{\leq m}Y)) \ra
    [1]s^a(i(w_{\geq m+1}Y)). 
  \end{equation}
  Since $i(w_{\geq m+1} Y) \in \tildew{\mathcal{T}}^{w \geq m+1}([0])$ we have
  $s^a(i(w_{\geq m+1} Y)) \in \tildew{\mathcal{T}}^{w \geq
    m+1+a}([a])$, and similarly 
  $s^a(i(w_{\leq m}Y)) \in \tildew{\mathcal{T}}^{w \leq m+a}([a])$.
  Take now $m=n-a$ and define the triangle
  \eqref{eq:taun-comp} to be
  \eqref{eq:taun-comp-aa}.
  Then \eqref{eq:omega-tau} is satisfied by 
  \ref{enum:functor-omega-iii}.
    
  Now assume $b>a$. Choose $a \leq c < b$. In the 
  diagram
  \begin{equation}
    \label{eq:construct-t-gr-supp}
    \xymatrix{
      {[1]w_{\geq n+1}\sigma_{\geq c+1}X} \ar[r] 
      & {[1]\sigma_{\geq c+1}X} \ar[r]
      & {[1]w_{\leq n}\sigma_{\geq c+1}X} \ar[r]^-{\anticomm}
      & {[2]w_{\geq n+1}\sigma_{\geq c+1}X} \\
      {w_{\geq n+1}\sigma_{\leq c}X} \ar[r] \ar@{~>}[u]
      & {\sigma_{\leq c}X} \ar[r] \ar[u]
      & {w_{\leq n}\sigma_{\leq c}X} \ar[r] \ar@{~>}[u]
      & {[1]w_{\geq n+1}\sigma_{\leq c}X} \ar@{~>}[u]^{[1]}
    }
  \end{equation}
  the vertical arrow in the middle is the last arrow in the
  $\sigma$-truncation triangle $\sigma_{\geq c+1}X \ra X \ra \sigma_{\leq c}X \ra
  [1]\sigma_{\geq c+1}X$. Using induction the lower row is 
  \eqref{eq:taun-comp} for $\sigma_{\leq c}X$ and the upper row is
  $[1]$ applied to \eqref{eq:taun-comp} for $\sigma_{\geq c+1}X$
  where we also multiply the last arrow by $-1$ as indicated by
  $\anticomm$; then this row is a triangle.
  In order to construct the indicated completion to a morphism of
  triangles
  we claim that
  \begin{equation}
    \label{eq:tr-filt-map-zero}
    \Hom([\epsilon]w_{\geq n+1}\sigma_{\leq c}X, [1]w_{\leq
      n}\sigma_{\geq c+1}X)=0 
    \text{ for $\epsilon \leq 2$.} 
  \end{equation}
  Since the left entry is in $\tildew{\mathcal{T}}(\leq c)$ and the
  right entry is in $\tildew{\mathcal{T}}(\geq c+1)$,
  it is sufficient by \ref{enum:functor-omega-iv}
  to show that
  \begin{equation*}
    \Hom([\epsilon]\omega (w_{\geq n+1}(\sigma_{\leq c}(X))),
    [1]\omega (w_{\leq n}(\sigma_{\geq c+1}(X))))=0  
    \text{ for $\epsilon \leq 2$.}
  \end{equation*}
  But this is true by axiom
  \ref{enum:ws-iii} 
  since the left entry
  is in $\mathcal{T}^{w \geq n+1-c-\epsilon} 
  \subset \mathcal{T}^{w \geq n-c-1}$ and
  the right entry is in 
  $\mathcal{T}^{w \leq n-(c+1)-1}=\mathcal{T}^{w \leq n-c-2}$ by
  \eqref{eq:omega-tau}.
  Hence claim \eqref{eq:tr-filt-map-zero} is proved.

  This claim for $\epsilon \in \{0,1\}$ and
  Proposition~\ref{p:BBD-1-1-9-copied-for-w-str} show that we can
  complete the arrow in \eqref{eq:construct-t-gr-supp} uniquely to a morphism
  of triangles as indicated by the squiggly arrows.

  Now multiply the last morphism in the first row of 
  \eqref{eq:construct-t-gr-supp} by $-1$;
  this makes the square on the right anti-commutative.
  Proposition~\ref{p:3x3-diagram-copied-for-w-str} and the uniqueness
  of the squiggly arrows show that this
  modified diagram fits as the first two rows into a $3\times
  3$-diagram
  \begin{equation*}
    \xymatrix{
      {[1]w_{\geq n+1}\sigma_{\geq c+1}X} \ar[r]
      & {[1]\sigma_{\geq c+1}X} \ar[r]
      & {[1]w_{\leq n}\sigma_{\geq c+1}X} \ar[r]
      \ar@{}[dr]|-{\anticomm}
      & {[2]w_{\geq n+1}\sigma_{\geq c+1}X } \\
      {w_{\geq n+1}\sigma_{\leq c}X} \ar[r] \ar@{~>}[u]
      & {\sigma_{\leq c}X} \ar[r] \ar[u]
      & {w_{\leq n}\sigma_{\leq c}X} \ar[r] \ar@{~>}[u]
      & {[1]w_{\geq n+1}\sigma_{\leq c}X} \ar@{~>}[u]\\
      {A} \ar[r] \ar[u]
      & {X} \ar[r] \ar[u]
      & {B} \ar[r] \ar[u]
      & {[1]A} \ar[u]\\
      {w_{\geq n+1}\sigma_{\geq c+1}X} \ar[r]  \ar[u]
      & {\sigma_{\geq c+1}X} \ar[r]  \ar[u]
      & {w_{\leq n}\sigma_{\geq c+1}X} \ar[r] \ar[u]
      & {[1]w_{\geq n+1}\sigma_{\geq c+1}X } \ar[u] 
    }
  \end{equation*}
  where we can assume that the second column is the
  $\sigma$-truncation triangle of $X$.

  We claim that we can define $w_{\geq n+1}X:=A$ and $w_{\leq n}X:=B$
  and that we can take the horizontal triangle $(A,X,B)$ as
  \eqref{eq:taun-comp}: 
  Apply the triangulated functor $\gr^j$ to the vertical column
  containing $A$. For $j \leq c$ this yields an isomorphism
  \begin{equation*}
    \gr^j A \sira \gr^j w_{\geq n+1}\sigma_{\leq c}X \in
    \mathcal{T}^{w \geq n+1-j},
  \end{equation*}
  and for $j > c$ we obtain isomorphisms
  \begin{equation*}
    \gr^j A \sila \gr^j w_{\geq n+1}\sigma_{\geq c+1}X \in
    \mathcal{T}^{w \geq n+1-j}.
  \end{equation*}
  This shows $A \in \tildew{\mathcal{T}}^{w \geq n+1}([a,b])$ where
  the statement about the range comes from the fact that the first
  isomorphism is zero for $j<a$ and the second one is zero for $j>b$. 
  Furthermore, if we apply $\omega$ to the column containing $A$, we
  obtain a triangle
  \begin{equation*}
    \xymatrix{
      {\omega (w_{\geq n+1}(\sigma_{\geq c+1}(X)))} \ar[r] 
      & 
      {\omega(A)} \ar[r]
      &
      {\omega (w_{\geq n+1}(\sigma_{\leq c}(X)))} \ar[r] 
      &
      {[1]\omega (w_{\geq n+1}(\sigma_{\geq c+1}(X)))}
    }
  \end{equation*}
  in which the first term is in $\mathcal{T}^{w \geq n+1-b}$ and the third
  term is in $\mathcal{T}^{w\geq n+1-c} \subset \mathcal{T}^{w\geq
    n+1-b}$. Hence $\omega(A) =\omega(w_{\geq n+1}X)$ is 
  in $\mathcal{T}^{w \geq n+1-b}$
  by Lemma~\ref{l:weight-str-basic-properties}
  \eqref{enum:weight-perp-prop}.
  Similarly we treat $B$. 
\end{proof}


\begin{thebibliography}{MSSSS10}

\bibitem[AI10]{aihara-iyama-silting-arXiv}
Takuma Aihara and Osamu Iyama.
\newblock Silting mutation in triangulated categories.
\newblock {\em Preprint}, 2010.
\newblock \href{http://arxiv.org/abs/1009.3370}{arXiv:1009.3370v2 [math.RT]}.

\bibitem[AK11]{achar-kitchen-koszul-mixed-arXiv}
Pramod~N. Achar and Sarah Kitchen.
\newblock Koszul duality and mixed {H}odge modules.
\newblock {\em Preprint}, 2011.
\newblock \href{http://arxiv.org/abs/1105.2181}{arXiv:1105.2181v1 [math.RT]}.

\bibitem[ASS06]{assem-simson-skowronski-1}
Ibrahim Assem, Daniel Simson, and Andrzej Skowro{\'n}ski.
\newblock {\em Elements of the representation theory of associative algebras.
  {V}ol. 1}, volume~65 of {\em London Mathematical Society Student Texts}.
\newblock Cambridge University Press, Cambridge, 2006.
\newblock Techniques of representation theory.

\bibitem[AT08]{achar-treumann-arXiv}
Pramod~N. Achar and David Treumann.
\newblock Purity and decomposition theorems for staggered sheaves.
\newblock {\em Preprint}, 2008.
\newblock \href{http://arxiv.org/abs/0808.3210}{arXiv:0808.3210v1 [math.AG]}.

\bibitem[BBD82]{BBD}
A.~A. Be{\u\i}linson, J.~Bernstein, and P.~Deligne.
\newblock Faisceaux pervers.
\newblock In {\em Analysis and topology on singular spaces, I (Luminy, 1981)},
  volume 100 of {\em Ast\'erisque}, pages 5--171. Soc. Math. France, Paris,
  1982.

\bibitem[Be{\u\i}87]{Beilinson}
A.~A. Be{\u\i}linson.
\newblock On the derived category of perverse sheaves.
\newblock In {\em $K$-theory, arithmetic and geometry (Moscow, 1984--1986)},
  volume 1289 of {\em Lecture Notes in Math.}, pages 27--41. Springer, Berlin,
  1987.

\bibitem[BN93]{neeman-homotopy-limits}
Marcel B{\"o}kstedt and Amnon Neeman.
\newblock Homotopy limits in triangulated categories.
\newblock {\em Compositio Math.}, 86(2):209--234, 1993.

\bibitem[Bon10]{bondarko-weight-str-vs-t-str}
M.~V. Bondarko.
\newblock Weight structures vs. {$t$}-structures; weight filtrations, spectral
  sequences, and complexes (for motives and in general).
\newblock {\em J. K-Theory}, 6(3):387--504, 2010.

\bibitem[BR07]{beligiannis-reiten-torsion-theories}
Apostolos Beligiannis and Idun Reiten.
\newblock Homological and homotopical aspects of torsion theories.
\newblock {\em Mem. Amer. Math. Soc.}, 188(883):viii+207, 2007.

\bibitem[BS01]{balmer-schlichting}
Paul Balmer and Marco Schlichting.
\newblock Idempotent completion of triangulated categories.
\newblock {\em J. Algebra}, 236(2):819--834, 2001.

\bibitem[Ill71]{illusie-cotan-i}
Luc Illusie.
\newblock {\em Complexe cotangent et d\'eformations. {I}}.
\newblock Lecture Notes in Mathematics, Vol. 239. Springer-Verlag, Berlin,
  1971.

\bibitem[IY08]{iyama-yoshino-mutation}
Osamu Iyama and Yuji Yoshino.
\newblock Mutation in triangulated categories and rigid {C}ohen-{M}acaulay
  modules.
\newblock {\em Invent. Math.}, 172(1):117--168, 2008.

\bibitem[KN10]{keller-nicolas-simple-arXiv}
Bernhard Keller and Pedro Nicolas.
\newblock Simple dg modules for positive dg algebras.
\newblock {\em Preprint}, 2010.
\newblock \href{http://arxiv.org/abs/1009.5904}{arXiv:1009.5904v2 [math.RT]}.

\bibitem[Kra08]{Krause-krull-schmidt-categories}
Henning Krause.
\newblock {K}rull-{R}emak-{S}chmidt categories and projective covers.
\newblock {\em Note}, pages 1--9, 2008.
\newblock
  http://www2.math.uni-paderborn.de/fileadmin/Mathematik/AG-Krause/publication%
s{$\underline{\hbox{\ }}$}krause/krs.pdf.

\bibitem[KS94]{KS}
Masaki Kashiwara and Pierre Schapira.
\newblock {\em Sheaves on manifolds}, volume 292 of {\em Grundlehren der
  Mathematischen Wissenschaften [Fundamental Principles of Mathematical
  Sciences]}.
\newblock Springer-Verlag, Berlin, 1994.
\newblock With a chapter in French by Christian Houzel, Corrected reprint of
  the 1990 original.

\bibitem[KS06]{KS-cat-sh}
Masaki Kashiwara and Pierre Schapira.
\newblock {\em Categories and sheaves}, volume 332 of {\em Grundlehren der
  Mathematischen Wissenschaften [Fundamental Principles of Mathematical
  Sciences]}.
\newblock Springer-Verlag, Berlin, 2006.

\bibitem[KV87]{keller-vossieck}
Bernhard Keller and Dieter Vossieck.
\newblock Sous les cat\'egories d\'eriv\'ees.
\newblock {\em C. R. Acad. Sci. Paris S\'er. I Math.}, 305(6):225--228, 1987.

\bibitem[LC07]{Karoubianness}
Jue Le and Xiao-Wu Chen.
\newblock Karoubianness of a triangulated category.
\newblock {\em J. Algebra}, 310(1):452--457, 2007.

\bibitem[Miy06]{miyachi-grothendieck}
Jun-ichi Miyachi.
\newblock Grothendieck groups of unbounded complexes of finitely generated
  modules.
\newblock {\em Arch. Math. (Basel)}, 86(4):317--320, 2006.

\bibitem[MSSSS10]{MSSS-AB-context-co-t-str-arXiv}
O.~Mendoza, E.~C. Saenz, V.~Santiago, and M.~J. Souto~Salorio.
\newblock Auslander-{B}uchweitz context and co-t-structures.
\newblock {\em Preprint}, 2010.
\newblock \href{http://arxiv.org/abs/1002.4604}{arXiv:1002.4604v2 [math.CT]}.

\bibitem[Nee01]{neeman-tricat}
Amnon Neeman.
\newblock {\em Triangulated categories}, volume 148 of {\em Annals of
  Mathematics Studies}.
\newblock Princeton University Press, Princeton, NJ, 2001.

\bibitem[Pau08]{pauk-co-t}
David Pauksztello.
\newblock Compact corigid objects in triangulated categories and
  co-{$t$}-structures.
\newblock {\em Cent. Eur. J. Math.}, 6(1):25--42, 2008.

\bibitem[Pau10]{pauk-note-co-t-arXiv}
David Pauksztello.
\newblock A note on compactly generated co-t-structures.
\newblock {\em Preprint}, 2010.
\newblock \href{http://arxiv.org/abs/1006.5347}{arXiv:1006.5347v2 [math.CT]}.

\bibitem[Sch06]{schlichting-neg-K-theory-dercat}
Marco Schlichting.
\newblock Negative {$K$}-theory of derived categories.
\newblock {\em Math. Z.}, 253(1):97--134, 2006.

\bibitem[Tho97]{thomason-class}
R.~W. Thomason.
\newblock The classification of triangulated subcategories.
\newblock {\em Compositio Math.}, 105(1):1--27, 1997.

\bibitem[Wil09]{wildeshaus-chow-arXiv}
J.~Wildeshaus.
\newblock Chow motives without projectivity.
\newblock {\em Preprint}, 2009.
\newblock \href{http://arxiv.org/abs/0806.3380}{arXiv:0806.3380v2 [math.AG]}.

\bibitem[WW09]{geordie-ben-HOMFLYPT}
Ben Webster and Geordie Williamson.
\newblock A geometric construction of colored {H}{O}{M}{F}{L}{Y}{P}{T}
  homology.
\newblock {\em Preprint}, 2009.
\newblock \href{http://arxiv.org/abs/0905.0486}{arXiv:0905.0486v3 [math.GT]}.

\end{thebibliography}

\def\cprime{$'$} \def\cprime{$'$} \def\cprime{$'$} \def\cprime{$'$}
  \def\Dbar{\leavevmode\lower.6ex\hbox to 0pt{\hskip-.23ex \accent"16\hss}D}
  \def\cprime{$'$} \def\cprime{$'$}

\end{document}